\documentclass{amsart}

\usepackage{lipsum}
\usepackage{amsfonts}
\usepackage{graphicx}
\usepackage{epstopdf}
\usepackage{algorithm}
\usepackage{algorithmic}
\usepackage{upgreek}
\usepackage{overpic}
\usepackage{enumerate}
\usepackage{enumitem}
\usepackage{bm}
\usepackage[hidelinks]{hyperref}

\usepackage{amsmath,amssymb}

\newcommand{\cN}{\mathcal{N}}
\newcommand{\cL}{\mathcal{L}}

\newcommand{\bE}{\mathbb{E}}

\newcommand{\bR}{\mathbb{R}}

\newcommand{\bN}{\mathbb{N}}

\newcommand{\vx}{\mathbf{x}}
\newcommand{\vy}{\mathbf{y}}

\newcommand{\vz}{\mathbf{z}}

\newcommand{\sfA}{\mathsf{A}}

\newcommand{\sfT}{\mathsf{T}}

\newcommand{\mcl}{\mathcal}

\newcommand{\mbf}{\mathbf}
\newcommand{\mbb}{\mathbb}

\newcommand{\bx}{\mbf x}

\newcommand{\Md}{M_\Omega}
\newcommand{\mU}{\mcl{U}}
\newcommand{\st}{{\rm\,s.t.}}
\newcommand{\R}{\mbb{R}}
\newcommand{\bz}{\mbf z}
\newcommand{\cc}{K}
\newcommand{\bphi}{{\bm{\phi}}}
\newcommand{\dd}{{\rm d}}

\DeclareMathOperator*{\argmin}{arg\,min}

\DeclareMathOperator*{\minimize}{{\rm minimize}}

\newenvironment{newremark}[1]{%
    \begin{remark}#1}{%
    \Endofdef\end{remark}%
}
\newcommand{\xqed}[1]{%
    \leavevmode\unskip\penalty9999 \hbox{}\nobreak\hfill
    \quad\hbox{\ensuremath{#1}}}
\newcommand{\Endofdef}{\xqed{\lozenge}}
\newtheorem{theorem}{Theorem}[section]
\newtheorem{lemma}[theorem]{Lemma}

\theoremstyle{definition}
\newtheorem{definition}[theorem]{Definition}

\theoremstyle{remark}
\newtheorem{remark}[theorem]{Remark}

\numberwithin{equation}{section}

\usepackage{enumitem}
\newtheorem{assumption}{Assumption}
\newtheorem{proposition}[theorem]{Proposition}

\usepackage{tikz}
\usetikzlibrary{shapes.arrows, patterns}
\definecolor{lightblue}{HTML}{a1b4c7}
\definecolor{orange}{HTML}{ea8810}
\definecolor{silver}{HTML}{b0aba8}
\definecolor{rust}{HTML}{b8420f}
\definecolor{seagreen}{HTML}{23553c}

\colorlet{lightsilver}{silver!30!white}
\colorlet{darkorange}{orange!85!black}
\colorlet{darksilver}{silver!85!black}
\colorlet{darklightblue}{lightblue!85!black}
\colorlet{darkrust}{rust!85!black}
\colorlet{darkseagreen}{seagreen!85!black}
\newcommand{\blue}[1]{#1}

\begin{document}

\title[Sparse Cholesky for Solving Nonlinear PDEs]{Sparse Cholesky Factorization for Solving Nonlinear PDEs via Gaussian Processes}


\author{Yifan Chen}
\address{Courant Institute, New York University, NY 10012, Corresponding author}
\curraddr{}
\email{yifan.chen@nyu.edu}
\thanks{}

\author{Houman Owhadi}
\address{Applied and Computational Mathematics, Caltech, Pasadena, CA 91106}
\curraddr{}
\email{owhadi@caltech.edu}
\thanks{}

\author{Florian Sch\"afer}
\address{School of Computational Science and Engineering
Georgia Institute of Technology, Atlanta, GA 30332}
\curraddr{}
\email{florian.schaefer@cc.gatech.edu}
\thanks{}


\date{}

\dedicatory{}

\begin{abstract}
\blue{In recent years, there has been widespread adoption of machine learning-based approaches to automate the solving of partial differential equations (PDEs). Among these approaches, Gaussian processes (GPs) and kernel methods have garnered considerable interest due to their flexibility, robust theoretical guarantees, and close ties to traditional methods. They can transform the solving of general nonlinear PDEs into solving quadratic optimization problems with nonlinear, PDE-induced constraints. However, the complexity bottleneck lies in computing with dense kernel matrices obtained from pointwise evaluations of the covariance kernel, and its \textit{partial derivatives}, a result of the {PDE constraint} and for which fast algorithms are scarce.

The primary goal of this paper is to provide a near-linear complexity algorithm for working with such kernel matrices. We present a sparse Cholesky factorization algorithm for these matrices based on the near-sparsity of the Cholesky factor under a novel ordering of pointwise and derivative measurements. The near-sparsity is rigorously justified by directly connecting the factor to GP regression and exponential decay of basis functions in numerical homogenization. We then employ the Vecchia approximation of GPs, which is optimal in the Kullback-Leibler divergence, to compute the approximate factor. This enables us to compute $\epsilon$-approximate inverse Cholesky factors of the kernel matrices with complexity $O(N\log^d(N/\epsilon))$ in space and $O(N\log^{2d}(N/\epsilon))$ in time. We integrate sparse Cholesky factorizations into optimization algorithms to obtain fast solvers of the nonlinear PDE. We numerically illustrate our algorithm's near-linear space/time complexity for a broad class of nonlinear PDEs such as the nonlinear elliptic, Burgers, and Monge-Amp\`ere equations. In summary, we provide a fast, scalable, and accurate method for solving general PDEs with GPs and kernel methods.
}

\end{abstract}
\maketitle
\tableofcontents


\section{Introduction} 

Machine learning and probabilistic inference \cite{murphy2012machine} have gained increasing popularity due to their capacity to automate the solution of computational problems.  Gaussian processes (GPs) \cite{williams2006gaussian} offer a promising approach, combining the theoretical rigor of traditional numerical algorithms with the flexible design of machine learning solvers \cite{owhadi2015bayesian,owhadi2017multigrid,raissi2018numerical,cockayne2019bayesian,chen2021solving}. 
Additionally, GPs exhibit deep connections to kernel methods \cite{scholkopf2002learning, schaback2006kernel}, neural networks \cite{neal1996priors,lee2017deep,jacot2018neural}, and meshless methods \cite{schaback2006kernel,zhang2000meshless}.

This paper investigates the computational efficiency of GPs and kernel methods in solving nonlinear PDEs, where dense kernel matrices are encountered, with entries derived from pointwise values and \textit{derivatives} of the covariance kernel function of the GP. The methodology developed herein could also be of practical interest in other contexts where derivative information of a GP or function is available, such as in Bayesian optimization \cite{wu2017bayesian} and PDE discovery \cite{long2022kernel}.
\subsection{The context} 
\blue{
Recent research proposes the use of GPs and kernel methods for solving nonlinear PDEs. The rationale behind the approach is to regard the PDE as data in machine learning, describing the relationship between pointwise evaluations and derivatives of a function at each collocation point. Then, by placing a GP prior on the unknown function and solving the maximum a posteriori (MAP) estimation given the PDE data at collocation points, a numerical solver for the PDEs can be obtained \cite{chen2021solving}. We review the methodology in Section \ref{sec: Solving nonlinear PDEs via GPs}. In short, the method transforms every nonlinear PDE into the following quadratic optimization problem with nonlinear, PDE-induced constraints:}
\begin{equation}
\label{eqn: the problem}
    \left\{
      \begin{aligned}
  &\min_{\vz \in \bR^N} && \vz^T K(\bphi, \bphi)^{-1} \vz \\
  &\mathrm{s.t.}  && { F(\vz)} = { \vy}\, ,
\end{aligned}
\right. 
\end{equation}
where $F$ and $\vy$ encode the PDE and source/boundary data. $K(\bphi, \bphi) \in \bR^{N\times N}$ is a positive definite kernel matrix whose entries are $k(\vx_i,\vx_j)$ or (a linear combination of) derivatives of the kernel function such as $\Delta_\vx k(\vx_i,\vx_j)$. Here $\vx_i, \vx_j \in \bR^d$ are some sampled collocation points in space. Entries like $k(\vx_i,\vx_j)$ arise from Diracs measurements while entries like $\Delta_\vx k(\vx_i,\vx_j)$ come from derivative measurements of the GP; for more details see Section \ref{sec: Solving nonlinear PDEs via GPs}. 

\blue{In \cite{chen2021solving, batlle2023error}, theoretical analyses and numerical experiments demonstrate the flexibility and provable convergence of the method. For typical PDE problems with smooth solutions, using Mat\'ern kernels and Gaussian kernels can achieve high accuracy with a moderate number (i.e. $N \sim 10^3 \ \text{to}\  10^4$) of collocation points; in such case the algorithm is more efficient than vanilla methods such as finite difference and is much simpler and more interpretable than neural network based approaches. However, computing with the dense matrix $K(\bphi, \bphi)$ na\"ively results in $O(N^3)$ space/time complexity, which becomes prohibitive as the number of collocation points further increase. In this paper, we aim to provide a general algorithm to make the GP approach scalable for solving nonlinear PDEs.}

\subsection{Contributions and organizations} \blue{This paper introduces an algorithm with a computational complexity of $O(N\log^d(N/\epsilon))$ in space and $O(N\log^{2d}(N/\epsilon))$ in time. This algorithm outputs a permutation matrix $P_{\rm perm}$ and a sparse upper triangular matrix $U$ with $O(N\log^d(N/\epsilon))$ nonzero entries. The resulting matrices satisfy the condition:
\begin{equation}
    \|K(\bphi, \bphi)^{-1} - P_{\rm perm}^TU{U}^TP_{\rm perm}\|_{\mathrm{Fro}} \leq \epsilon\, ,
\end{equation}
where $\|\cdot\|_{\mathrm{Fro}}$ is the Frobenius norm. The algorithm details are elaborated in Section \ref{sec: The sparse Cholesky factorization algorithm}. 

We rigorously analyze the error of the algorithm in Section \ref{sec: the theory}.  The analysis requires sufficient Dirac measurements within the domain and is established for a class of kernel functions that are Green functions of differential operators, such as Mat\'ern-like kernels. The theory relies on the interplay of linear algebra, Gaussian process conditioning, screening effects, and numerical homogenization, which demonstrate the \textit{exponential decay/near-sparsity} of the inverse Cholesky factor of the kernel matrix after permutation.


Utilizing these sparse factors, gradient-based optimization methods for \eqref{eqn: the problem} become scalable. We can also employ second-order methods, often more efficient, to solve \eqref{eqn: the problem} by linearizing the constraint and solving a sequential quadratic programming problem. This results in a linear system involving a reduced kernel matrix $K(\bphi^k,\bphi^k) := DF(\vz^k)K(\bphi,\bphi)(DF(\vz^k))^T$ at each iterate $\bz^k$; refer to Section \ref{sec: Gauss-Newton and preconditioned conjugate gradient}. For this reduced kernel matrix, where insufficient Dirac measurements are present in the domain, our theoretical guarantee for its sparse Cholesky factorization no longer holds. Nevertheless, we can apply the algorithm with a slightly different permutation and couple it with preconditioned conjugate gradient (pCG) methods to solve the linear system. Our experiments demonstrate that nearly constant steps of pCG suffice for convergence. For many nonlinear PDEs, we observe that the above sequential quadratic programming approach converges in $O(1)$ steps. Consequently, our algorithm leads to a near-linear space/time complexity solver for general nonlinear PDEs, assuming the sequential quadratic programming iterations converge. The assumption of convergence depends on the selection of kernels and the property of the PDE, and we demonstrate it numerically in solving nonlinear elliptic, Burgers, and Monge-Amp\`ere equations; see Section \ref{sec: Numerical experiments}. We make concluding remarks in Section \ref{sec: conclusions}.}



\subsection{Related work} We summarize relevant literature below.
\subsubsection{Machine learning PDEs} 
Machine learning methods, such as those based on neural networks (NNs) and GPs, have shown remarkable promise in automating scientific computing, for instance in solving PDEs. Recent developments in this field include operator learning using prepared solution data \cite{li2020fourier,bhattacharya2021model,nelsen2021random,lu2021learning} and learning a single solution without any solution data \cite{han2018solving,raissi2019physics,chen2021solving,karniadakis2021physics}. This paper focuses on the latter. NNs provide an expressive function representation. Empirical success has been widely reported in the literature. However, the training of NNs often requires significant tuning and takes much longer than traditional solvers \cite{grossmann2023can}. Considerable research efforts have been devoted to stabilizing and accelerating the training process \cite{krishnapriyan2021characterizing,wang2021understanding,wang2022and,daw2022rethinking,zeng2023competitive}. 

GP and kernel methods are based on a more interpretable and theoretically grounded function representation rooted in the Reproducing Kernel Hilbert Space (RKHS) theory \cite{wendland2004scattered,berlinet2011reproducing,owhadi2019operator}; with hierarchical kernel learning \cite{wilson2016deep,owhadi2019kernel,chen2021consistency,darcy2023one}, these representations can be made expressive as well. Nevertheless, working with dense kernel matrices is common, which often limits scalability. In the case of PDE problems, these matrices may also involve partial derivatives of the kernels \cite{chen2021solving}, and fast algorithms for such matrices are less developed compared to the derivative-free counterparts.

\subsubsection{Fast solvers for kernel matrices} 
Approximating dense kernel matrices (denoted by $\Theta$) is a classical problem in scientific computing and machine learning. Most existing methods focus on the case where $\Theta$ only involves the pointwise values of the kernel function. These algorithms typically rely on low-rank or sparse approximations, as well as their combination and multiscale variants. Low-rank techniques include Nystr\"om's approximations \cite{williams2000using,musco2017recursive,chen2022randomly}, rank-revealing Cholesky factorizations \cite{gu2004strong}, inducing points via a probabilistic view \cite{quinonero2005unifying}, and random features \cite{rahimi2007random}. Sparsity-based methods include covariance tapering \cite{furrer2006covariance}, local experts (see a review in \cite{liu2020gaussian}), and approaches based on precision matrices and stochastic differential equations \cite{lindgren2011explicit, roininen2011correlation,sanz2022spde,sanz2022finite}. Combining low-rank and sparse techniques can lead to efficient structured approximation \cite{wilson2015kernel} and can better capture short and long-range interactions \cite{sang2012full}. Multiscale and hierarchical ideas have also been applied to seek for a full-scale approximation of $\Theta$ with a \textit{near-linear} complexity. \blue{They include $\mathcal{H}$ matrix \cite{hackbusch1999sparse,hackbusch2000sparse,hackbusch2002data} and variants \cite{li2012new,ambikasaran2013mathcal,ambikasaran2015fast,l2016hierarchical,minden2017recursive,minden2017fast,litvinenko2019likelihood,geoga2020scalable} that rely on the low-rank structure of the off-diagonal block matrices at different scales; wavelets-based methods \cite{beylkin1991fast,gines1998lu} that use the sparsity of $\Theta$ in the wavelet basis; multiresolution predictive processes \cite{katzfuss2017multi}; and Vecchia approximations \cite{vecchia1988estimation,katzfuss2020vecchia} and sparse Cholesky factorizations \cite{schafer2021compression, schafer2021sparse} that rely on the approximately sparse correlation conditioned on carefully ordered points. }

For $\Theta$ that contains derivatives of the kernel function, several work \cite{eriksson2018scaling,padidar2021scaling,de2021high} has utilized structured approximation to scale up the computation; no rigorous accuracy guarantee is proved. The inducing points approach \cite{yang2018sparse,meng2022sparse} has also been explored; however since this method only employs a low-rank approximation, the accuracy and efficiency can be limited.
\subsubsection{Screening effects in spatial statistics} Notably, the sparse Cholesky factorization algorithm in \cite{schafer2021sparse}, formally equivalent to Vecchia's approximation \cite{vecchia1988estimation,katzfuss2020vecchia}, achieves a state-of-the-art complexity $O(N\log^d(N/\epsilon))$ in space and $O(N\log^{2d}(N/\epsilon))$ in time for a wide range of kernel functions, with a rigorous theoretical guarantee. This algorithm is designed for kernel matrices with derivative-free entries and is connected to the screening effect in spatial statistics \cite{stein2002screening,stein20112010}. The screening effect implies that approximate conditional independence of a spatial random field is likely to occur, under suitable ordering of points. The line of work \cite{owhadi2017multigrid,owhadi2019operator,schafer2021compression} provides quantitative exponential decay results for the conditional covariance in the setting of a coarse-to-fine ordering of data points, laying down the theoretical groundwork for \cite{schafer2021sparse}.

A fundamental question is how the screening effect behaves when derivative information of the spatial field is incorporated, and how to utilize it to extend sparse Cholesky factorization methods to kernel matrices that contain derivatives of the kernel. The screening effect studied within this new context can be useful for numerous applications where derivative-type measurements are available.

\section{Solving nonlinear PDEs via GPs} 
\label{sec: Solving nonlinear PDEs via GPs}
In this section, we review the GP framework in \cite{chen2021solving} for solving nonlinear PDEs. We will use a prototypical nonlinear elliptic equation as our running example to demonstrate the main ideas, followed by more complete recipes for general nonlinear PDEs.

Consider the following nonlinear elliptic PDE:
\begin{equation}
    \left\{\begin{aligned}
    -\Delta u + \tau(u) & = f \quad \text{in }\Omega\, ,\\
    u &= g \quad \text{on }\partial\Omega\, ,
    \end{aligned}
    \right.
\end{equation}
where $\tau$ is a nonlinear scalar function and $\Omega$ is a bounded open domain in $\bR^d$ with a Lipschitz boundary. We assume the equation has a strong solution in the classical sense. 
\subsection{The GP framework}
\label{sec: The GP framework}The first step is to sample $M_{\Omega}$ collocation points in the interior and $M_{\partial \Omega}$ on the boundary such that
\[\vx_{\Omega} = \{\vx_1,...,\vx_{M_\Omega}\} \subset \Omega\quad \text{and}\quad \vx_{\partial\Omega} = \{\vx_{M_{\Omega}+1},...,\vx_{M}\} \subset \partial \Omega\, ,\]
where $M = M_{\Omega}+M_{\partial \Omega}$. Then, by assigning a GP prior to the unknown function $u$ with mean $0$ and covariance function $K: \overline{\Omega} \times \overline{\Omega} \to \bR$, the method aims to compute the maximum a posterior (MAP) estimator of the GP given the sampled PDE data, which leads to the following optimization problem
 \begin{equation}\label{running-example-optimization-problem}
    \left\{
      \begin{aligned}
  &\minimize_{u \in \mU}~\|u\| && \\
  &\st \quad -\Delta u(\vx_m) + \tau(u(\vx_m)) = f(\vx_m), \quad &&\text{for } m=1,\ldots,\Md\,,\\
 & \hspace{6ex} u(\vx_m)=g(\vx_m),     \quad &&\text{for }m=\Md+1,\ldots,M\,.
\end{aligned}
\right.
\end{equation}
Here, $\|\cdot\|$ is the Reproducing Kernel Hilbert Space (RKHS) norm corresponding to the kernel/covariance function $K$. 

Regarding consistency, once $K$ is sufficiently regular, the above solution will converge to the exact solution of the PDE when $M_{\Omega},M_{\partial\Omega}\to \infty$; see Theorem 1.2 in \cite{chen2021solving} and convergence rates in \cite{batlle2023error}. The methodology can be seen as a nonlinear generalization of many radial basis function based meshless methods \cite{schaback2006kernel} and probabilistic numerics \cite{owhadi2015bayesian,cockayne2019bayesian}.

\subsection{The finite dimensional problem}
\label{sec: The finite dimensional problem}
The next step is to transform \eqref{running-example-optimization-problem} into a finite-dimensional problem for computation. We first introduce some notations:
\begin{itemize}[leftmargin=*]
    \item Notations for \textit{measurements}: We denote the measurement functions by
\[\phi^{(1)}_m=\updelta_{\vx_m}, 1\leq m \leq M
  \quad \text{and}\quad \phi^{(2)}_m = \updelta_{\vx_m} \circ \Delta, 1\leq m \leq M_{\Omega}\, ,\]
  where $\updelta_{\vx}$ is the
  Dirac delta function centered at $\vx$. They are in $\mU^*$, the dual space of $\mU$, for sufficiently regular kernel functions. 
  
  Further, we use the shorthand notation
  $\bphi^{(1)}$ and $\bphi^{(2)}$ for the $M$ and $M_\Omega$-dimensional vectors with entries $\phi^{(1)}_m$ and $\phi^{(2)}_m$
  respectively,
  and  $\bphi$  for the $N$-dimensional vector  obtained by concatenating $\bphi^{(1)}$ and $\bphi^{(2)}$, where $N = M+M_{\Omega}$. 
  \item Notations for \textit{primal dual pairing}: We use $[\cdot, \cdot]$ to denote the primal dual pairing, such that for $u \in \mU, \phi^{(1)}_m = \updelta_{\vx_m} \in \mU^*$, it holds that $[u, \phi^{(1)}_m] = u(\vx_m)$. Similarly $[u, \phi^{(2)}_m] = \Delta u(\vx_m)$ for $\phi^{(2)}_m = \updelta_{\vx_m} \circ \Delta \in \mU^*$. For simplicity of presentation, we oftentimes abuse the notation to write the primal-dual pairing in the $L^2$ integral form: $[u,\phi] = \int u(\vx)\phi(\vx) \, \dd \vx$.
  \item Notations for \textit{kernel matrices}:  We write $K(\bphi,\bphi)$ as the $N\times N$-matrix with entries
$\int K(\vx,\vx')\phi_m(\vx)\phi_j(\vx')\,\dd \vx\, \dd \vx'$ where $\phi_m$ denotes the
entries of $\bphi$. Here, the integral notation shall be interpreted as the primal-dual pairing as above.

Similarly, $K(\vx, \bphi)$ is the $N$ dimensional vector with entries $\int K(\vx,\vx')\phi_j(\vx')\, \dd \vx'$.
Moreover, we adopt the convention that if the variable inside a function is a set, it means that this function is applied to every element in this set; the output will be a vector or a matrix. As an example, $K(\vx_{\Omega},\vx_{\Omega}) \in \bR^{M_{\Omega}\times M_{\Omega}}$.
\end{itemize}
Then, based on a generalization of the representer theorem \cite{chen2021solving}, the minimizer of \eqref{running-example-optimization-problem} attains the form \[u^\dagger(\vx) = K(\vx, \bphi) K(\bphi, \bphi)^{-1} \vz^\dagger\, ,\] where $\vz^\dagger$ is the solution to the following finite dimensional quadratic optimization problem with nonlinear constraints
\begin{equation}
 \label{eqn: finite dim optimization for elliptic eqn}
   \left\{
   \begin{aligned}
  & \minimize_{\bz \in \R^{M + M_\Omega}}  \quad   \bz^T  K(\bphi,\bphi)^{-1}\bz \\
  &\st  \quad  -z^{(2)}_m + \tau(z^{(1)}_m) = f(\vx_m),  &&\text{for } m=1,\ldots,\Md\,, \\
  & \hspace{6ex} z^{(1)}_m=g(\vx_m), &&\text{for } m=\Md+1,\ldots, M\,.
\end{aligned}
\right.
  \end{equation}
Here, $\bz^{(1)} \in \R^{M}$, $\bz^{(2)} \in \R^{M_{\Omega}}$ and $\bz$ is the concatenation of them. For this specific example, we can write down $K(\vx, \bphi)$ and $K(\bphi,\bphi)$ explicitly:
 \begin{equation}
 \label{eqn: example of kernel matrices and vectors}
\begin{aligned}
& K(\vx, \bphi) = \left(K(\vx,\vx_{\Omega}), K(\vx,\vx_{\partial\Omega}), \Delta_{\vy} K(\vx,\vx_{\Omega})\right) \in \bR^{1\times N}\, ,\\
    &K(\bphi, \bphi) =
    \begin{pmatrix}
    K(\vx_{\Omega},\vx_{\Omega}) & K(\vx_{\Omega},\vx_{\partial\Omega}) & \Delta_{\vy} K(\vx_{\Omega},\vx_{\Omega}) \\
    K(\vx_{\partial\Omega},\vx_{\Omega}) & K(\vx_{\partial\Omega},\vx_{\partial\Omega}) & \Delta_{\vy} K(\vx_{\partial\Omega},\vx_{\Omega})\\
    \Delta_{\vx} K(\vx_{\Omega},\vx_{\Omega}) & \Delta_{\vx} K(\vx_{\Omega},\vx_{\partial\Omega}) & \Delta_{\vx}\Delta_{\vy}K(\vx_{\Omega},\vx_{\Omega}) 
    \end{pmatrix}\in \bR^{N\times N}\, .
\end{aligned}
\end{equation}
Here, $\Delta_{\vx},\Delta_{\vy}$ are the Laplacian operator for the first and second arguments of $k$, respectively. Clearly, evaluating the loss function and its gradient requires us to deal with the dense kernel matrix $K(\bphi, \bphi)$ with entries comprising \textit{derivatives} of $k$.

\subsection{The general case}
For general PDEs, the methodology leads to the optimization problem
\begin{equation*}
    \left\{
      \begin{aligned}
        & \min_{u \in \mU}~  && \| u\| \\
        & \mathrm{s.t.}  && \text{PDE constraints at $\{\vx_1, \dots, \vx_M\} \in \overline{\Omega}$}\, , 
      \end{aligned}
      \right.
  \end{equation*}
and the equivalent finite dimensional problem
\begin{equation}
\label{eqn: general opt for general nonlinear PDEs}
    \left\{
      \begin{aligned}
  &\min_{\vz \in \bR^N} && \vz^T K(\bphi, \bphi)^{-1} \vz \\
  &\mathrm{s.t.}  && { F(\vz)} = { \vy}\, ,
\end{aligned}
\right. 
\end{equation}
where $\bphi$ is the concatenation of Diracs measurements and derivative measurements of $u$; they are induced by the PDE at the sampled points. The function $F$ encodes the PDE, and the vector $\vy$ encodes the right hand side and boundary data. Again, it is clear that the computational bottleneck lies in the part $K(\bphi,\bphi)^{-1}$.

\begin{newremark}
\label{remark: def of derivative measurements}
    Here, we use ``derivative measurement'' to mean a functional in $\mU^*$ whose action on a function in $\mU$ leads to a linear combination of its derivatives. Mathematically, suppose the highest order of derivatives is $J$, then the corresponding derivative measurement at point $\vx_m$ can be written as $\phi = \sum_{|\gamma|\leq J} {a_\gamma} \updelta_{\vx_m} \circ D^{\gamma}$ with the multi-index $\gamma = (\gamma_1,...,\gamma_d) \in \bN^d$ and $|\gamma|:=\sum_{k=1}^d \gamma_k \leq J$. Here $D^{\gamma}:=D^{\gamma_1}_{\vx^{(1)}}\cdots D^{\gamma_d}_{\vx^{(d)}}$ is a $|\gamma|$-th order differential operator, and we use the notation $\vx = (\vx^{(1)},..., \vx^{(d)})$. We require linear independence between these measurements to ensure $K(\bphi, \bphi)$ is invertible.
\end{newremark}
\section{The sparse Cholesky factorization algorithm}
\label{sec: The sparse Cholesky factorization algorithm}
In this section, we present a sparse Cholesky factorization algorithm for $K(\bphi,\bphi)^{-1}$.
Theoretical results will be presented in Section \ref{sec: the theory} based on the interplay between linear algebra, Gaussian process conditioning, screening effects in spatial statistics, and numerical homogenization.

In Subsection \ref{sec: Previous results for derivative-free measurements}, we summarize the state-of-the-art sparse Cholesky factorization algorithm for kernel matrices with derivative-free measurements. In Subsection \ref{sec: The case of derivative-type measurements}, we discuss an extension of the idea to kernel matrices with derivative-type measurements, which are the main focus of this paper. The algorithm presented in Subsection \ref{sec: The case of derivative-type measurements} leads to near-linear complexity evaluation of the loss function and its gradient in the GP method for solving PDEs. First-order methods thus become scalable. We then extend the algorithm to second-order optimization methods (e.g., the Gauss-Newton method) in Section \ref{sec: Gauss-Newton and preconditioned conjugate gradient}.

\subsection{The case of derivative-free measurements}
\label{sec: Previous results for derivative-free measurements}
We start the discussion with the case where $\bphi$ contains Diracs-type measurements only.

Consider a set of points $\{\vx_i\}_{i\in I} \subset \Omega$, where $I = \{1,2,...,M\}$ as in Subsection \ref{sec: The GP framework}. We assume the points are \textit{scattered}; to quantify this, we have the following definition of homogeneity:
\begin{definition}
    The homogeneity parameter of the points $\{\vx_i\}_{i\in I} \subset \Omega$ conditioned on a set $\sfA$ is defined as
\[\delta(\{\vx_i\}_{i\in I}; \sfA) = \frac{\min_{\vx_i\neq\vx_j \in I} \mathrm{dist}(\vx_i,\{\vx_j\}\cup \sfA)}{\max_{\vx\in\Omega} \mathrm{dist}(\vx, \{\vx_i\}_{i\in I} \cup \sfA)} \, .  \]
When $\sfA = \emptyset$, we also write $\delta(\{\vx_i\}_{i\in I}) := \delta(\{\vx_i\}_{i\in I}; \emptyset)$.
\end{definition}
Throughout this paper, we assume $\delta(\{\vx_i\}_{i\in I}) > 0$. One can understand that a larger $\delta(\{\vx_i\}_{i\in I})$ makes the distribution of points more homogeneous in space. It can also be useful to consider $\sfA = \partial \Omega$ if one wants the points not too close to the boundary.

Let $\bphi$ be the collection of $\updelta_{\vx_i}, 1\leq i \leq M$; all of them are Diracs-type measurements and thus derivative-free. In \cite{schafer2021sparse}, a sparse Choleksy factorization algorithm was proposed to factorize $K(\bphi,\bphi)^{-1}$. We summarize this algorithm (with a slight modification\footnote{The method in \cite{schafer2021sparse} was presented to get the lower triangular Cholesky factors. Our paper presents the method for solving the upper triangular Cholesky factors since it gives a more concise description. As a consequence of this difference, in the reordering step, we are led to a reversed ordering compared to that in \cite{schafer2021sparse}.}) in the following three steps: reordering, sparsity pattern, and Kullback-Leibler (KL) minimization.
\subsubsection{Reordering} The first step is to reorder these points from \textit{coarse to fine} scales. It can be achieved by the maximum-minimum distance ordering (maximin ordering) \cite{guinness2018permutation}. We define a generalization to conditioned maximin ordering as follows:
\begin{definition}[Conditioned Maximin Ordering]
\label{def: maximin ordering}
   The maximin ordering conditioned on a set $\sfA$ for points $\{\vx_i, i\in I\}$ is obtained by successively selecting the point $\vx_i$ that is furthest away from $\sfA$ and the already picked points. If $\sfA$ is an empty set, then we select an arbitrary index $i \in I$ as the first to start. Otherwise, we choose the first index as
\[i_1 = {\arg \max}_{i\in I} \operatorname{dist}(\vx_i, \sfA)\, . \]
For the first $q$ indices already chosen, we choose
\[i_{q+1} = {\arg \max}_{i\in I \backslash \{i_1,...,i_q\} } \operatorname{dist}(\vx_i, \{\vx_{i_1},...,\vx_{i_q}\} \cup \sfA)\, . \]
\end{definition}
Usually we set $\sfA = \partial \Omega$ or $\emptyset$.
We introduce the operator $P: I \to I$ to map the order of the measurements to the index of the corresponding points, i.e., $P(q) = i_q$. 
One can define the \textit{lengthscale} of each ordered point as 
\begin{equation}
    l_i = \mathrm{dist}(\vx_{P(i)}, \{\vx_{P(1)},...,\vx_{P(i-1)}\} \cup \sfA)\, .
\end{equation}
Let $\Theta = K(\tilde{\bphi},\tilde{\bphi}) \in \bR^{N\times N}$ be the kernel matrix after reordering the measurements in $\bphi$ to $\tilde{\bphi} = (\phi_{P(1)},...,\phi_{P(M)})$; we have $N=M$ in this setting. An important observation is that the Cholesky factors of $\Theta$ and $\Theta^{-1}$ could exhibit \textit{near-sparsity} under the maximin ordering. Indeed, as an example, suppose $\Theta^{-1} = U^\star {U^\star}^T$ where $U^\star$ is the upper Cholesky factor. Then in Figure \ref{fig:screening effcts, diracs}, we show the magnitude of $U^\star_{ij}, i\leq j$ for a Mat\'ern kernel, where $j=1000, 2000$; the total number of points is $M=51^2$. It is clear from the figure that the entries decay very fast when the points move far away from the current $j$th ordered point.
\begin{figure}[ht]
    \centering
    \includegraphics[width=6cm]{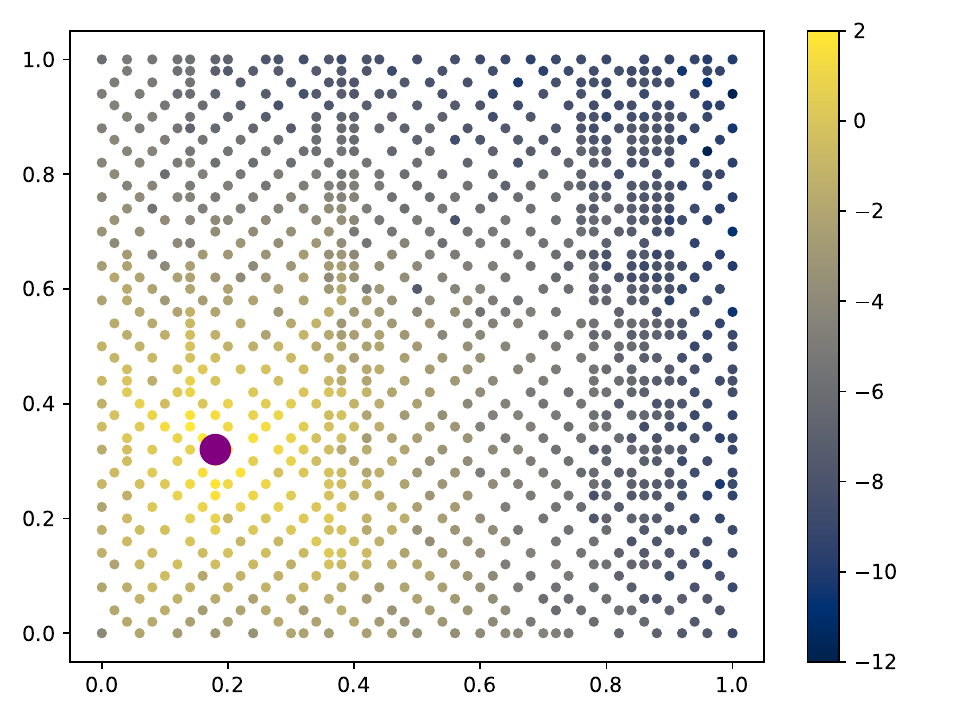}
    \includegraphics[width=6cm]{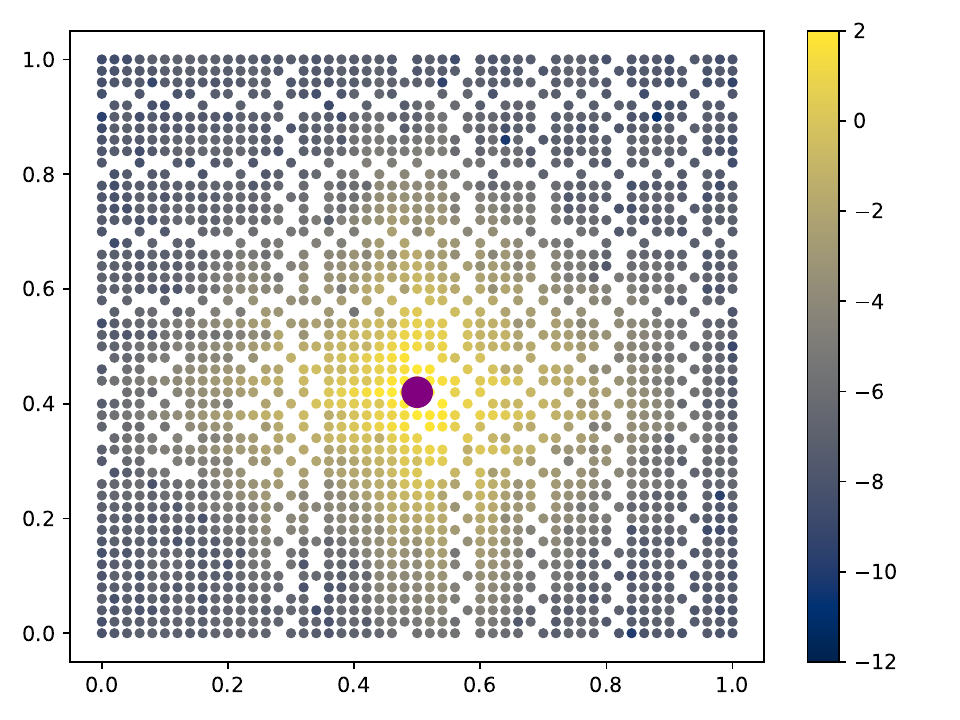}
    \caption{Demonstration of screening effects in the context of Diracs measurements using the Mat\'ern kernel with $\nu=5/2$ and lengthscale $0.3$. The data points are equidistributed in $[0,1]^2$ with grid size $h=0.02$. In the left figure, we display the $1000$th point (the big point) in the maximin ordering with $\sfA = \emptyset$, where all points ordered before it (i.e., $i<1000$) are colored with intensity according to the corresponding $|U_{ij}^\star|$. The right figure is generated in the same manner but for the $2000$th point in the ordering.}
    \label{fig:screening effcts, diracs}
\end{figure}

\begin{newremark}
\label{remark: why coarse-to-fine ordering}
One may wonder why such a coarse-to-fine reordering leads to sparse Cholesky factors. In fact, we can interpret entries of $U^\star$ as the conditional covariance of some GP. More precisely, consider the GP $\xi \sim \mathcal{GP}(0,K)$. Then, by definition, the Gaussian random variables $Y_i := [\xi, \tilde{\phi}_i] \sim \cN(0, K(\tilde{\phi}_i,\tilde{\phi}_i))$. We have the following relation:
\begin{equation}
\label{eqn: U as conditional GP}
\frac{U^\star_{ij}}{U^\star_{jj}} = (-1)^{i\neq j} \frac{\mathrm{Cov}[Y_i,Y_j|Y_{1:j-1\backslash \{i\}}]}{\mathrm{Var}[Y_i|Y_{1:j-1\backslash \{i\}}]}, \quad i \leq j\, . 
\end{equation}
Here we used the MATLAB notation such that $Y_{1:j-1\backslash \{i\}}$ corresponds to $\{Y_q: 1\leq q \leq j-1, q \neq i\}$. Proof of this formula can be found in Appendix \ref{appendix: Connections between Cholesky factors, conditional covariance, conditional expectation, and Gamblets}.


Formula \eqref{eqn: U as conditional GP} links the values of $U^\star$ to the conditional covariance of a GP. In spatial statistics, it is well-known from empirical evidence that conditioning a GP on coarse-scale measurements results in very small covariance values between finer-scale measurements. This phenomenon, known as \textit{screening effects}, has been discussed in works such as \cite{stein2002screening,stein20112010}. The implication is that conditioning on coarse scales \textit{screens out} fine-scale interactions.

As a result, one would expect the corresponding Cholesky factor to become sparse upon reordering. Indeed, the off-diagonal entries exhibit exponential decay. A rigorous proof of the quantitative decay can be found in \cite{schafer2021compression}, where the measurements consist of Diracs functionals only, and the kernel function is the Green function of some differential operator subject to Dirichlet boundary conditions.
The proof of Theorem 6.1 in \cite{schafer2021compression} effectively implies that 
\begin{equation}
\label{eqn: Diracs, decay of Uij}
    |U^\star_{ij}| \leq C_1 l_M^{C_2}\exp\left(-\frac{\operatorname{dist}(\vx_{P(i)},\vx_{P(j)})}{C_1 l_j}\right)
\end{equation} for some generic constants $C_1, C_2$ depending on the domain, kernel function, and homogeneity parameter of the points. We will prove such decay also holds when derivative-type measurements are included, in Section \ref{sec: the theory} under a novel ordering. It is worth mentioning that our analysis also provides a much simpler proof for \eqref{eqn: Diracs, decay of Uij}.
\end{newremark}
\subsubsection{Sparsity pattern} With the ordering determined, our next step is to identify the sparsity pattern of the Cholesky factor under the ordering. 

For a tuning parameter $\rho \in \mathbb{R}^{+}$, we select the upper-triangular sparsity set $S_{P,l,\rho} \subset I \times I$ as
\begin{equation}
    \label{eqn: def sparsity pattern}
    S_{P,l,\rho} = \{(i,j) \subset I \times I: i \leq j,\,  \operatorname{dist}(\vx_{P(i)},\vx_{P(j)}) \leq \rho l_j \}\, .
\end{equation}
The choice of the sparsity pattern is motivated by the quantitative exponential decay result mentioned in Remark \ref{remark: why coarse-to-fine ordering}. Here in the subscript, $P$ stands for the ordering, $l$ is the lengthscale parameter associated with the ordering, and $\rho$ is a hyperparameter that controls the size of the sparsity pattern. We sometimes drop the subscript to simplify the notation when there is no confusion. We note that the cardinality of the set $S_{P,l,\rho}$ is bounded by $O(N\rho^d)$, through a ball-packing argument (see Appendix \ref{appendix: Ball-packing arguments}).

\begin{newremark}
\label{rmk: complexity maximin ordering sparsity patterns}
    The maximin ordering and the sparsity pattern can be constructed with computational complexity $O(N\log^2(N)\rho^d)$ in time and $O(N\rho^d)$ in space; see Algorithm 4.1 and Theorem 4.1 in \cite{schafer2021sparse}.
\end{newremark}

\subsubsection{KL minimization}
\label{sec: KL minimization}
With the ordering and sparsity pattern identified, the last step is to use KL minimization to compute the best approximate sparse Cholesky factors given the pattern. 

Define the set of sparse upper-triangular matrices with sparsity pattern $S_{P,l,\rho}$ as 
\begin{equation}
\label{eqn: sparse mtx set}
    \mathcal{S}^{\rm mtx}_{P,l,\rho}:=\{A \in \bR^{N\times N}: A_{ij}\neq 0 \Rightarrow (i,j) \in S_{P,l,\rho}\}\, .
\end{equation}
For each column $j$, denote $s_j = \{i: (i,j) \in S_{P,l,\rho}\}$. The cardinality of the set $s_j$ is denoted by $\# s_j$.

The KL minimization step seeks to find
\begin{equation}
\label{eqn: KL min}
    U = {\arg\min}_{\hat{U} \in \mathcal{S}^{\rm mtx}_{P,l,\rho}} \operatorname{KL}\left(\cN(0,\Theta)\parallel\cN(0,(\hat{U}\hat{U}^T)^{-1})\right)\, .
\end{equation}
It turns out that the above problem has an explicit solution
\begin{equation}
    \label{eqn: explicit formula for KL min}
    U_{s_j,j} = \frac{\Theta_{s_j,s_j}^{-1}\textbf{e}_{\# s_j}}{\sqrt{\textbf{e}_{\# s_j}^T\Theta_{s_j,s_j}^{-1}\textbf{e}_{\# s_j}}}\, ,
\end{equation}
where $\textbf{e}_{\# s_j}$ is a standard basis vector in $\mathbb{R}^{\#s_j}$ with the last entry being $1$ and other entries equal $0$. Here, $\Theta_{s_j,s_j}^{-1} := (\Theta_{s_j,s_j})^{-1}$ \blue{where $\Theta_{s_j,s_j} \in \mathbb{R}^{(\# s_j) \times (\# s_j)}$ is a submatrix of $\Theta$ with index set $s_j$}. The proof of this explicit formula follows a similar approach to that of Theorem 2.1 in \cite{schafer2021sparse}, with the only difference being the use of upper Cholesky factors. A detailed proof is provided in Appendix \ref{appendix: Explicit formula for the KL minimization}. It is worth noting that the optimal solution is equivalent to the Vecchia approximation used in spatial statistics; see discussions in \cite{schafer2021sparse}.

With the KL minimization, we can find the best approximation measured in the KL divergence sense, given the sparsity pattern. The computation is embarrassingly parallel, noting that the formula \eqref{eqn: explicit formula for KL min} are independent for different columns. 

\begin{newremark}
    
    For the algorithm described above, the computational complexity is upper-bounded by $O(\sum_{1\leq i \leq N} (\# s_j)^3)$ in time and $O(\#S + \max_{1\leq i \leq N} (\# s_j)^2)$ in space when using dense Cholesky factorization to invert $\Theta_{s_j,s_j}$.

When using the sparsity pattern $S_{P,l,\rho}$, we can obtain $\#s_j = O(\rho^d)$ and $\#S = O(N\rho^d)$ via a ball-packing argument (see Appendix \ref{appendix: Ball-packing arguments}). This yields a complexity of $O(N\rho^{3d})$ in time and $O(N\rho^d)$ in space. 
\end{newremark}

\begin{newremark}
\label{rmk: derivative-free meas, supernodes, complexity}

The concept of \textit{supernodes} \cite{schafer2021sparse}, which relies on an extra parameter $\lambda$, can be utilized to group the sparsity pattern of nearby measurements and create an \textit{aggregate sparsity pattern $S_{P,l,\rho, \lambda}$}. This technique reduces computation redundancy and improves the arithmetic complexity of the KL minimization to $O(N\rho^{2d})$ in time (see Appendix \ref{appendix: Supernodes and aggregate sparsity pattern}). In this paper, we consistently employ this approach.
\end{newremark}
\begin{newremark}
    In \cite{schafer2021sparse}, it was shown in Theorem 3.4 that $\rho = O(\log(N/\epsilon))$ suffices to get an $\epsilon$-approximate factor for a large class of kernel matrices, so the complexity of the KL minimization is $O(N\log^{2d}(N/\epsilon))$ in time and $O(N\log^{d}(N/\epsilon))$ in space. Note that the ordering and aggregate sparsity pattern can be constructed in time complexity $O(N\log^2(N)\rho^d)$ and space complexity $O(N\rho^d)$; the complexity of this construction step is usually of a lower order compared to that of the KL minimization. Moreover, this step can be pre-computed. 
\end{newremark}
\subsection{The case of derivative measurements} The last subsection discusses the sparse Cholesky factorization algorithm for kernel matrices generated by derivative-free measurements. When using GPs to solve PDEs and inverse problems, $\bphi$ can contain derivative measurements, which are the main focus of this paper. This subsection aims to deal with such scenarios.
\label{sec: The case of derivative-type measurements}
\subsubsection{The nonlinear elliptic PDE case}


To begin with, we will consider the example in Subsection \ref{sec: The finite dimensional problem}, where we have Diracs measurements $\phi^{(1)}_m=\updelta_{\vx_m}$ for $1\leq m \leq M$, and Laplacian-type measurements $\phi^{(2)}_m = \updelta_{\vx_m} \circ \Delta$ for $1\leq m \leq M_{\Omega}$. Our objective is to extend the algorithm discussed in the previous subsection to include these derivative measurements.

An important question we must address is the ordering of these measurements. Specifically, should we consider the Diracs measurements before or after the derivative-type measurements?  To explore this question, we conduct the following experiment. First, we order all derivative-type measurements, $\phi^{(2)}_m$ for $1\leq m \leq M_{\Omega}$, in an arbitrary manner. We then follow this ordering with any Diracs measurement, labeled $M_{\Omega}+1$ in the order. For this measurement, we plot the magnitude of the corresponding Cholesky factor of $\Theta^{-1}$, i.e., $|U^\star_{ij}|$ for $i \leq j$ and $j = M_{\Omega}+1$, similar to the approach taken in Figure \ref{fig:screening effcts, diracs}. The results are shown in the left part of Figure \ref{fig:screening effcts, derivatives}.


Unfortunately, we do not observe an evident decay in the left of Figure \ref{fig:screening effcts, derivatives}. This may be due to the fact that, even when conditioned on the Laplacian-type measurements, the Diracs measurements can still exhibit long-range interactions with other measurements. This is because there are degrees of freedom of harmonic functions that are not captured by Laplacian-type measurements, and thus, the correlations may not be effectively screened out.

\begin{figure}[ht]
    \centering
    \includegraphics[width=6cm]{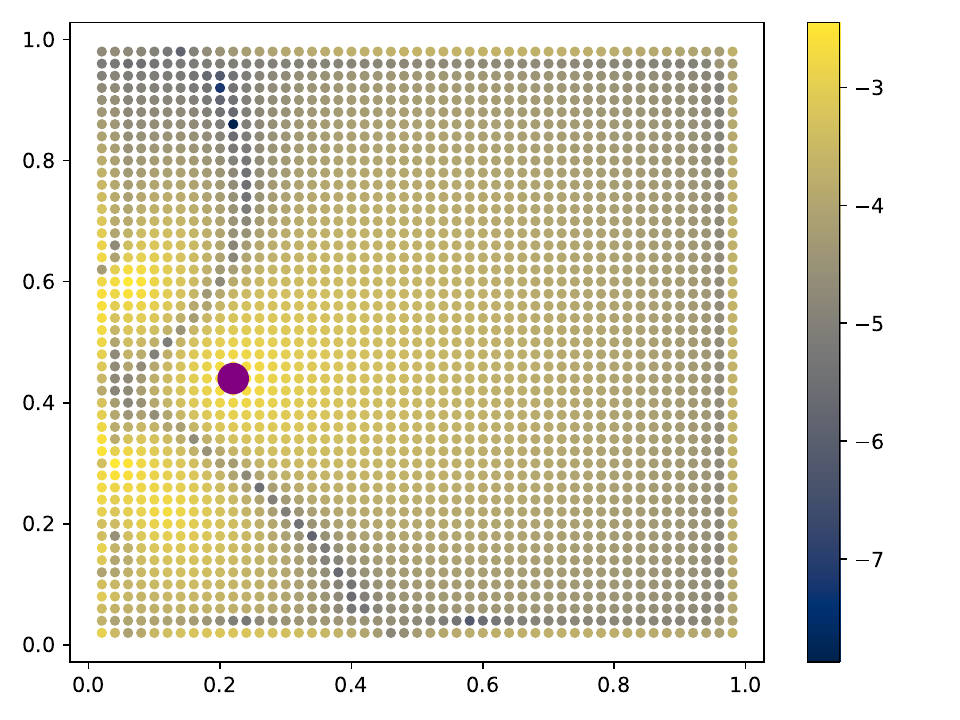}
    \includegraphics[width=6cm]{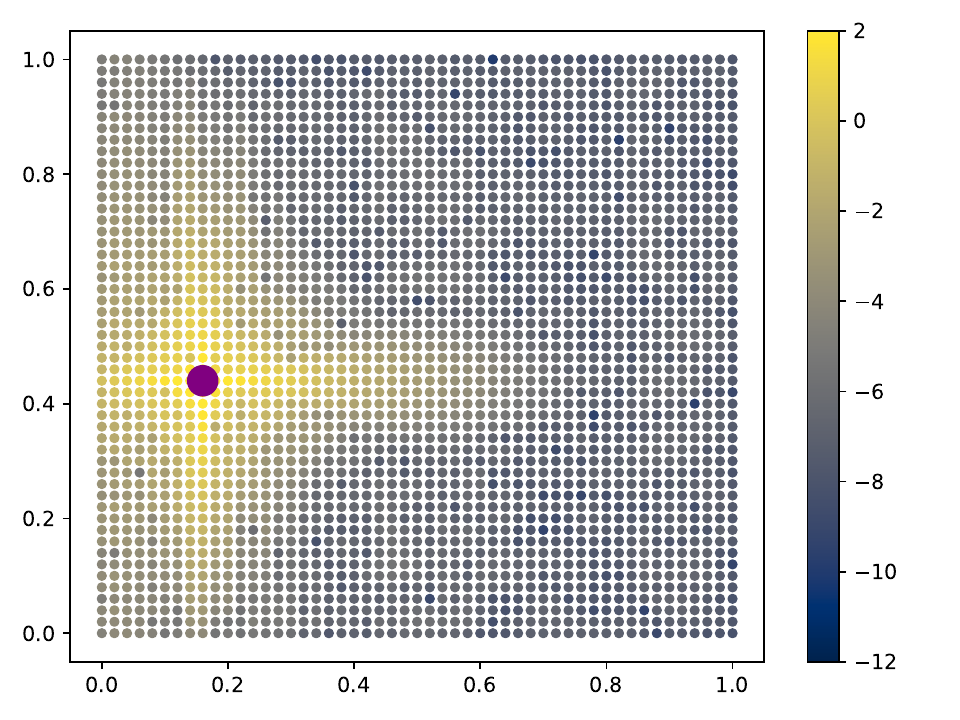}
    \caption{Demonstration of screening effects in the context of derivative-type measurements using the Mat\'ern kernel with $\nu = 5/2$ and lengthscale $0.3$. The data points are equidistributed in $[0,1]^2$ with grid size $h=0.02$. In the left figure, we order the Laplacian measurements first and then select a Diracs measurement which is the big point. The points are colored with intensity according to $|U_{ij}^\star|$. In the right figure, we order the Dircas measurements first and then select a Laplacian measurement; we display things in the same manner as the left figure.}
    \label{fig:screening effcts, derivatives}
\end{figure}


Alternatively, we can order the Dirac measurements first and then examine the same quantity as described above for any Laplacian measurement. This approach yields the right part of Figure \ref{fig:screening effcts, derivatives}, where we observe a fast decay as desired. This indicates that the derivative measurements should come after the Dirac measurements, or equivalently, that the derivative measurements should be treated as \textit{finer scales} compared to the pointwise measurements.

With the above observation, we can design our new ordering as follows. For the nonlinear elliptic PDE example in Subsection \ref{sec: The finite dimensional problem}, we order the Diracs measurements $\phi^{(1)}_m=\updelta_{\vx_m}, 1\leq m \leq M$ first using the maximin ordering with $\sfA = \emptyset$ mentioned earlier. Then, we add the derivative-type measurements $ \updelta_{\vx_m} \circ \Delta, 1\leq m \leq M_{\Omega}$ in arbitrary order to the ordering.

Again, for our notations, we use $P: I_N \to I$ to map the index of the ordered measurements to the index of the corresponding points. Here $I_N :=\{1,2,..., N\}$,  $N = M + M_{\Omega}$ and the cardinality of $I$ is $M$. We define the lengthscales of the ordered measurements to be
\begin{equation}
\label{eqn: lengthscale def, general PDE}
    l_i = 
    \begin{cases}
    \mathrm{dist}(\vx_{P(i)}, \{\vx_{P(1)},...,\vx_{P(i-1)}\} \cup \sfA) ,& \text{if } i \leq M,\\
    l_M,              & \text{otherwise.} \
\end{cases}
\end{equation}
We will justify the above choice of lengthscales in our theoretical study in Section \ref{sec: the theory}.

With the ordering and the lengthscales determined, we can apply the same steps in the last subsection to identify sparsity patterns:
\begin{equation}
    \label{eqn: sparsity pattern, general PDE}
    S_{P,l,\rho} = \{(i,j) \subset I_N \times I_N: i \leq j,\,  \operatorname{dist}(\vx_{P(i)},\vx_{P(j)}) \leq \rho l_j \} \subset I_N \times I_N\, .
\end{equation}
Then, we can use KL minimization as in Subsection \ref{sec: KL minimization} (see \eqref{eqn: sparse mtx set}, \eqref{eqn: KL min}, and \eqref{eqn: explicit formula for KL min}) to find the optimal sparse factors under the pattern. This leads to our sparse Cholesky factorization algorithm for kernel matrices with derivative-type measurements. 

\begin{newremark}
\label{remark: complexity of cholesky, derivative data, nonlinear elliptic pde eg}
    Similar to Remark \ref{rmk: derivative-free meas, supernodes, complexity}, the above KL minimization step (with the idea of supernodes to aggregate the sparsity pattern) can be implemented in time complexity $O(N\rho^{2d})$ and space complexity $O(N\rho^{d})$, for a parameter $\rho \in \mathbb{R}^+$ that determines the size of the sparsity set. 
\end{newremark}

We present some numerical experiments to demonstrate the accuracy of such an algorithm. In Figure \ref{fig:K phi phi accuracy}, we show the error measured in the KL divergence sense, namely $\operatorname{KL}\left(\cN(0,\Theta)\parallel\cN(0,(U^{\rho}{U^\rho}^T)^{-1})\right)$ where $U^{\rho}$ is the computed sparse factor. The figures show that the KL error decays exponentially fast regarding $\rho$. The rate is faster for less smooth kernels, and for the same kernel, the rate remains the same when there are more physical points. \blue{Note that for a fixed $\rho$, when the number of points increases, the KL error will increase. This effect is captured in our Theorem \ref{thm: cholesky factor decay}, and is expected since when we have more points, finer scale information is incorporated, the matrix becomes more ill-conditioned, and the approximation error will deteriorate. Moreover, Theorem \ref{thm: KL minimization accuracy} implies that once $\rho$ increases logarithmically regarding the number of points, we can achieve the same KL error. This is also evident from the right of Figure 3.
} In the left of Figure \ref{fig:K phi phi time and reduced kernel accuracy}, we show the CPU time of the algorithm, which scales nearly linearly regarding the number of points.
\begin{figure}[ht]
    \centering
    \includegraphics[width=6cm]{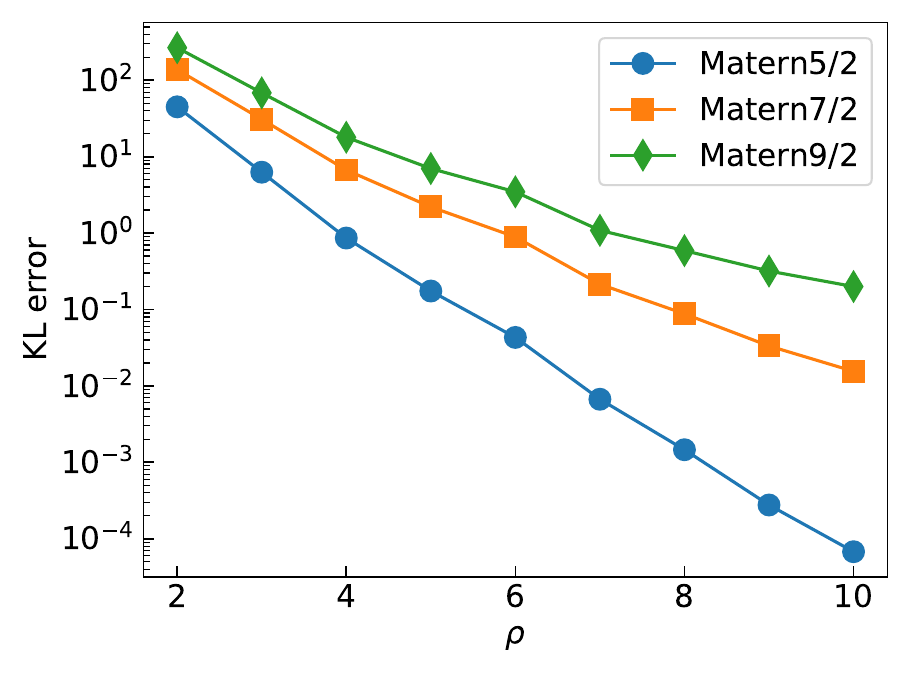}
    \includegraphics[width=6cm]{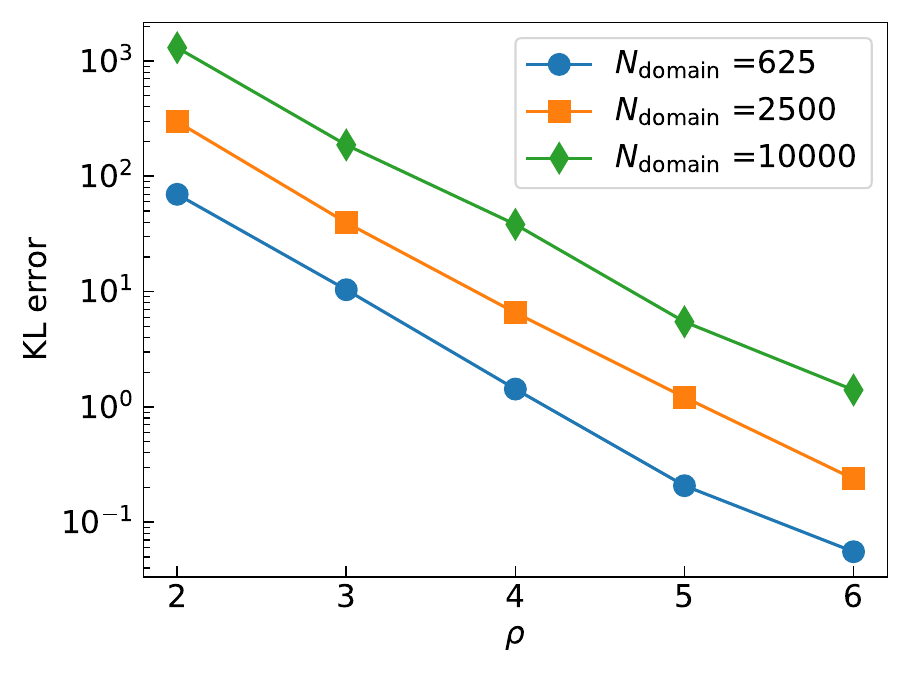}
    \caption{Demonstration of the accuracy of the sparse Cholesky factorization for $K(\bphi,\bphi)^{-1}$ in the nonlinear elliptic PDE example. In the left figure, we choose Mat\'ern kernels with $\nu = 5/2,7/2,9/2$ and lengthscale $l=0.3$; the physical points are fixed to be equidistributed in $[0,1]^2$ with grid size $h=0.05$; we plot the error measured in the KL sense with regard to different $\rho$. In the right figure, we fix the Mat\'ern kernels with $\nu = 5/2$ and lengthscale $l=0.3$. We vary the number of physical points, which are equidistributed with grid size $h=0.04, 0.02, 0.01$; thus $N_{\mathrm{domain}} = 625, 2500, 10000$ correspondingly.}
    \label{fig:K phi phi accuracy}
\end{figure}

\begin{figure}[ht]
    \centering
    \includegraphics[width=6cm]{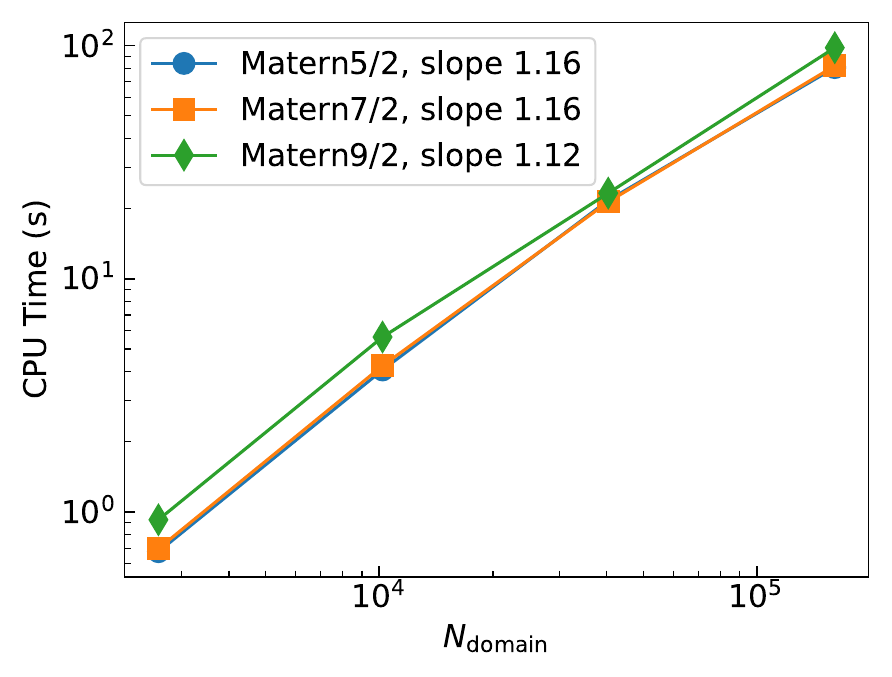}
    \includegraphics[width=6cm]{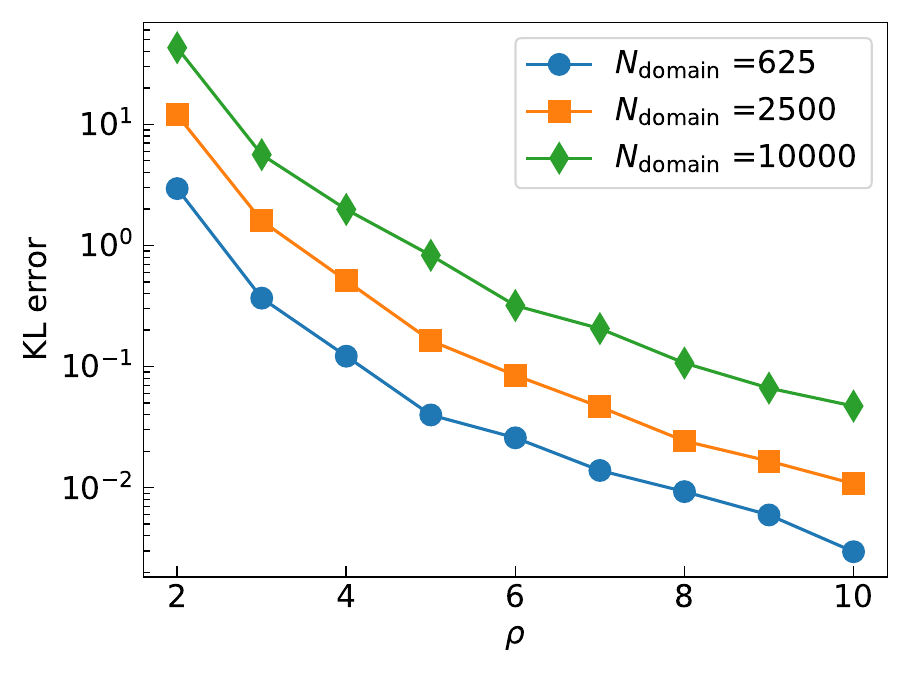}
    \caption{In the left figure, we choose Mat\'ern kernels with $\nu = 5/2,7/2,9/2$ and lengthscale $l=0.3$; the physical points are equidistributed in $[0,1]^2$ with different grid sizes; we plot the CPU time of the factorization algorithm for $K(\bphi,\bphi)^{-1}$, using the personal computer MacBook Pro 2.4 GHz Quad-Core Intel Core i5. In the right figure, we study the sparse Cholesky factorization for the reduced kernel matrix $K(\bphi^k, \bphi^k)^{-1}$. We fix the Mat\'ern kernels with $\nu = 5/2$ and lengthscale $l=0.3$. We vary the number of physical points, which are equidistributed with grid size $h=0.04, 0.02, 0.01$. We plot the KL error with regard to different $\rho$.}
    \label{fig:K phi phi time and reduced kernel accuracy}
\end{figure}

\subsubsection{General case} \label{sec: derivative measurements sparse cholesky factorization}
We present the algorithm discussed in the last subsection for general PDEs. In the general case \eqref{eqn: general opt for general nonlinear PDEs}, we need to deal with $K(\bphi, \bphi)$ where $\bphi$ is the concatenation of Diracs measurements and derivative-type measurements that are derived from the PDE. Suppose the number of physical points is $M$; they are $\{\vx_1,...,\vx_M\}$ and the index set is denoted by $I = \{1,...,M\}$. Without loss of generality, we can assume $\bphi$ contains Diracs measurements $\updelta_{\vx_m}$ at all these points and some derivative measurements at these points up to order $J \in \mathbb{N}$. See the definition of derivative measurements in Remark \ref{remark: def of derivative measurements}.
The reason we can assume Diracs measurements are in $\bphi$ is that one can always add $0u(\vx)$ to the PDE if there are no terms involving $u(\vx)$. The presence of these Diracs measurements is the key to get provable guarantee of the algorithm; for details see Section \ref{sec: the theory}.

Denote the total number of measurements by $N$, as before. We order the Diracs measurements $\updelta_{\vx_m}, 1\leq m \leq M$ first using the maximin ordering with $\sfA = \emptyset$. Then, we add the derivative-type measurements in an arbitrary order to the ordering.

Similar to the last subsection, we use the notation $P: I_N \to I$ to map the index of the ordered measurements to the index of the corresponding points; here $I_N :=\{1,2,..., N\}$. The lengthscales of the ordered measurements are defined via \eqref{eqn: lengthscale def, general PDE}.
With the ordering, one can identify the sparsity pattern as in \eqref{eqn: sparsity pattern, general PDE} (and aggregate it using supernodes as discussed in Remark \ref{rmk: derivative-free meas, supernodes, complexity}) and use the KL minimization \eqref{eqn: sparse mtx set}\eqref{eqn: KL min}\eqref{eqn: explicit formula for KL min} to compute $\epsilon$-approximate factors the same way as before. We outline the general algorithmic procedure in Algorithm \ref{alg:Sparse Cholesky for K bphi bphi}.

\begin{algorithm}
\caption{Sparse Cholesky factorization for $K(\bphi,\bphi)^{-1}$}
\label{alg:Sparse Cholesky for K bphi bphi}
\begin{algorithmic}[1]
\STATE{\textbf{Input}: Measurements $\bphi$, kernel function $K$, sparsity parameter $\rho$, supernodes parameter $\lambda$}
\STATE{\textbf{Output}: $U^{\rho}, P_{\rm perm} \ \mathrm{s.t.}\  K(\bphi,\bphi)^{-1} \approx P_{\rm perm}^TU^{\rho}{U^{\rho}}^TP_{\rm perm}$}
\STATE{Reordering and sparsity pattern: 
we first order the Diracs measurements using the maximin ordering. Next, we order the derivative measurements in an arbitrary order. This process yields a permutation matrix denoted by $P_{\rm perm}$, such that $P_{\rm perm}\bphi = \tilde{\bphi}$, and lengthscales $l$ for each measurement in $\tilde{\bphi}$. Under the ordering, we construct the aggregate sparsity pattern $S_{P,l,\rho, \lambda}$ based on the chosen values of $\rho$ and $\lambda$.
}
\STATE{KL minimization: solve \eqref{eqn: KL min} with $\Theta = K(\tilde{\bphi},\tilde{\bphi})$, by \eqref{eqn: explicit formula for KL min}, to obtain $U^{\rho}$}
\RETURN $U^{\rho}, P_{\rm perm}$
\end{algorithmic} \end{algorithm}

The complexity is of the same order as in Remark \ref{remark: complexity of cholesky, derivative data, nonlinear elliptic pde eg}; the hidden constant in the complexity estimate depends on $J$, the maximum order of the derivative measurements. We will present theoretical analysis for the approximation accuracy in Section \ref{sec: the theory}, which implies that $\rho = O(\log(N/\epsilon))$ suffices to provide $\epsilon$-approximation for a large class of kernel matrices.

\section{Theoretical study}
\label{sec: the theory}
In this section, we perform a theoretical study of the sparse Cholesky factorization algorithm in Subsection \ref{sec: The case of derivative-type measurements} for $K(\bphi,\bphi)^{-1}$.
\subsection{Set-up for rigorous results}
\label{sec: theory, set-up}
We present the setting of kernels, physical points, and measurements for which we will provide rigorous analysis of the algorithm.
\subsubsection*{Kernel}
We first describe the domains and the function spaces. Suppose $\Omega$ is a bounded convex domain in $\bR^d$ with a Lipschitz boundary. Without loss of generality, we assume $\mathrm{diam}(\Omega) \leq 1$; otherwise, we can scale the domain. Let $H_0^s(\Omega)$ be the Sobolev space in $\Omega$ with order $s \in \bN$ derivatives in $L^2$ and zero traces. Let the operator
\[\cL: H^s_0(\Omega) \to H^{-s}(\Omega)\]
satisfy Assumption \ref{assumption on operator}. This assumption is the same as in Section 2.2 of \cite{owhadi2019operator}.
\begin{assumption} The following conditions hold for $\cL: H^s_0(\Omega) \to H^{-s}(\Omega)$:
\label{assumption on operator}
\begin{enumerate}[label=(\roman*)]
    \item symmetry: $[u, \cL v] =[v, \cL u]$;
    \item positive definiteness: $[u, \cL u] > 0$ for $\|u\|_{H^s_0(\Omega)} > 0$;
    \item boundedness: \[\|\cL\|:=\sup_u \frac{\|\cL u\|_{H^{-s}(\Omega)}}{\|u\|_{H^s_0(\Omega)}} < \infty,\quad \|\cL^{-1}\|:=\sup_u \frac{\|\cL^{-1} u\|_{H^{s}_0(\Omega)}}{\|u\|_{H^{-s}(\Omega)}} < \infty\, ;\]
    \item locality: $[u, \cL v] = 0$ if $u$ and $v$ have disjoint supports.
\end{enumerate}
\end{assumption}
We assume $s>d/2$ so Sobolev's embedding theorem shows that $H^s_0(\Omega) \subset C(\Omega)$, and thus $\updelta_{\vx} \in H^{-s}(\Omega)$ for $\vx \in \Omega$. We consider the kernel function to be the Green function $K(\vx,\vy) := [\updelta_{\vx},  \mathcal{L}^{-1} \updelta_{\vy}]$. An example of $\cL$ could be $(-\Delta)^{s}$; we use the zero Dirichlet boundary condition to define $\cL^{-1}$, which leads to a Mat\'ern-like kernel.
\subsubsection*{Physical points}
Consider a scattered set of points $\{\vx_i\}_{i\in I} \subset \Omega$, where $I = \{1,2,...,M\}$ as in Subsection \ref{sec: Previous results for derivative-free measurements}; the homogeneity parameter of these points is assumed to be positive:
\[\delta(\{\vx_i\}_{i\in I}; \partial\Omega) = \frac{\min_{\vx_i\neq\vx_j \in I} \mathrm{dist}(\vx_i,\{\vx_j\}\cup \partial\Omega)}{\max_{\vx\in\Omega} \mathrm{dist}(\vx, \{\vx_i\}_{i\in I} \cup \partial \Omega)} > 0\, .  \]
This condition ensures that the points are scattered homogeneously. Here we set $\sfA = \partial\Omega$ since we consider zero Dirichlet's boundary condition and no points will be on the boundary. The accuracy in our theory will depend on $\delta(\{\vx_i\}_{i\in I}; \partial\Omega)$.
\subsubsection*{Measurements}
The setting is the same as in Subsection \ref{sec: derivative measurements sparse cholesky factorization}. We assume $\bphi$ contains Diracs measurements at \textit{all} of the scattered points, and it also contains derivative-type measurements at some of these points up to order $J \in \bN$. We require $J < s - d/2$ so that the Sobolev embedding theorem guarantees these derivative measurements are well-defined.

For simplicity of analysis, we assume all the measurements are of the type $\updelta_{\vx_i} \circ D^{\gamma}$ with the multi-index $\gamma = (\gamma_1,...,\gamma_d) \in \bN^d$ and $|\gamma|:=\sum_{k=1}^d \gamma_k \leq J$; here $D^{\gamma}=D^{\gamma_1}_{\vx^{(1)}}\cdots D^{\gamma_d}_{\vx^{(d)}}$; see Remark \ref{remark: def of derivative measurements}. Note that when $|\gamma|=0$, $\updelta_{\vx_i} \circ D^{\gamma}$ corresponds to Diracs measurements.
The total number of measurements is denoted by $N$. 

Note that the aforementioned assumption does not apply to the scenario of Laplacian measurements in the case of a nonlinear elliptic PDE example. This exclusion is solely for the purpose of proof convenience, as it necessitates linear independence of measurements. However, similar proofs can be applied to Laplacian measurements once linear independence between the measurements is ensured; \blue{see Remark \ref{appendix-laplace-measurements-proof}.}

\subsection{Theory}
Under the setting in Subsection \ref{sec: theory, set-up}, we consider the ordering $P:\{1,2,...,N\} \to \{1,2,...,M\}$ described in Subsection \ref{sec: derivative measurements sparse cholesky factorization}. Recall that for this $P$, we first order the Diracs measurements using the maximin ordering conditioned on $\partial\Omega$ (since there are no boundary points); then, we follow the ordering with an arbitrary order of the derivative measurements. The lengthscale parameters are defined via
\begin{equation*}
    l_i = 
    \begin{cases}
    \mathrm{dist}(\vx_{P(i)}, \{\vx_{P(1)},...,\vx_{P(i-1)}\} \cup \partial\Omega) ,& \text{if } i \leq M,\\
    l_M,              & \text{otherwise.} \
\end{cases}
\end{equation*}
We write $\Theta = K(\tilde{\bphi},\tilde{\bphi}) \in \bR^{N\times N} $, which is the kernel matrix after reordering the measurements in $\bphi$ to $\tilde{\bphi} = (\phi_{P(1)},...,\phi_{P(N)})$. 
\begin{theorem}
\label{thm: cholesky factor decay}
Under the setting in Subsection \ref{sec: theory, set-up} and the above given ordering $P$, we consider the upper triangular Cholesky factorization $\Theta^{-1}=U^\star {U^\star}^T$. Then, for $1\leq i\leq j \leq N$, we have
\[\left|\frac{U^\star_{ij}}{U^\star_{jj}}\right|  \leq C l_j^{-2s}\exp\left(-\frac{\mathrm{dist}(\vx_{P(i)}, \vx_{P(j)})}{Cl_j}\right)\ \ \mathrm{ and }\ \ |U^\star_{jj}|\leq Cl_M^{-s+d/2}\, ,\]
where $C$ is a generic constant that depends on $\Omega, \delta(\{\vx_i\}_{i\in I}; \partial\Omega), d, s, J, \|\cL\|, \|\cL^{-1}\|$.
\end{theorem}

The proof for Theorem \ref{thm: cholesky factor decay} can be found in Appendix \ref{appendix: Proof of thm: cholesky factor decay}. The proof relies on the interplay between GP regression, linear algebra, and numerical homogenization. Specifically, we use \eqref{eqn: U as conditional GP} to represent the ratio $U_{ij}^\star/U_{jj}^\star$ as the normalized conditional covariance of a GP. Our technical innovation is to connect this normalized term to the conditional expectation of the GP, leading to the identity
\begin{equation*}
\left|\frac{U^\star_{ij}}{U^\star_{jj}}\right| = \left|\frac{\mathrm{Cov}[Y_i,Y_j|Y_{1:j-1\backslash {i}}]}{\mathrm{Var}[Y_i|Y_{1:j-1\backslash {i}}]}\right| = \left|\bE[Y_j|Y_i = 1, Y_{1:j-1\backslash {i}} = 0]\right|\, ,
\end{equation*}
where $Y \sim \cN(0,\Theta)$. This conditional expectation directly connects to the operator-valued wavelets, or \textit{Gamblets}, in the numerical homogenization literature \cite{owhadi2017multigrid,owhadi2019operator}. We can apply PDE tools to establish the decay result of Gamblets. Remarkably, the connection to the conditional expectation simplifies the analysis for general measurements, compared to the more lengthy proof based on exponential decay matrix algebra in \cite{schafer2021compression} for Diracs measurements only.

In our setting, we need additional analytic results regarding the derivative measurements to prove the exponential decay of the Gamblets, which is one of the technical contributions of this paper. Finally, for $|U_{jj}^\star|$, we obtain the estimate by bounding the lower and upper eigenvalues of $\Theta$. For details, see Appendix \ref{appendix: Proof of thm: cholesky factor decay} and \ref{prop: eigenvalue lower bounds on K phi phi}.

With Theorem \ref{thm: cholesky factor decay}, we can establish that $|U_{ij}^\star|$ is exponentially small when $(i,j)$ is outside the sparsity set $S_{P, l,\rho}$. This property enables us to show that the sparse Cholesky factorization algorithm leads to provably accurate sparse factors when $\rho = O(\log(N/\epsilon))$. See Theorem \ref{thm: KL minimization accuracy} for details.
\begin{theorem}
\label{thm: KL minimization accuracy}
Under the setting in Theorem \ref{thm: cholesky factor decay}, suppose $U^{\rho}$ is obtained by the KL minimization \eqref{eqn: KL min} with the sparsity parameter $\rho \in \bR^{+}$. Then, there exists a constant depending on $\Omega, \delta(\{\vx_i\}_{i\in I}; \partial\Omega), d, s, J, \|\cL\|,$ $\|\cL^{-1}\|$, such that if $\rho \geq C \log(N/\epsilon)$, we have
\begin{equation*}
    \operatorname{KL}\left(\cN(0,\Theta)\parallel\cN(0,(U^{\rho} {U^{\rho}}^T)^{-1})\right) + \|\Theta^{-1} - U^{\rho} {U^{\rho}}^T\|_{\mathrm{Fro}} + \|\Theta - (U^{\rho} {U^{\rho}}^T)^{-1}\|_{\mathrm{Fro}} \leq \epsilon\, ,
\end{equation*}
where $\epsilon<1$ and $\|\cdot\|_{\mathrm{Fro}}$ is the Frobenius norm. 
\end{theorem}

The proof can be found in Appendix \ref{appendix: Proof of thm: KL minimization accuracy}. It is based on the KL optimality of $U^{\rho}$ and a comparison inequality between KL divergence and Frobenious norm shown in Lemma B.8 of \cite{schafer2021sparse}. 
\begin{newremark}
    Theorem \ref{thm: KL minimization accuracy} will still hold when the idea of supernodes in Remark \ref{rmk: derivative-free meas, supernodes, complexity} is used since it only makes the sparsity pattern larger.
\end{newremark}

\section{Second order optimization methods} 
\label{sec: Gauss-Newton and preconditioned conjugate gradient}
Using the algorithm in Subsection \ref{sec: The case of derivative-type measurements}, we get a sparse Cholesky factorization for $K(\bphi, \bphi)^{-1}$, and thus we have a fast evaluation of the loss function in \eqref{eqn: finite dim optimization for elliptic eqn} (and more generally in \eqref{eqn: general opt for general nonlinear PDEs}) and its gradient. Therefore, first-order methods can be implemented efficiently.

In \cite{chen2021solving}, a second-order Gauss-Newton algorithm is used to solve the optimization problem and is observed to converge very fast, typically in $3$ to $8$ iterations. In this subsection, we discuss how to make such a second-order method scalable based on the sparse Cholesky idea. As before, we first illustrate our ideas on the nonlinear elliptic PDE example \eqref{running-example-optimization-problem} and then describe the general algorithm.
\subsection{Gauss-Newton iterations} 
\label{sec: Gauss-Newton iterations} For the nonlinear elliptic PDE example, the optimization problem we need to solve is \eqref{eqn: finite dim optimization for elliptic eqn}. 
Using the equation $ z^{(2)}_m =\tau(z^{(1)}_m)-f(\bx_m)$ and the boundary data,
  we can eliminate $\bz^{(2)}$ and rewrite \eqref{eqn: finite dim optimization for elliptic eqn} as an unconstrained problem:
   \begin{equation}\label{running-example-unconstrained-optimization-problem}
     \minimize_{\bz^{(1)}_{\Omega} \in \R^{\Md}}~
     \big(\bz^{(1)}_{\Omega}, g(\bx_{\partial\Omega}), f(\bx_{\Omega})-\tau(\bz^{(1)}_\Omega)\big)
     \cc(\bphi,\bphi)^{-1}
     \begin{pmatrix}\bz^{(1)}_{\Omega} \\
       g(\bx_{\partial\Omega}) \\
       f(\bx_{\Omega})-\tau(\bz^{(1)}_{\Omega})\end{pmatrix},
   \end{equation}
where $\bz^{(1)}_{\Omega}$ denotes the $\Md$-dimensional vector of the $z_i$ for $i=1, \dots, \Md$ associated to the interior points $\bx_{\Omega}$ while $f(\bx_{\Omega}), g(\bx_{\partial\Omega})$ and $\tau(\bz^{(1)}_\Omega)$ are vectors obtained by applying the corresponding functions to entries of their input vectors. 
To be clear, the expression in \eqref{running-example-unconstrained-optimization-problem} represents a weighted least-squares optimization problem, and the transpose signs in the row vector multiplying the matrix have been suppressed for notational brevity. In \cite{chen2021solving}, a Gauss-Newton method has been proposed to solve this problem. This method linearizes the nonlinear function $\tau$ at the current iteration and solves the resulting quadratic optimization problem to obtain the next iterate.

In this paper, we present the Gauss-Newton algorithm in a slightly different but equivalent way that is more convenient for exposition. To that end, we consider the general formulation in \eqref{eqn: general opt for general nonlinear PDEs}. In the nonlinear elliptic PDE example, we have $\bphi = (\updelta_{\vx_{\Omega}}, \updelta_{\vx_{\partial \Omega}}, \updelta_{\vx_\Omega}\circ \Delta)$ where $\updelta_{\vx_{\Omega}}$ denotes the collection of Diracs measurements $(\updelta_{\vx_{1}}, ..., \updelta_{\vx_{M_{\Omega}}})$; the definition of $\updelta_{\vx_{\partial \Omega}}$ and $\updelta_{\vx_\Omega}\circ \Delta$ follows similarly. We also write correspondingly $\vz = (\vz_{\Omega},\vz_{\partial\Omega}, \vz_{\Omega}^{\Delta}) \in \bR^{N} $ with $N = M+M_{\Omega}$. Then $F(\vz) = (-\vz_{\Omega}^{\Delta} + \tau(\vz_{\Omega}), \vz_{\partial\Omega}) \in \bR^M$ and $\vy = (f(\vx_{\Omega}), g(\vx_{\partial\Omega})) \in \bR^M$, such that $F(\vz) = \vy$.

The Gauss-Netwon iterations for solving \eqref{running-example-unconstrained-optimization-problem} is equivalent to the following sequential quadratic programming approach for solving \eqref{eqn: general opt for general nonlinear PDEs}: for $k \in \bN$, assume $\vz^{k}$ obtained, then $\vz^{k+1}$ is given by
\begin{equation}
\label{eqn: linearize the constraint}
\begin{aligned}
    \vz^{k+1} = &\argmin_{\vz \in \bR^N} \quad  \vz^T K(\bphi, \bphi)^{-1} \vz \\
    &\mathrm{s.t.} \quad   { F(\vz^k) + DF(\vz^k)(\vz - \vz^k)} = { \vy}\, ,
\end{aligned}
\end{equation}
where $DF(\vz^k) \in \bR^{M\times N}$ is the Jacobian of $F$ at $\vz^k$. The above is a quadratic optimization with a linear constraint. Using Lagrangian multipliers, we get the explicit formula of the solution: $\vz^{k+1} = K(\bphi,\bphi)(DF(\vz^k))^T \gamma$, where $\gamma \in \bR^{M}$ solves the linear system
\begin{equation}
    \left(DF(\vz^k)K(\bphi,\bphi)(DF(\vz^k))^T\right)\gamma = \vy - F(\vz^k) + DF(\vz^k)\vz^k\, .
\end{equation}
Now, we introduce the \textit{reduced} set of measurements $\bphi^{k} = DF(\vz^k)\bphi$. For the nonlinear elliptic PDE, we have \[\bphi^{k} = (-\updelta_{\vx_\Omega}\circ \Delta + \tau(\vz_{\Omega}^k) \cdot \updelta_{\vx_\Omega}, \updelta_{\vx_{\partial\Omega}})\] where $\tau(\vz^k_{\Omega}) \cdot \updelta_{\vx_\Omega}:= (\tau(\vz_1^k)\updelta_{\vx_1}, ..., \tau(\vz^k_{M_{\Omega}})\updelta_{\vx_{M_{\Omega}}})$. Then, we can equivalently write the solution as $\vz^{k+1} = K(\bphi,\bphi)(DF(\vz^k))^T\gamma$ where $\gamma$ satisfies
\begin{equation}
\label{eqn: reduced mtx linear system}
    K(\bphi^k,\bphi^k)\gamma = \vy - F(\vz^k) + DF(\vz^k)\vz^k\, .
\end{equation}
Note that $K(\bphi^k,\bphi^k) \in \bR^{M\times M}$, in contrast to  $K(\bphi,\bphi) \in \bR^{N\times N}$. The dimension is reduced.
The computational bottleneck lies in 
the linear system with the reduced kernel matrix $K(\bphi^k,\bphi^k)$. 
\subsection{Sparse Cholesky factorization for the reduced kernel matrices} 
\label{sec: Sparse Cholesky factorization for the reduced kernel matrices}
As $K(\bphi^k,\bphi^k)$ is also a kernel matrix with derivative-type measurements, we hope to use the sparse Cholesky factorization idea to approximate its inverse. The first question, again, is how to order these measurements. 

To begin with, we examine the structure of the reduced kernel matrix. Note that as $F(\vz)=\vy$ encodes the PDE at the collocation points, the linearization of $F$ in \eqref{eqn: linearize the constraint} is also equivalent to first linearizing the PDE at the current solution and then applying the kernel method.  Thus, $\bphi^k$ will typically contain $M_{\Omega}$ interior measurements corresponding to the linearized PDE at the interior points and $M-M_{\Omega}$ boundary measurements corresponding to the sampled boundary condition. For problems with Dirichlet's boundary condition, which are the main focus of this paper, the boundary measurements are of Diracs type. It is worth noting that, in contrast to $K(\bphi,\bphi)$, we no longer have Diracs measurements at every interior point now.

We propose to order the boundary Diracs measurements first, using the maximin ordering on $\partial \Omega$. Then, we order the interior derivative-type measurements using the maximin ordering in $\Omega$, conditioned on $\partial \Omega$. We use numerical experiments to investigate the screening effects under such ordering. Suppose $\Theta$ is the reordered version of the reduced kernel matrix $K(\bphi^k,\bphi^k)$, then similar to Figures \ref{fig:screening effcts, diracs} and \ref{fig:screening effcts, derivatives}, we show the magnitude of the corresponding Cholesky factor of $\Theta^{-1}=U^\star {U^\star}^T$, i.e., we plot $|U^\star_{ij}|$ for $i \leq j$; here $j$ is selected to correspond to some boundary and interior points.  The result can be found in Figure \ref{fig:screening effcts, reduced mtx}.
\begin{figure}[ht]
    \centering
    \includegraphics[width=6cm]{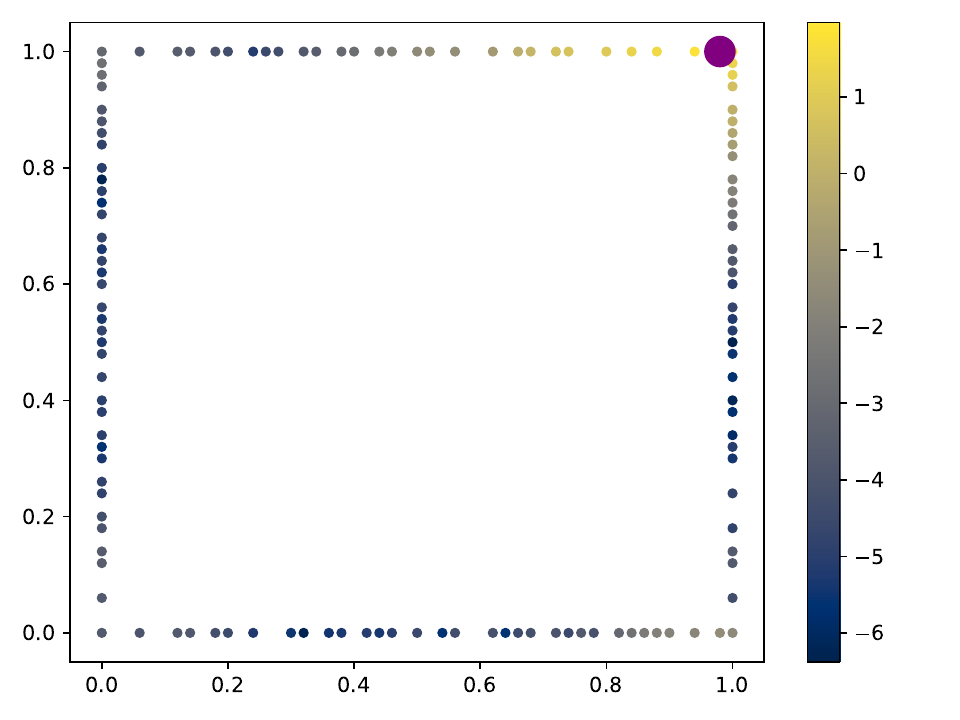}
    \includegraphics[width=6cm]{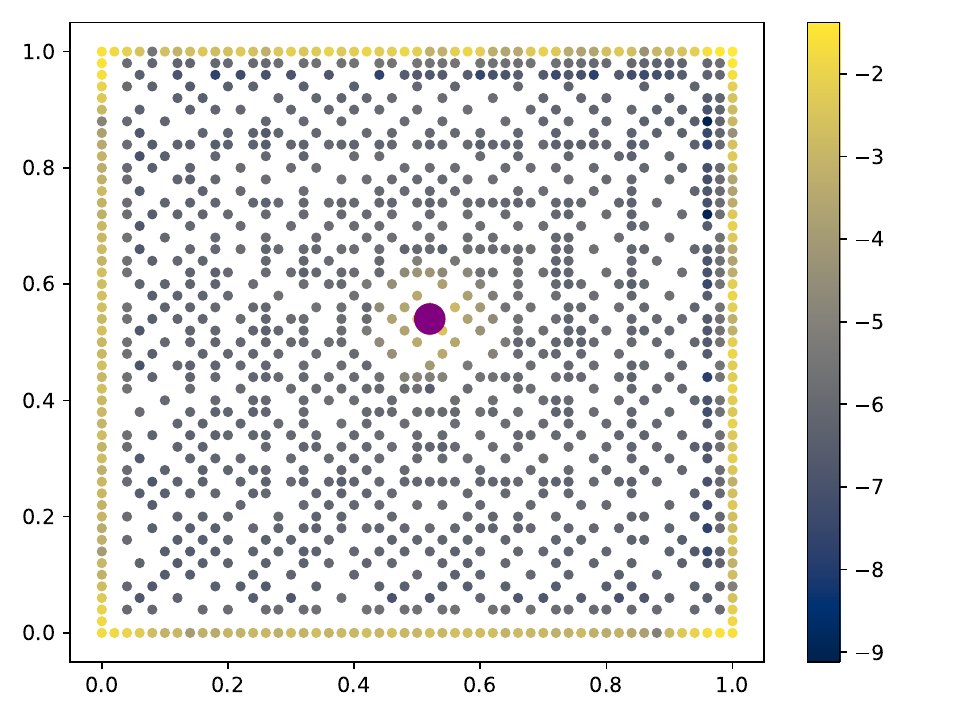}
    \caption{Demonstration of screening effects for the reduced kernel matrix. We choose the Mat\'ern kernel with $\nu = 5/2$; the lengthscale parameter is $0.3$. The data points are equidistributed in $[0,1]^2$ with grid size $h=0.02$. In the left figure, we show a boundary point, and all the points ordered before are marked with a color whose intensity scales with the entry value $|U_{ij}^\star|$. The right figure is obtained in the same manner but for an interior measurement.}
    \label{fig:screening effcts, reduced mtx}
\end{figure}

From the left figure, we observe a desired screening effect for boundary Diracs measurements. However, in the right figure, we observe that the interior derivative-type measurements still exhibit a strong conditional correlation with boundary points. That means that the correlation with boundary points is not screened thoroughly. This also implies that the presence of the interior Diracs measurements is the key to the sparse Choleksy factors for the previous $K(\tilde{\bphi},\tilde{\bphi})^{-1}$.

The right of Figure \ref{fig:screening effcts, reduced mtx} demonstrates a negative result: one cannot hope that the Cholesky factor of $\Theta^{-1}$ will be as sparse as before. However, algorithmically, we can still apply the sparse Cholesky factorization to the matrix.
We present numerical experiments to test the accuracy of such factorization. In the right of Figure \ref{fig:K phi phi time and reduced kernel accuracy}, we show the KL errors of the resulting factorization concerning the sparsity parameter $\rho$. Even though the screening effect is not perfect, as we discussed above, we still observe a consistent decay of the KL errors when $\rho$ increases.

In addition, although we cannot theoretically guarantee the factor is as accurate as before, we can use it as a preconditioner to solve linear systems involving $K(\bphi^k,\bphi^k)$. In practice, we observe that this idea works very well, and nearly constant steps of preconditioned conjugate gradient (pCG) iterations can lead to an accurate solution to \eqref{eqn: reduced mtx linear system}. As a demonstration, in Figure \ref{fig:pcg convergence history}, we show the pCG iteration history when the preconditioning idea is employed. The stopping criterion for pCG is that the relative tolerance is smaller than $2^{-26} \approx 10^{-8}$, which is the default criterion in Julia. From the figures, we can see that pCG usually converges in 10-40 steps, and this fast convergence is insensitive to the numbers of points. When $\rho$ is large, the factor is more accurate, and the preconditioner is better, leading to smaller number of required pCG steps. Among all the cases, the number of pCG steps required to reach the stopping criterion is of the same magnitude, and is not large.
\begin{figure}[ht]
    \centering
    \includegraphics[width=6cm]{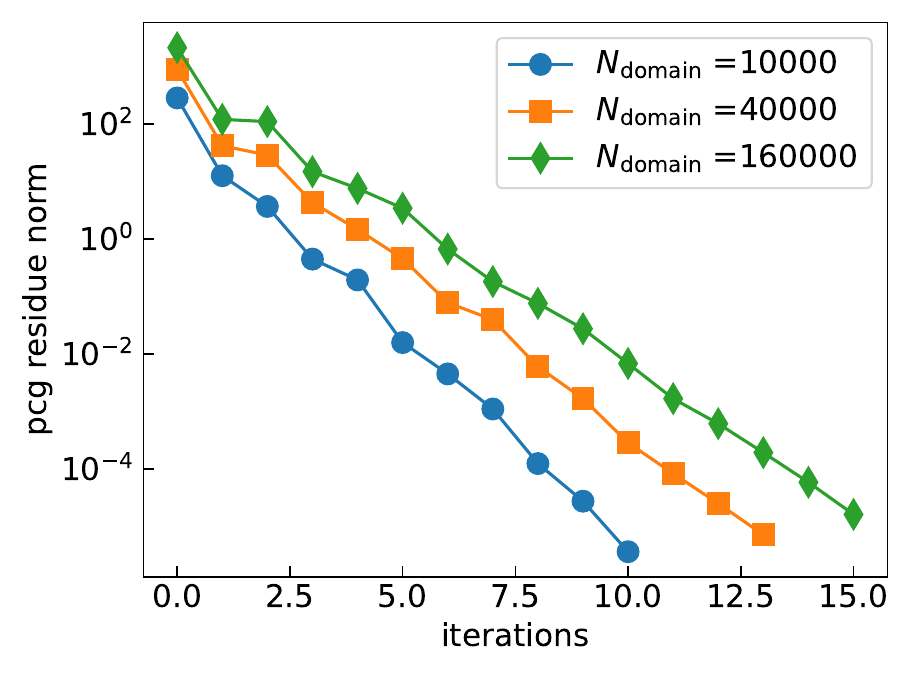}
    \includegraphics[width=6cm]{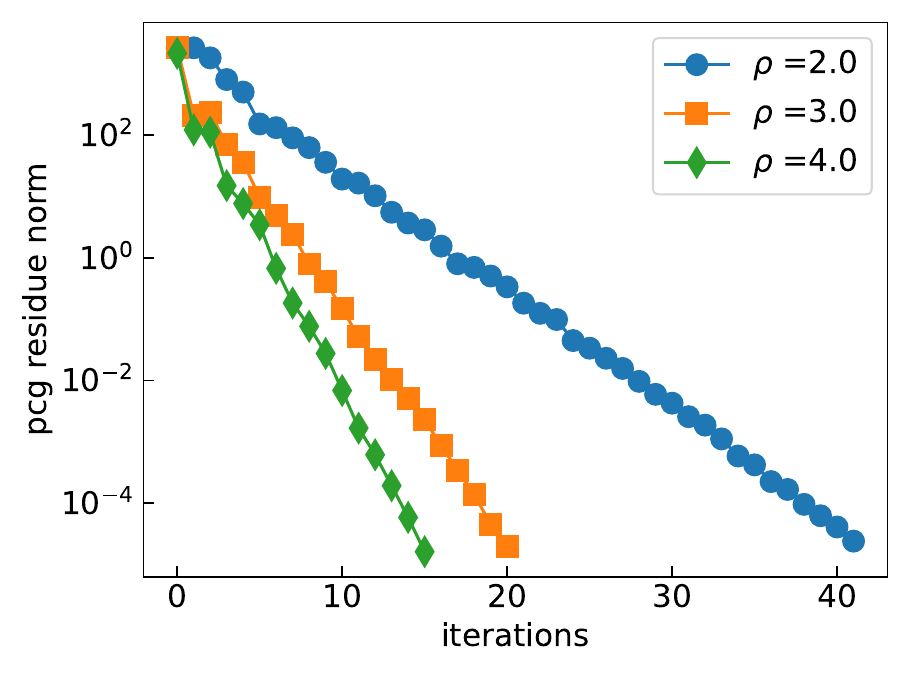}
    \caption{Demonstration of the convergence history of the pCG iteration. We choose the Mat\'ern kernel with $\nu = 5/2$; the lengthscale parameter is $0.3$. 
    In the left figure, we choose the data points to be equidistributed in $[0,1]^2$ with different grid sizes; $\rho = 4.0$. We show the equation residue norms of the iterates in each pCG iteration. In the right figure, we choose the data points to be equidistributed in $[0,1]^2$ with grid size $0.0025$ so that $N_{\rm domain} = 160000$. We plot the pCG iteration history for $\rho=2.0,3.0,4.0$.}
    \label{fig:pcg convergence history}
\end{figure}

It is worth noting that since $K(\bphi^k,\bphi^k) = DF(\vz^k)K(\bphi,\bphi)(DF(\vz^k))^T$ and we have a provably accurate sparse Cholesky factorization for $K(\bphi,\bphi)^{-1}$, the matrix-vector multiplication for $K(\bphi^k,\bphi^k)$ in each pCG iteration is efficient.

\subsection{General case} The description in the last subsection applies directly to general nonlinear PDEs, which correspond to general $\bphi$ and $F$ in \eqref{eqn: general opt for general nonlinear PDEs}. We use the maximin ordering on the boundary, followed by the conditioned maximin ordering in the interior. We denote the ordering by $Q: I \to I$. The lengthscale is defined by
\begin{equation}
    l_i = \mathrm{dist}(\vx_{Q(i)}, \{\vx_{Q(1)},...,\vx_{Q(i-1)}\} \cup \sfA)\, ,
\end{equation}
where for a boundary measurement, $\sfA=\emptyset$ and for an interior measurement $\sfA = \partial \Omega$.
With the ordering and lengthscales, we create the sparse pattern through \eqref{eqn: def sparsity pattern} (and aggregate it using the supernodes idea) and apply the KL minimization in Subsection \ref{sec: Previous results for derivative-free measurements} to obtain an approximate factor for $K(\bphi^k,\bphi^k)^{-1}$. The general algorithmic procedure is outlined in Algorithm \ref{alg:Sparse Cholesky for K bphi bphi reduced}. We now denote the sparsity parameter for the reduced kernel matrix by $\rho_{\rm r}$.
\begin{algorithm}
\caption{Sparse Cholesky factorization for $K(\bphi^k,\bphi^k)^{-1}$}
\label{alg:Sparse Cholesky for K bphi bphi reduced}
\begin{algorithmic}[1]
\STATE{\textbf{Input}: Measurements $\bphi^k$, kernel function $K$, sparsity parameter $\rho_{\rm r}$, supernodes parameter $\lambda$}
\STATE{\textbf{Output}: $U^{\rho_{\rm r}}_{\rm r}, Q_{\rm perm}$}
\STATE{Reordering and sparsity pattern: 
we first order the boundary measurements using the maximin ordering. Next, we order the interior measurements using the maximin ordering conditioned on $\partial\Omega$. This process yields a permutation matrix denoted by $Q_{\rm perm}$ such that $Q_{\rm perm}\bphi^k =\tilde{\bphi}^k$, and lengthscales $l$ for each measurement in $\tilde{\bphi}^k$. Under the ordering, we construct the aggregate sparsity pattern $S_{Q,l,\rho_{\rm r}, \lambda}$ based on the chosen values of $\rho_{\rm r}$ and $\lambda$.
}
\STATE{KL minimization: solve \eqref{eqn: KL min} with $\Theta = K(\tilde{\bphi}^k,\tilde{\bphi}^k)$, by \eqref{eqn: explicit formula for KL min}, to obtain $U^{\rho_{\rm r}}_{\rm r}$}
\RETURN $U^{\rho_{\rm r}}_{\rm r}, Q_{\rm perm}$
\end{algorithmic} \end{algorithm}

Now, putting all things together, we outline the general algorithmic procedure for solving the PDEs using the second-order Gauss-Newton method, in Algorithm \ref{alg:Sparse Cholesky accelerated Gauss-Newton}. 
\begin{algorithm}
\caption{Sparse Cholesky accelerated Gauss-Newton for solving \eqref{eqn: general opt for general nonlinear PDEs}}
\label{alg:Sparse Cholesky accelerated Gauss-Newton}
\begin{algorithmic}[1]
\STATE{\textbf{Input}: Measurements $\bphi$, data functional $F$, data vector $\vy$, kernel function $K$, number of Gauss-Newton steps $t$, sparsity parameters $\rho, \rho_{\mathrm{r}}$, supernodes parameter $\lambda$}
\STATE{\textbf{Output}: Solution $\bz^t$}
\STATE{Factorize $K(\bphi,\bphi)^{-1} \approx P_{\rm perm}^TU^{\rho}{U^{\rho}}^TP_{\rm perm}$ using Algorithm \ref{alg:Sparse Cholesky for K bphi bphi}}
\STATE{Set $k = 0$, $\bz^k = \textbf{0}$ or other user-specified initial guess}
\WHILE{$k < t$}
\STATE{Form the reduced measurements $\bphi^{k} = DF(\vz^k)\bphi$}
\STATE{Factorize $K(\bphi^k,\bphi^k)^{-1}$ to get $Q_{\rm perm}^TU_{\rm r}^{\rho_{\rm r}}{U_{\rm r}^{\rho_{\rm r}}}^T Q_{\rm perm}$ using Algorithm \ref{alg:Sparse Cholesky for K bphi bphi reduced}}
\STATE{Use pCG to solve \eqref{eqn: reduced mtx linear system} with the preconditioner $Q_{\rm perm}^TU_{\rm r}^{\rho_{\rm r}}{U_{\rm r}^{\rho_{\rm r}}}^T Q_{\rm perm}$}
\STATE{$\vz^{k+1} = (P^T_{\rm perm}U^{\rho}{U^{\rho}}^TP_{\rm perm}) \backslash ((DF(\vz^k))^T \gamma)$}
\STATE{$k=k+1$}
\ENDWHILE
\RETURN $\bz^t$
\end{algorithmic} \end{algorithm}

 For the choice of parameters, we usually set $t$ to be between 2 to 10. Setting $\rho = O(\log(N/\epsilon))$ suffices to obtain an $\epsilon$-accurate approximation of $K(\bphi,\bphi)^{-1}$. We do not have a theoretical guarantee for the factorization algorithm applied to the reduced kernel matrix $K(\bphi^k,\bphi^k)^{-1}$. Still, our experience indicates that setting $\rho_{\rm r} = \rho$ or a constant such as $\rho_{\rm r} =3$ works well in practice. 
 We note that a larger $\rho_{\rm r}$ increases the factorization time while decreasing the necessary pCG steps to solve the linear system, as demonstrated in the right of Figure \ref{fig:pcg convergence history}. There is a trade-off here in general. 
 
 The overall complexity of Algorithm \ref{alg:Sparse Cholesky accelerated Gauss-Newton} for solving \eqref{eqn: linearize the constraint} is $O(N\log^2(N)\rho^d + N\rho^{2d}+Mt\rho_{\rm r}^{2d} + T_{\rm pCG})$ in time and $O(N\rho^{d} + M\rho_{\rm r}^d)$ in space, where $O(N\log^2(N)\rho^d)$ is the time for generating the ordering and sparsity pattern, $O(N\rho^{2d})$ is for the factorization, and $O(Mt\rho_{\rm r}^{2d})$ is for the factorizations in all the GN iterations, $T_{\rm pCG}$ is the time that the pCG iterations take. 
 

Based on empirical observations, we have found that $T_{\rm pCG}$ scales nearly linearly with respect to $N\rho^d$. This is because a nearly constant number of pCG iterations are sufficient to obtain an accurate solution, and each pCG iteration takes at most $O(N\rho^d)$ time, as explained in the matrix-vector multiplication mentioned at the end of Subsection \ref{sec: Sparse Cholesky factorization for the reduced kernel matrices}. Additionally, it is worth noting that the time required for generating the ordering and sparsity pattern ($O(N\log^2(N)\rho^d)$) is negligible in practice, compared to that for the KL minimization. Furthermore, the ordering and sparsity pattern can be pre-computed once and reused for multiple runs.


\section{Numerical experiments} In this section, we use Algorithm \ref{alg:Sparse Cholesky accelerated Gauss-Newton} to solve nonlinear PDEs. The numerical experiments are conducted on the personal computer MacBook Pro 2.4 GHz Quad-Core Intel Core i5. In all the experiments, the physical data points are equidistributed on a grid; we specify its size in each example. We always set the sparsity parameter for the reduced kernel matrix $\rho_{\rm r} = \rho$, the sparsity parameter for the original matrix. We adopt the supernodes ideas in all the examples and set the parameter $\lambda = 1.5$. 

Our theory guarantees that once the Diracs measurements are ordered first by the maximin ordering, the derivative measurements can be ordered arbitrarily. In practice, for convenience, we order them from lower-order to high-order derivatives, and for the same type of derivatives, we order the corresponding measurements based on their locations, in the same maximin way as the Diracs measurements.

Our codes are in \url{https://github.com/yifanc96/PDEs-GP-KoleskySolver}.
\label{sec: Numerical experiments}
\subsection{Nonlinear elliptic PDEs}
Our first example is the nonlinear elliptic equation
\begin{equation}
    \left\{\begin{aligned}
    -\Delta u + \tau(u) & = f \quad \text{in }\Omega\, ,\\
    u &= g \quad \text{on }\partial\Omega\, ,
    \end{aligned}
    \right.
\end{equation}
with $\tau(u) = u^3$. Here $\Omega = [0,1]^2$. We set 
\[ u(\vx) = \sum_{k=1}^{600} \frac{1}{k^6}\sin(k\pi x_1)\sin(k\pi x_2) \]
as the ground truth and use it to generate the boundary and right hand side data. We set the number of Gauss-Newton iterations to be 3. The initial guess for the iteration is a zero function. The lengthscale of the kernels is set to be $0.3$.

We first study the solution error and the CPU time regarding $\rho$. We choose the number of interior points to be $N_{\rm domain} = 40000$; or equivalently, the grid size $h = 0.005$. In the left side of Figure \ref{fig:NonLinElliptic regarding rho}, we observe that a larger $\rho$ leads to a smaller $L^2$ error of the solution. For the Mat\'ern kernel with $\nu = 5/2, 7/2$, we observe that such accuracy improvement saturates at $\rho = 2$ or $4$, while when $\nu = 9/2$. the accuracy keeps improving until $\rho=10$. This high accuracy for large $\nu$ is because the solution $u$ is fairly smooth. Using smoother kernels can lead to better approximation accuracy. On the other hand, smoother kernels usually need a larger $\rho$ to achieve the same level of approximation accuracy, as we have demonstrated in the left of Figure \ref{fig:K phi phi accuracy}.

In the right side of Figure \ref{fig:NonLinElliptic regarding rho}, we show the CPU time required to compute the solution for different kernels and $\rho$. A larger $\rho$ generally leads to a longer CPU time. But there are some exceptions: for the Mat\'ern kernel with $\nu = 7/2, 9/2$, the CPU time for $\rho=3$ is shorter than that for $\rho=2$. Again, the reason is that these smoother kernels often require a larger $\rho$ for accurate approximations. When $\rho$ is very small, although the sparse Cholesky factorization is very fast, the pCG iterations could take long since the preconditioner matrix does not approximate the matrix well. 

\begin{figure}[ht]
    \centering
    \includegraphics[width=6cm]{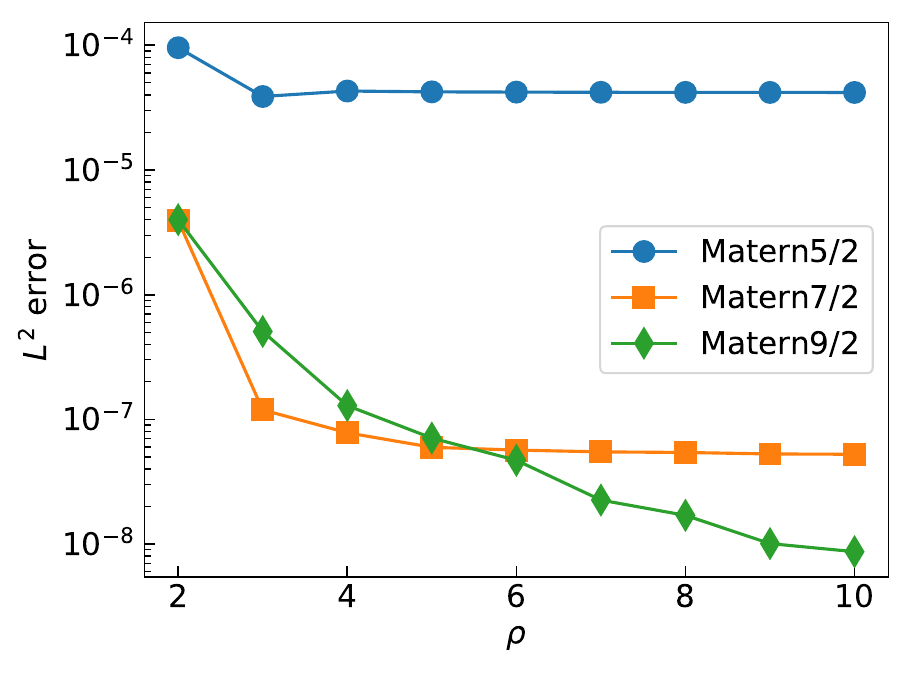}
    \includegraphics[width=6cm]{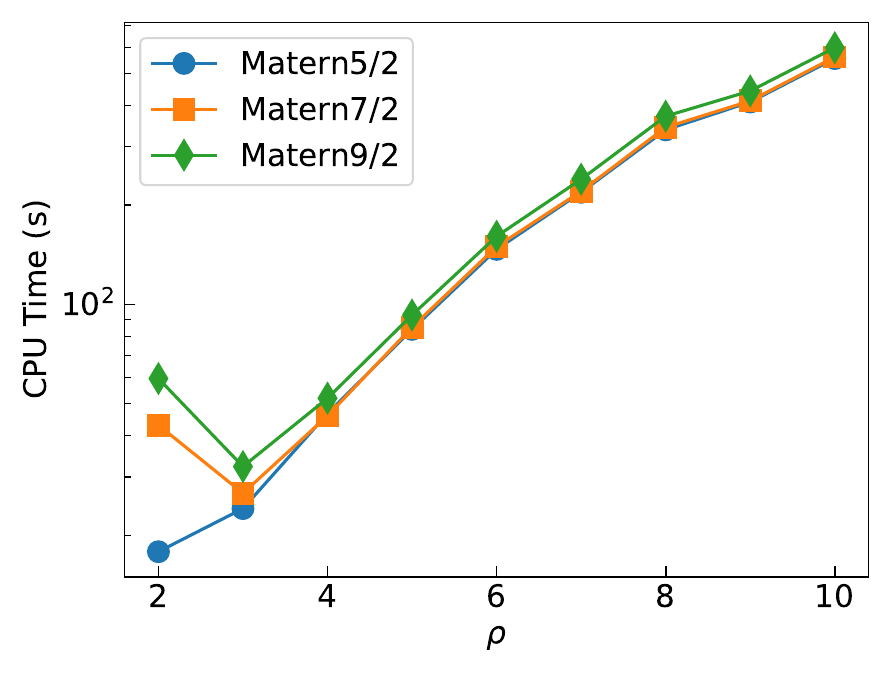}
    \caption{Nonlinear elliptic PDE example. The left figure concerns the $L^2$ errors of the solution, while the right figure concerns the CPU time. Both plots are with regard to $\rho$. We set $N_{\rm domain} = 40000$.}
    \label{fig:NonLinElliptic regarding rho}
\end{figure}

We then study the $L^2$ errors and CPU time regarding the number of physical points. We fix $\rho = 4.0$. In the left of Figure \ref{fig:NonLinElliptic accuray and time}, we observe that the accuracy improves when $N_{\rm domain}$ increases. For the smoother Mat\'ern kernels with $\nu=7/2, 9/2$, they will hit an accuracy floor of $10^{-7}$. This is because we only have a finite number of Gauss-Newton steps and a finite $\rho$. In the right of Figure \ref{fig:NonLinElliptic accuray and time}, a near-linear complexity in time regarding the number of points is demonstrated.
\begin{figure}[ht]
    \centering
    \includegraphics[width=6cm]{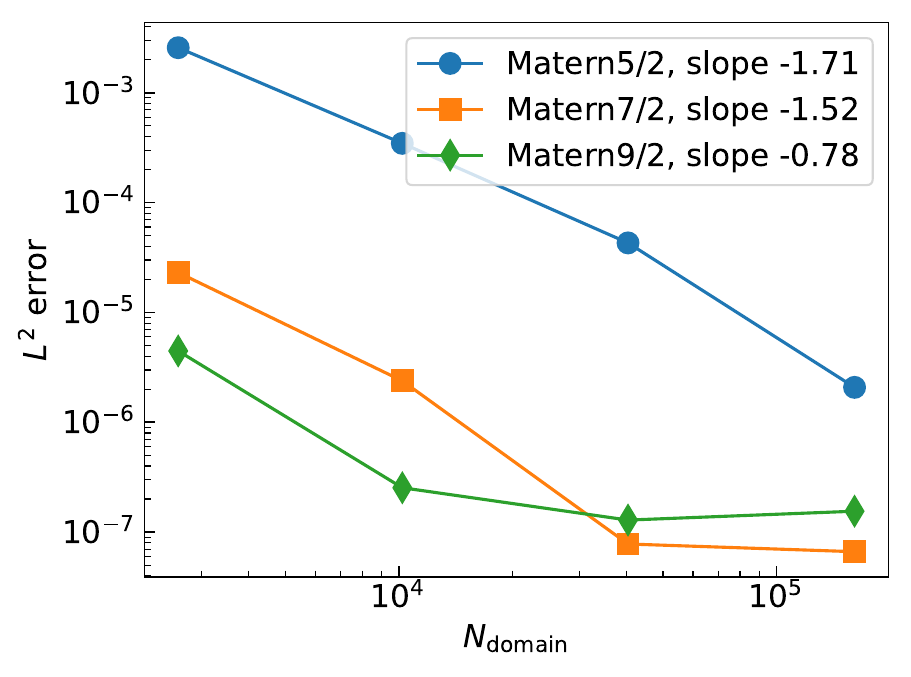}
    \includegraphics[width=6cm]{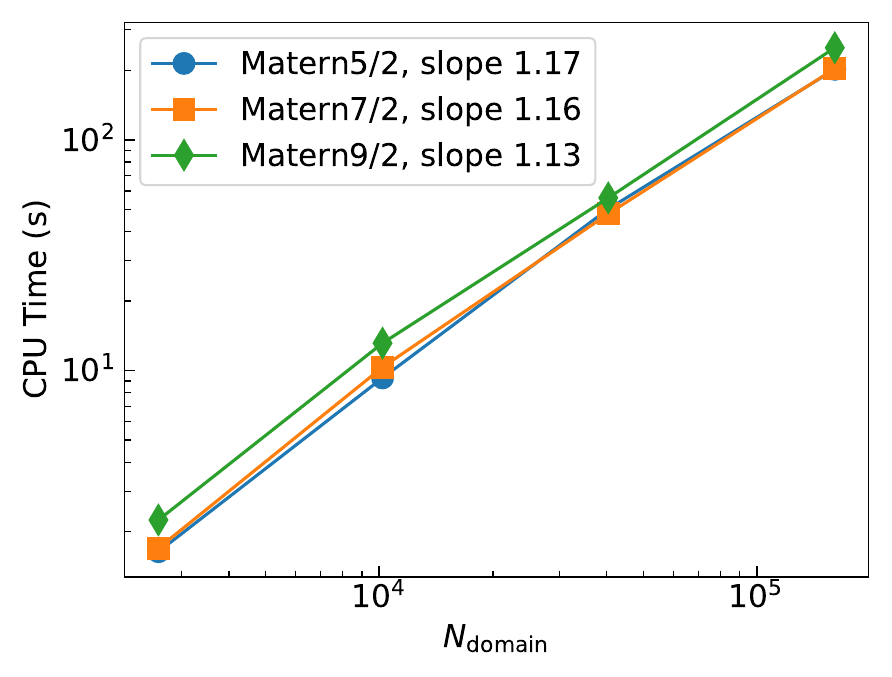}
    \caption{Nonlinear elliptic PDE example. The left figure concerns the $L^2$ errors of the solution, while the right figure concerns the CPU time. Both plots are with regard to the number of physical points in the domain. We set $\rho =4.0$.}
    \label{fig:NonLinElliptic accuray and time}
\end{figure}

\subsection{Burgers' equation}
Our second example concerns the time-dependent Burgers equation:
\begin{equation}
    \label{Burgers-proto-PDE}
    \begin{aligned}
      \partial_t u +  u \partial_x u  - 0.001  \partial_x^2 u  &= 0, \quad \forall
      (x,t) \in (-1, 1)  \times (0,1]\, , \\
      u(x, 0) & = - \sin( \pi x)\, , \\
      u(-1, t) & = u(1, t)  = 0\, .
  \end{aligned}
  \end{equation}
Rather than using a spatial-temporal GP as in \cite{chen2021solving}, we first discretize the equation in time and then use a spatial GP to solve the resulting PDE in space. This reduces the dimensionality of the system and is more efficient. More precisely, we use the Crank–Nicolson scheme with time stepsize $\Delta t$ to obtain
\begin{equation}
\label{Burgers Crank Nicolson}
\begin{aligned}
    \frac{\hat{u}(x,t_{n+1}) - \hat{u}(x,t_n)}{\Delta t} + \frac{1}{2}&\left(\hat{u}(x,t_{n+1})\partial_x\hat{u}(x,t_{n+1}) + \hat{u}(x,t_{n})\partial_x\hat{u}(x,t_{n})\right) \\
    & = \frac{0.001}{2}\left(\partial_x^2\hat{u}(x,t_{n+1})+\partial_x^2\hat{u}(x,t_{n})\right)\, ,
\end{aligned}
\end{equation}
where $\hat{u}(t_n, x)$ is an approximation of the true solution $u(t_n,x)$ with $t_n = n\Delta t$. When $\hat{u}(\cdot, t_n)$ is known, \eqref{Burgers Crank Nicolson} is a spatial PDE for the function $\hat{u}(\cdot, t_{n+1})$. We can solve \eqref{Burgers Crank Nicolson} iteratively starting from $n=0$. We use two steps of Gauss-Newton iterations with the initial guess as the solution at the last time step.

We set $\Delta t = 0.02$ and compute the solution at $t = 1$. The lengthscale of the kernels is chosen to be $0.02$. We set $\rho =4.0$ in the factorization algorithm. In the left of Figure \ref{fig:burgers solution and time}, we show our numerical solution by using a grid of size $h=0.001$ and the true solution computed by using the Cole-Hopf transformation. We see that they match very well, and the shock is captured. This is possible because we use a grid of small size so that the shock is well resolved. With a very small grid size, we need to deal with many large-size dense kernel matrices, and we use the sparse Cholesky factorization algorithm to handle such a challenge.

In the right of Figure \ref{fig:burgers solution and time}, we show the CPU time of our algorithm regarding different $N_{\rm domain}$. We clearly observe a near-linear complexity in time. The total CPU time is less than $10$ seconds to handle $50$ dense kernel matrices (since $1/\Delta t = 50$) of size larger than $10^4$ (the dimension of $K(\bphi,\bphi)$ is around $3\times N_{\rm domain}$ since we have three types of measurements) sequentially.
\begin{figure}[ht]
    \centering
    \includegraphics[width=5.5cm]{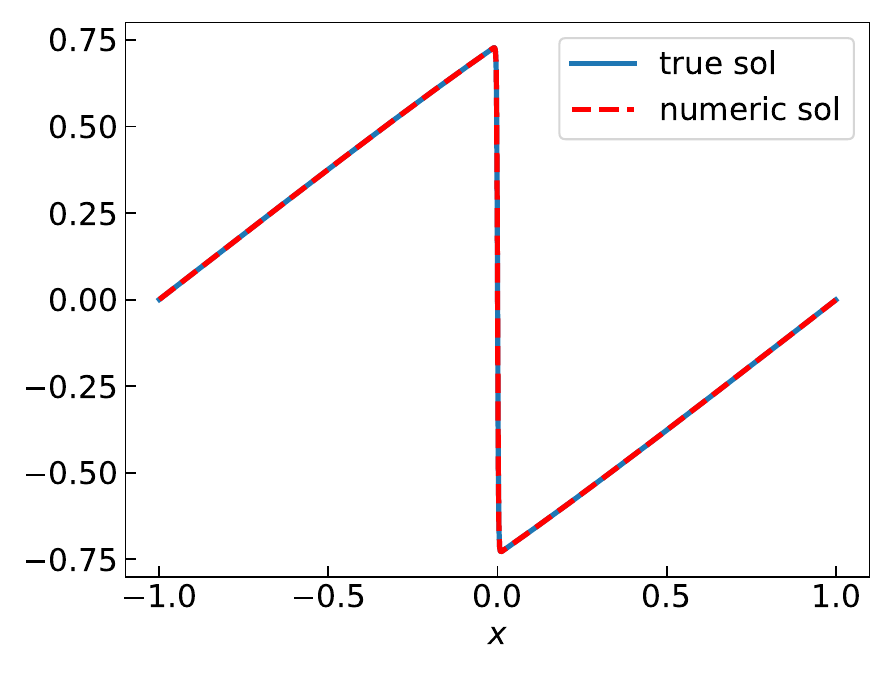}
    \includegraphics[width=6cm]{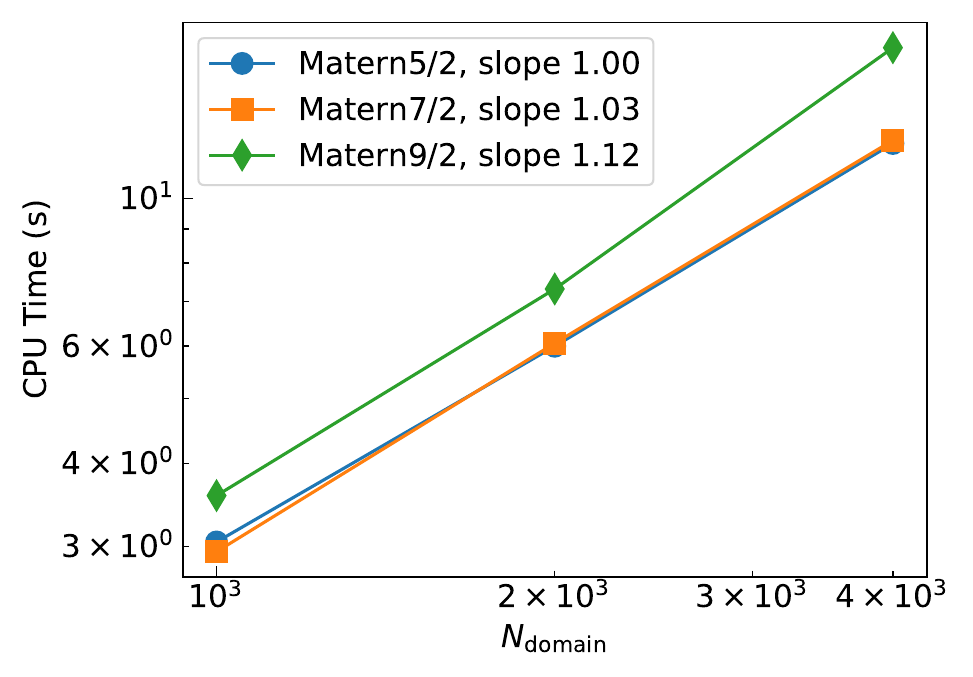}
    \caption{Burgers' equation example. The left figure is a demonstration of the numerical solution and true solution at $t=1$. The right figure concerns the CPU time regarding the number of physical points. We set $\rho = 4.0$.}
    \label{fig:burgers solution and time}
\end{figure}

We also show the accuracy of our solutions in the following Table \ref{table: burgers accuracy}. 
\begin{table}[ht]
\centering
\begin{tabular}{llll}
\hline
$N_{\rm domain}$ & $1000$   & $2000$   & $4000$   \\ \hline
$L^2$ error      & 1.729e-4 & 6.111e-5 & 7.453e-5 \\
$L^{\infty}$ error & 1.075e-3 & 2.745e-4 & 1.075e-4 \\ \hline
\end{tabular}
\caption{Burgers' equation example. The $L^2$ and $L^\infty$ errors of the computed solution at $t=1$. We use the Mat\'ern kernel with $\nu = 7/2$. The sparsity parameter $\rho = 4.0$.}
\label{table: burgers accuracy}
\end{table}
We observe high accuracy, $O(10^{-5})$ in the $L^2$ norm and $O(10^{-4})$ in the $L^\infty$ norm. The $L^2$ errors do not decrease when we increase the number of points from $2000$ to $4000$. It is because we use a fixed time stepsize $\Delta t= 0.02$ and a fixed $\rho = 4.0$.
\subsection{Monge-Amp\`ere equation} Our last example is the Monge-Amp\`ere equation in two dimensional space.
\begin{equation}
    \operatorname{det}(D^2 u) = f, \quad \vx \in (0,1)^2\, .
\end{equation}
Here, we choose $u(\vx) = \exp(0.5((x_1-0.5)^2+(x_2-0,5)^2))$ to generate the boundary and right hand side data. To ensure uniqueness of the solution, some convexity assumption is usually needed. Here, to test the wide applicability of our methodology, we directly implement Algorithm \ref{alg:Sparse Cholesky accelerated Gauss-Newton}. We adopt $3$ steps of Gauss-Newton iterations with the initial guess $u(\vx) = \frac{1}{2}\|\vx\|^2$. We choose the Mat\'ern kernel with $\nu = 5/2$. The lengthscale of the kernel is set to be $0.3$.

\begin{figure}[ht]
    \centering
    \includegraphics[width=6cm]{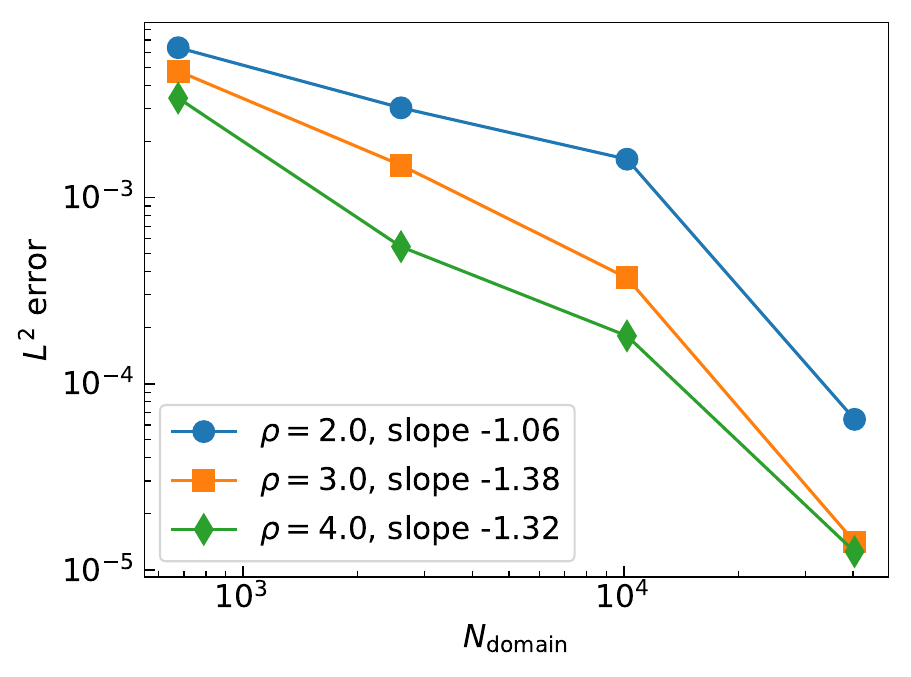}
    \includegraphics[width=6cm]{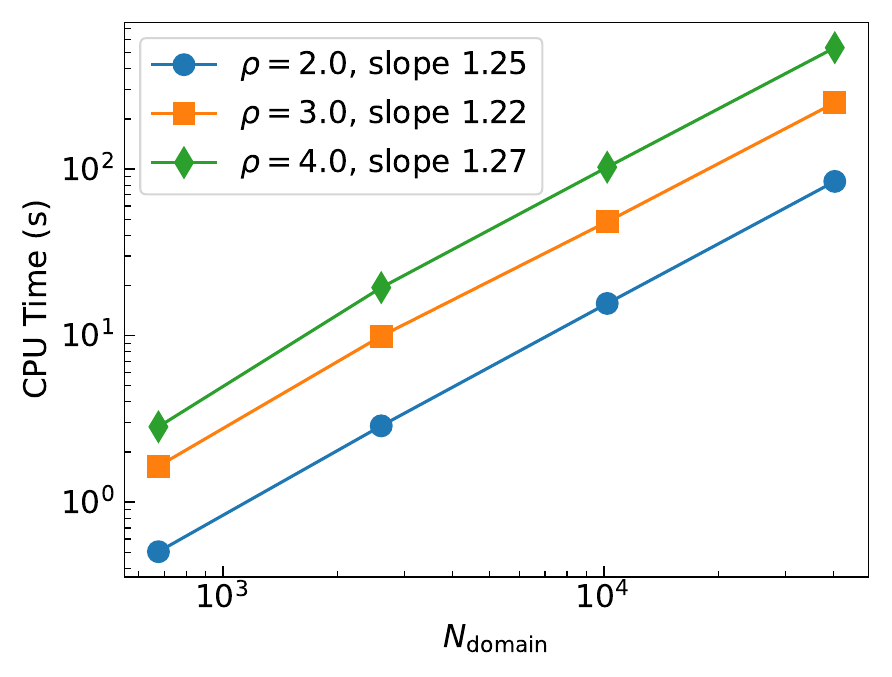}
    \caption{The Monge-Amp\`ere equation example. The left figure concerns the $L^2$ errors, while the right figure concerns the CPU time. Both are with respect to the number of physical points in space, and in both figures, we consider $\rho = 2.0,3.0,4.0$. We choose the Mat\'ern kernel with $\nu = 5/2$ in this example.}
    \label{fig:MongeAmpere accuray and time}
\end{figure}

In Figure \ref{fig:MongeAmpere accuray and time}, we present the $L^2$ errors of the solution and the CPU time with respect to $N_{\rm domain}$. Once again, we observe a nearly linear complexity in time. However, since $\det (D^2 u)$ involves several partial derivatives of the function, we need to differentiate our kernels accordingly; we use auto-differentiation in Julia for convenience, which is slightly slower than the hand-coded derivatives used in our previous numerical examples. Consequently, the total CPU time is longer compared to the earlier examples, although the scaling regarding $N_{\rm domain}$ remains similar.

As $N_{\rm domain}$ increases, the $L^2$ solution errors decrease for $\rho = 2.0,3.0,4.0$. This indicates that our kernel method is convergent for such a fully nonlinear PDE. However, since we do not incorporate singularity into the solution, this example may not correspond to the most challenging setting. Nonetheless, the success of this simple methodology combined with systematic fast solvers demonstrates its potential for promising automation and broader applications in solving nonlinear PDEs.
\section{Conclusions}
\label{sec: conclusions}
In this paper, we have investigated a sparse Cholesky factorization algorithm that enables scaling up the GP method for solving nonlinear PDEs. Our algorithm relies on a novel ordering of the Diracs and derivative-type measurements that arise in the GP-PDE methodology. With this ordering, the Cholesky factor of the inverse kernel matrix becomes approximately sparse, and we can use efficient KL minimization, equivalent to Vecchia approximation, to compute the sparse factors. We have provided rigorous analysis of the approximation accuracy by showing the exponential decay of the conditional covariance of GPs and the Cholesky factors of the inverse kernel matrix, for a wide class of kernel functions and derivative measurements.

When using second-order Gauss-Newton methods to solve the nonlinear PDEs, a reduced kernel matrix arises, in which many interior Dirac measurements are absent. In such cases, the decay is weakened, and the accuracy of the factorization deteriorates. To compensate for this loss of accuracy, we use pCG iterations with this approximate factor as a preconditioner. In our numerical experiments, our algorithm achieves high accuracy, and the computation time scales near-linearly with the number of points. This justifies the potential of GPs for solving general PDEs with automation, efficiency, and accuracy. We anticipate extending our algorithms to solving inverse problems and our theories to more kernel functions and measurement functionals in the future.

\blue{Due to the power of $d$ in the complexity, the algorithm in this paper is best suited for low-dimensional PDE problems. However, we note that for high-dimensional problems, if the points lie on a low-dimensional manifold of dimension $\tilde{d}$, then the complexity of our algorithm will depend on $\tilde{d}$ rather than $d$. This observation has been discussed in \cite{schafer2021sparse} for the case where the measurements are Diracs and will also apply to our paper, where derivative measurements induced by PDEs are considered. In truly high-dimensional problems, low-rank approximations should be pursued rather than the more ambitious full-scale approximations achieved by sparse Cholesky factorization in this paper. Potential approaches could include randomly pivoted Cholesky \cite{chen2022randomly}, which also involves selecting a specific ordering and $P_{\rm perm}$ (based on adaptive randomized ordering rather than the physically motivated coarse-to-fine ordering) but without further sparsifying these Cholesky factors. It is of interest to explore their connections to refine the low rank and sparse approximations for high dimensional scientific computing.}



\section*{Acknowledgements}
YC and HO acknowledge  support from the Air Force Office of Scientific Research under MURI award number FA9550-20-1-0358 (Machine Learning and Physics-Based Modeling and Simulation). YC is also partly supported by NSF Grants DMS-2205590.
HO also acknowledges  support from the Department of Energy under award number DE-SC0023163 (SEA-CROGS: Scalable, Efficient and Accelerated Causal Reasoning Operators, Graphs and Spikes for Earth and Embedded Systems).
FS acknowledges support from the Office of Naval Research under award number N00014-23-1-2545 (Untangling Computation). We thank Xianjin Yang for helpful comments on an earlier version of this article.

\bibliographystyle{plain}

\begin{thebibliography}{10}

\bibitem{ambikasaran2013mathcal}
Sivaram Ambikasaran and Eric Darve.
\newblock An $ \mathcal{O} (n \log n)$ fast direct solver for partial
  hierarchically semi-separable matrices: With application to radial basis
  function interpolation.
\newblock {\em Journal of Scientific Computing}, 57:477--501, 2013.

\bibitem{ambikasaran2015fast}
Sivaram Ambikasaran, Daniel Foreman-Mackey, Leslie Greengard, David~W Hogg, and
  Michael O’Neil.
\newblock Fast direct methods for {Gaussian} processes.
\newblock {\em IEEE transactions on pattern analysis and machine intelligence},
  38(2):252--265, 2015.

\bibitem{batlle2023error}
Pau Batlle, Yifan Chen, Bamdad Hosseini, Houman Owhadi, and Andrew~M Stuart.
\newblock Error analysis of kernel/gp methods for nonlinear and parametric
  pdes.
\newblock {\em arXiv preprint arXiv:2305.04962}, 2023.

\bibitem{berlinet2011reproducing}
Alain Berlinet and Christine Thomas-Agnan.
\newblock {\em Reproducing kernel Hilbert spaces in probability and
  statistics}.
\newblock Springer Science \& Business Media, 2011.

\bibitem{beylkin1991fast}
Gregory Beylkin, Ronald Coifman, and Vladimir Rokhlin.
\newblock Fast wavelet transforms and numerical algorithms {I}.
\newblock {\em Communications on pure and applied mathematics}, 44(2):141--183,
  1991.

\bibitem{bhattacharya2021model}
Kaushik Bhattacharya, Bamdad Hosseini, Nikola~B Kovachki, and Andrew~M Stuart.
\newblock Model reduction and neural networks for parametric {PDEs}.
\newblock {\em The SMAI journal of computational mathematics}, 7:121--157,
  2021.

\bibitem{chen2022randomly}
Yifan Chen, Ethan~N Epperly, Joel~A Tropp, and Robert~J Webber.
\newblock Randomly pivoted {Cholesky}: Practical approximation of a kernel
  matrix with few entry evaluations.
\newblock {\em arXiv preprint arXiv:2207.06503}, 2022.

\bibitem{chen2021solving}
Yifan Chen, Bamdad Hosseini, Houman Owhadi, and Andrew~M Stuart.
\newblock Solving and learning nonlinear {PDEs} with {Gaussian} processes.
\newblock {\em Journal of Computational Physics}, 447:110668, 2021.

\bibitem{chen2020function}
Yifan Chen and Thomas~Y Hou.
\newblock Function approximation via the subsampled {Poincar{\'e}} inequality.
\newblock {\em Discrete and Continuous Dynamical Systems}, 41(1):169--199,
  2020.

\bibitem{chen2022multiscale}
Yifan Chen and Thomas~Y Hou.
\newblock Multiscale elliptic {PDE} upscaling and function approximation via
  subsampled data.
\newblock {\em Multiscale Modeling \& Simulation}, 20(1):188--219, 2022.

\bibitem{chen2021consistency}
Yifan Chen, Houman Owhadi, and Andrew Stuart.
\newblock Consistency of empirical {Bayes} and kernel flow for hierarchical
  parameter estimation.
\newblock {\em Mathematics of Computation}, 90(332):2527--2578, 2021.

\bibitem{cockayne2019bayesian}
Jon Cockayne, Chris~J Oates, Timothy~John Sullivan, and Mark Girolami.
\newblock Bayesian probabilistic numerical methods.
\newblock {\em SIAM review}, 61(4):756--789, 2019.

\bibitem{darcy2023one}
Matthieu Darcy, Boumediene Hamzi, Giulia Livieri, Houman Owhadi, and Peyman
  Tavallali.
\newblock One-shot learning of stochastic differential equations with data
  adapted kernels.
\newblock {\em Physica D: Nonlinear Phenomena}, 444:133583, 2023.

\bibitem{daw2022rethinking}
Arka Daw, Jie Bu, Sifan Wang, Paris Perdikaris, and Anuj Karpatne.
\newblock Rethinking the importance of sampling in physics-informed neural
  networks.
\newblock {\em arXiv preprint arXiv:2207.02338}, 2022.

\bibitem{de2021high}
Filip De~Roos, Alexandra Gessner, and Philipp Hennig.
\newblock High-dimensional {Gaussian} process inference with derivatives.
\newblock In {\em International Conference on Machine Learning}, pages
  2535--2545. PMLR, 2021.

\bibitem{eriksson2018scaling}
David Eriksson, Kun Dong, Eric Lee, David Bindel, and Andrew~G Wilson.
\newblock Scaling {Gaussian} process regression with derivatives.
\newblock {\em Advances in neural information processing systems}, 31, 2018.

\bibitem{furrer2006covariance}
Reinhard Furrer, Marc~G Genton, and Douglas Nychka.
\newblock Covariance tapering for interpolation of large spatial datasets.
\newblock {\em Journal of Computational and Graphical Statistics},
  15(3):502--523, 2006.

\bibitem{geoga2020scalable}
Christopher~J Geoga, Mihai Anitescu, and Michael~L Stein.
\newblock Scalable {G}aussian process computations using hierarchical matrices.
\newblock {\em Journal of Computational and Graphical Statistics},
  29(2):227--237, 2020.

\bibitem{gines1998lu}
D~Gines, G~Beylkin, and J~Dunn.
\newblock {LU} factorization of non-standard forms and direct multiresolution
  solvers.
\newblock {\em Applied and Computational Harmonic Analysis}, 5(2):156--201,
  1998.

\bibitem{grossmann2023can}
Tamara~G Grossmann, Urszula~Julia Komorowska, Jonas Latz, and Carola-Bibiane
  Sch{\"o}nlieb.
\newblock Can physics-informed neural networks beat the finite element method?
\newblock {\em arXiv preprint arXiv:2302.04107}, 2023.

\bibitem{gu2004strong}
Ming Gu and Luiza Miranian.
\newblock Strong rank revealing {Cholesky} factorization.
\newblock {\em Electronic Transactions on Numerical Analysis}, 17:76--92, 2004.

\bibitem{guinness2018permutation}
Joseph Guinness.
\newblock Permutation and grouping methods for sharpening {Gaussian} process
  approximations.
\newblock {\em Technometrics}, 60(4):415--429, 2018.

\bibitem{hackbusch1999sparse}
Wolfgang Hackbusch.
\newblock A sparse matrix arithmetic based on {H}-matrices. {Part I}:
  Introduction to {H}-matrices.
\newblock {\em Computing}, 62(2):89--108, 1999.

\bibitem{hackbusch2002data}
Wolfgang Hackbusch and Steffen B{\"o}rm.
\newblock Data-sparse approximation by adaptive {H} 2-matrices.
\newblock {\em Computing}, 69:1--35, 2002.

\bibitem{hackbusch2000sparse}
Wolfgang Hackbusch and Boris~N Khoromskij.
\newblock A sparse {H}-matrix arithmetic, part {II}: Application to
  multi-dimensional problems.
\newblock {\em Computing}, 64(1):21--47, 2000.

\bibitem{han2018solving}
Jiequn Han, Arnulf Jentzen, and Weinan E.
\newblock Solving high-dimensional partial differential equations using deep
  learning.
\newblock {\em Proceedings of the National Academy of Sciences},
  115(34):8505--8510, 2018.

\bibitem{hauck2023super}
Moritz Hauck and Daniel Peterseim.
\newblock Super-localization of elliptic multiscale problems.
\newblock {\em Mathematics of Computation}, 92(341):981--1003, 2023.

\bibitem{henning2013oversampling}
Patrick Henning and Daniel Peterseim.
\newblock Oversampling for the multiscale finite element method.
\newblock {\em Multiscale Modeling \& Simulation}, 11(4):1149--1175, 2013.

\bibitem{hou2017sparse}
Thomas~Y Hou and Pengchuan Zhang.
\newblock Sparse operator compression of higher-order elliptic operators with
  rough coefficients.
\newblock {\em Research in the Mathematical Sciences}, 4:1--49, 2017.

\bibitem{jacot2018neural}
Arthur Jacot, Franck Gabriel, and Cl{\'e}ment Hongler.
\newblock Neural tangent kernel: Convergence and generalization in neural
  networks.
\newblock {\em Advances in neural information processing systems}, 31, 2018.

\bibitem{karniadakis2021physics}
George~Em Karniadakis, Ioannis~G Kevrekidis, Lu~Lu, Paris Perdikaris, Sifan
  Wang, and Liu Yang.
\newblock Physics-informed machine learning.
\newblock {\em Nature Reviews Physics}, 3(6):422--440, 2021.

\bibitem{katzfuss2017multi}
Matthias Katzfuss.
\newblock A multi-resolution approximation for massive spatial datasets.
\newblock {\em Journal of the American Statistical Association},
  112(517):201--214, 2017.

\bibitem{katzfuss2020vecchia}
Matthias Katzfuss, Joseph Guinness, Wenlong Gong, and Daniel Zilber.
\newblock Vecchia approximations of {Gaussian-process} predictions.
\newblock {\em Journal of Agricultural, Biological and Environmental
  Statistics}, 25:383--414, 2020.

\bibitem{kornhuber2018analysis}
Ralf Kornhuber, Daniel Peterseim, and Harry Yserentant.
\newblock An analysis of a class of variational multiscale methods based on
  subspace decomposition.
\newblock {\em Mathematics of Computation}, 87(314):2765--2774, 2018.

\bibitem{krishnapriyan2021characterizing}
Aditi Krishnapriyan, Amir Gholami, Shandian Zhe, Robert Kirby, and Michael~W
  Mahoney.
\newblock Characterizing possible failure modes in physics-informed neural
  networks.
\newblock {\em Advances in Neural Information Processing Systems},
  34:26548--26560, 2021.

\bibitem{l2016hierarchical}
Kenneth L.~Ho and Lexing Ying.
\newblock Hierarchical interpolative factorization for elliptic operators:
  Integral equations.
\newblock {\em Communications on Pure and Applied Mathematics},
  69(7):1314--1353, 2016.

\bibitem{lee2017deep}
Jaehoon Lee, Yasaman Bahri, Roman Novak, Samuel~S Schoenholz, Jeffrey
  Pennington, and Jascha Sohl-Dickstein.
\newblock Deep neural networks as {Gaussian} processes.
\newblock {\em arXiv preprint arXiv:1711.00165}, 2017.

\bibitem{li2012new}
Shengguo Li, Ming Gu, Cinna~Julie Wu, and Jianlin Xia.
\newblock New efficient and robust {HSS} {Cholesky} factorization of {SPD}
  matrices.
\newblock {\em SIAM Journal on Matrix Analysis and Applications},
  33(3):886--904, 2012.

\bibitem{li2020fourier}
Zongyi Li, Nikola Kovachki, Kamyar Azizzadenesheli, Burigede Liu, Kaushik
  Bhattacharya, Andrew Stuart, and Anima Anandkumar.
\newblock Fourier neural operator for parametric partial differential
  equations.
\newblock {\em arXiv preprint arXiv:2010.08895}, 2020.

\bibitem{lindgren2011explicit}
Finn Lindgren, H{\aa}vard Rue, and Johan Lindstr{\"o}m.
\newblock An explicit link between {Gaussian} fields and {Gaussian} {Markov}
  random fields: the stochastic partial differential equation approach.
\newblock {\em Journal of the Royal Statistical Society: Series B (Statistical
  Methodology)}, 73(4):423--498, 2011.

\bibitem{litvinenko2019likelihood}
Alexander Litvinenko, Ying Sun, Marc~G Genton, and David~E Keyes.
\newblock Likelihood approximation with hierarchical matrices for large spatial
  datasets.
\newblock {\em Computational Statistics \& Data Analysis}, 137:115--132, 2019.

\bibitem{liu2020gaussian}
Haitao Liu, Yew-Soon Ong, Xiaobo Shen, and Jianfei Cai.
\newblock When gaussian process meets big data: A review of scalable {GPs}.
\newblock {\em IEEE transactions on neural networks and learning systems},
  31(11):4405--4423, 2020.

\bibitem{long2022kernel}
Da~Long, Nicole Mrvaljevic, Shandian Zhe, and Bamdad Hosseini.
\newblock A kernel approach for pde discovery and operator learning.
\newblock {\em arXiv preprint arXiv:2210.08140}, 2022.

\bibitem{lu2021learning}
Lu~Lu, Pengzhan Jin, Guofei Pang, Zhongqiang Zhang, and George~Em Karniadakis.
\newblock Learning nonlinear operators via {DeepONet} based on the universal
  approximation theorem of operators.
\newblock {\em Nature machine intelligence}, 3(3):218--229, 2021.

\bibitem{lu2002inverses}
Tzon-Tzer Lu and Sheng-Hua Shiou.
\newblock Inverses of 2$\times$ 2 block matrices.
\newblock {\em Computers \& Mathematics with Applications}, 43(1-2):119--129,
  2002.

\bibitem{maalqvist2014localization}
Axel M{\aa}lqvist and Daniel Peterseim.
\newblock Localization of elliptic multiscale problems.
\newblock {\em Mathematics of Computation}, 83(290):2583--2603, 2014.

\bibitem{meng2022sparse}
Rui Meng and Xianjin Yang.
\newblock Sparse {Gaussian} processes for solving nonlinear {PDEs}.
\newblock {\em arXiv preprint arXiv:2205.03760}, 2022.

\bibitem{minden2017fast}
Victor Minden, Anil Damle, Kenneth~L Ho, and Lexing Ying.
\newblock Fast spatial {G}aussian process maximum likelihood estimation via
  skeletonization factorizations.
\newblock {\em Multiscale Modeling \& Simulation}, 15(4):1584--1611, 2017.

\bibitem{minden2017recursive}
Victor Minden, Kenneth~L Ho, Anil Damle, and Lexing Ying.
\newblock A recursive skeletonization factorization based on strong
  admissibility.
\newblock {\em Multiscale Modeling \& Simulation}, 15(2):768--796, 2017.

\bibitem{murphy2012machine}
Kevin~P Murphy.
\newblock {\em Machine learning: a probabilistic perspective}.
\newblock MIT press, 2012.

\bibitem{musco2017recursive}
Cameron Musco and Christopher Musco.
\newblock Recursive sampling for the {Nystr\"om} method.
\newblock {\em Advances in neural information processing systems}, 30, 2017.

\bibitem{neal1996priors}
Radford~M Neal.
\newblock Priors for infinite networks.
\newblock {\em Bayesian learning for neural networks}, pages 29--53, 1996.

\bibitem{nelsen2021random}
Nicholas~H Nelsen and Andrew~M Stuart.
\newblock The random feature model for input-output maps between {Banach}
  spaces.
\newblock {\em SIAM Journal on Scientific Computing}, 43(5):A3212--A3243, 2021.

\bibitem{owhadi2015bayesian}
Houman Owhadi.
\newblock Bayesian numerical homogenization.
\newblock {\em Multiscale Modeling \& Simulation}, 13(3):812--828, 2015.

\bibitem{owhadi2017multigrid}
Houman Owhadi.
\newblock Multigrid with rough coefficients and multiresolution operator
  decomposition from hierarchical information games.
\newblock {\em Siam Review}, 59(1):99--149, 2017.

\bibitem{owhadi2019operator}
Houman Owhadi and Clint Scovel.
\newblock {\em Operator-Adapted Wavelets, Fast Solvers, and Numerical
  Homogenization: From a Game Theoretic Approach to Numerical Approximation and
  Algorithm Design}, volume~35.
\newblock Cambridge University Press, 2019.

\bibitem{owhadi2019kernel}
Houman Owhadi and Gene~Ryan Yoo.
\newblock Kernel flows: from learning kernels from data into the abyss.
\newblock {\em Journal of Computational Physics}, 389:22--47, 2019.

\bibitem{padidar2021scaling}
Misha Padidar, Xinran Zhu, Leo Huang, Jacob Gardner, and David Bindel.
\newblock Scaling {Gaussian} processes with derivative information using
  variational inference.
\newblock {\em Advances in Neural Information Processing Systems},
  34:6442--6453, 2021.

\bibitem{quinonero2005unifying}
Joaquin Quinonero-Candela and Carl~Edward Rasmussen.
\newblock A unifying view of sparse approximate {Gaussian} process regression.
\newblock {\em The Journal of Machine Learning Research}, 6:1939--1959, 2005.

\bibitem{rahimi2007random}
Ali Rahimi and Benjamin Recht.
\newblock Random features for large-scale kernel machines.
\newblock {\em Advances in neural information processing systems}, 20, 2007.

\bibitem{raissi2019physics}
Maziar Raissi, Paris Perdikaris, and George~E Karniadakis.
\newblock Physics-informed neural networks: A deep learning framework for
  solving forward and inverse problems involving nonlinear partial differential
  equations.
\newblock {\em Journal of Computational physics}, 378:686--707, 2019.

\bibitem{raissi2018numerical}
Maziar Raissi, Paris Perdikaris, and George~Em Karniadakis.
\newblock Numerical {Gaussian} processes for time-dependent and nonlinear
  partial differential equations.
\newblock {\em SIAM Journal on Scientific Computing}, 40(1):A172--A198, 2018.

\bibitem{roininen2011correlation}
Lassi Roininen, Markku~S Lehtinen, Sari Lasanen, Mikko Orisp{\"a}{\"a}, and
  Markku Markkanen.
\newblock Correlation priors.
\newblock {\em Inverse problems and imaging}, 5(1):167--184, 2011.

\bibitem{sang2012full}
Huiyan Sang and Jianhua~Z Huang.
\newblock A full scale approximation of covariance functions for large spatial
  data sets.
\newblock {\em Journal of the Royal Statistical Society: Series B (Statistical
  Methodology)}, 74(1):111--132, 2012.

\bibitem{sanz2022finite}
Daniel Sanz-Alonso and Ruiyi Yang.
\newblock Finite element representations of gaussian processes: Balancing
  numerical and statistical accuracy.
\newblock {\em SIAM/ASA Journal on Uncertainty Quantification},
  10(4):1323--1349, 2022.

\bibitem{sanz2022spde}
Daniel Sanz-Alonso and Ruiyi Yang.
\newblock The {SPDE} approach to {M}at{\'e}rn fields: Graph representations.
\newblock {\em Statistical Science}, 37(4):519--540, 2022.

\bibitem{schaback2006kernel}
Robert Schaback and Holger Wendland.
\newblock Kernel techniques: from machine learning to meshless methods.
\newblock {\em Acta numerica}, 15:543--639, 2006.

\bibitem{schafer2021sparse}
Florian Sch\"afer, Matthias Katzfuss, and Houman Owhadi.
\newblock Sparse {Cholesky} factorization by {Kullback--Leibler} minimization.
\newblock {\em SIAM Journal on Scientific Computing}, 43(3):A2019--A2046, 2021.

\bibitem{schafer2021compression}
Florian Sch\"afer, TJ~Sullivan, and Houman Owhadi.
\newblock Compression, inversion, and approximate {PCA} of dense kernel
  matrices at near-linear computational complexity.
\newblock {\em Multiscale Modeling \& Simulation}, 19(2):688--730, 2021.

\bibitem{scholkopf2002learning}
Bernhard Sch{\"o}lkopf, Alexander~J Smola, Francis Bach, et~al.
\newblock {\em Learning with kernels: support vector machines, regularization,
  optimization, and beyond}.
\newblock MIT press, 2002.

\bibitem{stein2002screening}
Michael~L Stein.
\newblock The screening effect in kriging.
\newblock {\em The Annals of Statistics}, 30(1):298--323, 2002.

\bibitem{stein20112010}
Michael~L Stein.
\newblock 2010 {R}ietz lecture: When does the screening effect hold?
\newblock {\em The Annals of Statistics}, 39(6):2795--2819, 2011.

\bibitem{vecchia1988estimation}
Aldo~V Vecchia.
\newblock Estimation and model identification for continuous spatial processes.
\newblock {\em Journal of the Royal Statistical Society: Series B
  (Methodological)}, 50(2):297--312, 1988.

\bibitem{wang2021understanding}
Sifan Wang, Yujun Teng, and Paris Perdikaris.
\newblock Understanding and mitigating gradient flow pathologies in
  physics-informed neural networks.
\newblock {\em SIAM Journal on Scientific Computing}, 43(5):A3055--A3081, 2021.

\bibitem{wang2022and}
Sifan Wang, Xinling Yu, and Paris Perdikaris.
\newblock When and why pinns fail to train: A neural tangent kernel
  perspective.
\newblock {\em Journal of Computational Physics}, 449:110768, 2022.

\bibitem{wendland2004scattered}
Holger Wendland.
\newblock {\em Scattered data approximation}, volume~17.
\newblock Cambridge university press, 2004.

\bibitem{williams2000using}
Christopher Williams and Matthias Seeger.
\newblock Using the {N}ystr{\"o}m method to speed up kernel machines.
\newblock {\em Advances in neural information processing systems}, 13, 2000.

\bibitem{williams2006gaussian}
Christopher~KI Williams and Carl~Edward Rasmussen.
\newblock {\em Gaussian processes for machine learning}, volume~2.
\newblock MIT press Cambridge, MA, 2006.

\bibitem{wilson2015kernel}
Andrew Wilson and Hannes Nickisch.
\newblock Kernel interpolation for scalable structured {Gaussian} processes
  {(KISS-GP)}.
\newblock In {\em International conference on machine learning}, pages
  1775--1784. PMLR, 2015.

\bibitem{wilson2016deep}
Andrew~Gordon Wilson, Zhiting Hu, Ruslan Salakhutdinov, and Eric~P Xing.
\newblock Deep kernel learning.
\newblock In {\em Artificial intelligence and statistics}, pages 370--378.
  PMLR, 2016.

\bibitem{wu2017bayesian}
Jian Wu, Matthias Poloczek, Andrew~G Wilson, and Peter Frazier.
\newblock Bayesian optimization with gradients.
\newblock {\em Advances in neural information processing systems}, 30, 2017.

\bibitem{yang2018sparse}
Ang Yang, Cheng Li, Santu Rana, Sunil Gupta, and Svetha Venkatesh.
\newblock Sparse approximation for {Gaussian} process with derivative
  observations.
\newblock In {\em AI 2018: Advances in Artificial Intelligence: 31st
  Australasian Joint Conference, Wellington, New Zealand, December 11-14, 2018,
  Proceedings}, pages 507--518. Springer, 2018.

\bibitem{zeng2023competitive}
Qi~Zeng, Yash Kothari, Spencer~H Bryngelson, and Florian~Tobias Schaefer.
\newblock Competitive physics informed networks.
\newblock In {\em The Eleventh International Conference on Learning
  Representations}, 2023.

\bibitem{zhang2000meshless}
Xiong Zhang, Kang~Zhu Song, Ming~Wan Lu, and X~Liu.
\newblock Meshless methods based on collocation with radial basis functions.
\newblock {\em Computational mechanics}, 26:333--343, 2000.

\end{thebibliography}

\newpage
\appendix
\section{Supernodes and aggregate sparsity pattern}
\label{appendix: Supernodes and aggregate sparsity pattern}
The supernode idea is adopted from \cite{schafer2021sparse}, which allows to \textit{re-use} the Cholesky factors in the computation of \eqref{eqn: explicit formula for KL min} to update multiple columns at once.

We group the measurements into supernodes consisting of measurements whose points are \textit{close in location and have similar lengthscale parameters $l_i$}. To do this, we select the last index $j \in I$ of the measurements in the ordering that has not been aggregated into a supernode yet and aggregate the indices in $\{i : (i,j) \in S_{P,l,\rho}, l_i \leq \lambda l_j\}$ that have not been aggregated yet into a common supernode, for some $\lambda > 1$. We repeat this procedure until every measurement has been aggregated into a supernode. We denote the set of all supernodes as $\tilde{I}$ and write $i \leadsto \tilde{i}$ for $i \in I$ and $\tilde{i} \in \tilde{I}$ if $\tilde{i}$ is the supernode to which $i$ has been aggregated.

The idea is to assign the same sparsity pattern to all the measurements of the same supernode. To achieve so, we define the sparsity set for a supernode as the union of the sparsity sets of all the nodes it contains, namely $s_{\tilde{i}} := \{j : \exists i \leadsto  \tilde{i}, j \in s_i\}$. Then, we introduce the aggregated sparsity pattern
\[S_{P,l,\rho, \lambda} : = \bigcup_{\tilde{j}} \bigcup_{j \leadsto \tilde{j}} \{(i,j) \subset I\times I: i\leq j, i \in s_{\tilde{j}}\}\, .  \]
Under mild assumptions (Theorem B.5 in \cite{schafer2021sparse}), one can show that there are $O(N/\rho^d)$ number of supernodes and each supernode contains $O(\rho^d)$ measurements. The size of the sparsity set for a supernode $\# s_{\tilde{j}} = O(\rho^d)$. For a visual demonstration of the grouping and aggregate sparsity pattern, see Figure \ref{fig:aggregation}, which is taken from Figure 3 in \cite{schafer2021sparse}.
\begin{figure}[ht]
	\centering
		\begin{tikzpicture}[scale=0.7]
			\input{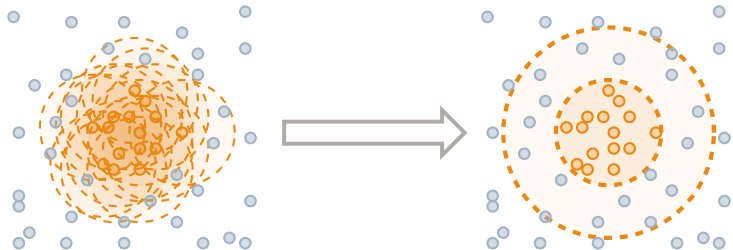}
		\end{tikzpicture}
    \caption{
    The figure on the left illustrates the original pattern $S_{P,\ell,\rho}$. For each orange point $j$, its sparsity pattern $s_j$ includes all points within a circle with a radius of $\rho$. On the right, all points $j$ that are located close to each other and have similar lengthscales are grouped into a supernode $\tilde{j}$. The supernode can be represented by a list of \emph{parents} (the orange points within an inner sphere of radius $\approx \rho$, or all $j \leadsto \tilde{j}$) and \emph{children} (all points within a radius $\leq 2\rho$, which correspond to the sparsity set $s_{\tilde{j}}$).  Figure reproduced from \cite{schafer2021sparse} with author permission.}
    \label{fig:aggregation}
\end{figure}

Now, we can compute \eqref{eqn: explicit formula for KL min} with the aggregated sparsity pattern more efficiently. Let
$s_j^* = \{i: (i,j) \in S_{P,l,\rho,\lambda}\}$ be the individual sparsity pattern for $j$ in the aggregated pattern $S_{P,l,\rho, \lambda}$. In \eqref{eqn: explicit formula for KL min}, we need to compute matrix-vector products for $\Theta_{s_j^*,s_j^*}^{-1}$. For that purpose, one can apply the Cholesky factorization to $\Theta_{s_j^*,s_j^*}$. Na\"ively computing Cholesky factorizations of every $\Theta_{s_j^*,s_j^*}$
will result in $O(N\rho^{3d})$ arithmetic complexity. However, due to the supernode construction, we can factorize $\Theta_{s_{\tilde{j}},s_{\tilde{j}}}$ once and then use the resulting factor to obtain the Cholesky factors of $\Theta_{s_j^*,s_j^*}$ directly for all $j \leadsto \tilde{j}$. This is because our construction guarantees that \[\Theta_{s_j^*,s_j^*} = \Theta_{s_{\tilde{j}},s_{\tilde{j}}} [1:\#s_j^*,1:\#s_j^*]\, ,\] where we used the MATLAB notation.   The above relation shows that sub-Cholesky factors of $\Theta_{s_{\tilde{j}},s_{\tilde{j}}}$ become the Cholesky factors of $\Theta_{s_j^*,s_j^*}$ for $j \leadsto \tilde{j}$.

\textit{Therefore, one step of Cholesky factorization works for all $O(\rho^d)$ measurements in the supernode}. In total, the arithmetic complexity is upper bounded by $O(\rho^{3d}\times N/\rho^d) = O(N\rho^{2d})$. For more details of the algorithm, we refer to section 3.2, in particular Algorithm 3.2, in \cite{schafer2021sparse}.

It was shown that the aggregate sparsity pattern could be constructed with time complexity $O(N\log(N)+N\rho^d)$ and space complexity $O(N)$; see Theorem C.3 in \cite{schafer2021sparse}. They are of a lower order compared to the time complexity $O(N\log^2(N)\rho^d)$ and space complexity $O(N\rho^d)$ for generating the maximin ordering and the original sparsity pattern $S_{P,l,\rho}$ (see Remark \ref{rmk: complexity maximin ordering sparsity patterns}).
\section{Ball-packing arguments} The ball-packing argument is useful to bound the cardinality of the sparsity pattern.
\label{appendix: Ball-packing arguments}
\begin{proposition}
    Consider the maximin ordering (Definition \ref{def: maximin ordering}) and the sparsity pattern defined in \eqref{eqn: def sparsity pattern}. For each column $j$, denote $s_j = \{i: (i,j) \in S_{P,l,\rho}\}$. The cardinality of the set $s_j$ is denoted by $\# s_j$. Then, it holds that $\# s_j = O(\rho^d)$.
\end{proposition}
\begin{proof}
    Fix a $j$. For any $i \in s_j$, we have $\operatorname{dist}(\vx_{P(i)},\vx_{P(j)}) \leq \rho l_j$. Moreover, by the definition of the maximin ordering, we know that for $i, i' \in s_j$ and $i \neq i'$, it holds that $\operatorname{dist}(\vx_{P(i)},\vx_{P(i')}) \geq l_j$. Thus, the cardinality of $s_i$ is bounded by the number of disjoint balls of radius $l_j$ that the ball $B(\vx_{P(i)}, 2\rho l_j)$ can contain. Clearly, $\#s_i = O(\rho^d)$. The proof is complete.
\end{proof}
\section{Explicit formula for the KL minimization} By direct calculation, one can show the KL minimization attains an explicit formula. The proof of this explicit formula follows a similar approach to that of Theorem 2.1 in \cite{schafer2021sparse}, with the only difference being the use of upper Cholesky factors.
\label{appendix: Explicit formula for the KL minimization}
\begin{proposition}
    The solution to \eqref{eqn: KL min} is given by \eqref{eqn: explicit formula for KL min}.
\end{proposition}
\begin{proof}
    We use the explicit formula for the KL divergence between two multivariate Gaussians:
    \begin{equation}
    \label{eqn: KL between two Gaussians}
    \begin{aligned}
        &\operatorname{KL}\left(\cN(0,\Theta)\parallel\cN(0,(UU^T)^{-1})\right)\\
        =&\frac{1}{2}[-\log\det (U^T\Theta U) + \operatorname{tr}(U^T\Theta U) - N]\, .
    \end{aligned}
    \end{equation}
    \blue{To identify the minimizer, the constant and scaling do not matter, so we focus on the $-\log\det(U^T\Theta U)$ and $\operatorname{tr}(U^T\Theta U)$ parts. By writing $U = [U_{:,1}, ..., U_{:,M}]$, we get
    \begin{equation}
        -\log\det (U^T\Theta U) + \operatorname{tr}(U^T\Theta U) = \sum_{j=1}^M [-2\log U_{jj} + \operatorname{tr}(U_{:,j}^T\Theta U_{:,j})] -\log \det \Theta\, .
    \end{equation}
    Thus, the minimization is decoupled into each column of $U$. Due to the choice of the sparsity pattern, we have $U_{:,j}^T\Theta U_{:,j} = U_{s_j,j}^T\Theta_{s_j,s_j}U_{s_j,j}$. We can further simplify the formula
    \begin{equation}
         \operatorname{tr}(U_{:,j}^T\Theta U_{:,j}) =  \operatorname{tr}(U_{s_j,j}^T\Theta_{s_j,s_j} U_{s_j,j}) \, .
    \end{equation}
    It suffices to identify the minimizer of $- 2\log U_{jj} + \operatorname{tr}(U_{s_j,j}^T\Theta_{s_j,s_j} U_{s_j,j})$. Taking the derivatives, we get the optimality condition:
    \begin{equation}
        -\frac{2}{U_{jj}}\textbf{e}_{\# s_j}  + 2\Theta_{s_j,s_j} U_{s_j,j} = 0\, ,
    \end{equation}
    where $\textbf{e}_{\# s_j}$ is a standard basis vector in $\mathbb{R}^{\#s_j}$ with the last entry being $1$ and other entries equal $0$.}
    
    Solving this equation leads to the solution
    \begin{equation}
    U_{s_j,j} = \frac{\Theta_{s_j,s_j}^{-1}\textbf{e}_{\# s_j}}{\sqrt{\textbf{e}_{\# s_j}^T\Theta_{s_j,s_j}^{-1}\textbf{e}_{\# s_j}}}\, .
\end{equation}
The proof is complete.
\end{proof}

\section{Proofs of the main theoretical results}
\subsection{Proof of Theorem \ref{thm: cholesky factor decay}}
\label{appendix: Proof of thm: cholesky factor decay}
\begin{proof}[Proof of Theorem \ref{thm: cholesky factor decay}]
We rely on the interplay between GP regression, linear algebra, and numerical homogenization to prove this theorem. Consider the GP $\xi \sim \cN(0,\cL^{-1})$. For each measurement functional, we define the Gaussian random variables $Y_i = [\xi, \tilde{\phi}_i] \sim \cN(0, [\tilde{\phi}_i, \cL^{-1}\tilde{\phi}_i]) = \cN(0, \Theta_{ii})$. As mentioned in Remark \ref{remark: why coarse-to-fine ordering} and proved in Proposition \ref{prop: connect Cholesky factor to cond var}, we have a relation between the Cholesky factor and the GP conditioning, as 
\[\frac{U^\star_{ij}}{U^\star_{jj}} = (-1)^{i\neq j} \frac{\mathrm{Cov}[Y_i,Y_j|Y_{1:j-1\backslash \{i\}}]}{\mathrm{Var}[Y_i|Y_{1:j-1\backslash \{i\}}]}, i \leq j \, . \]
Moreover, by Proposition \ref{prop: connect cond var to cond expect}, one can connect the conditional covariance of GPs with conditional expectation, such that
\begin{equation}
\label{eqn: conditional covariance to expectation}
    \frac{\mathrm{Cov}[Y_i,Y_j|Y_{1:j-1\backslash \{i\}}]}{\mathrm{Var}[Y_i|Y_{1:j-1\backslash \{i\}}]} = \bE[Y_j|Y_i = 1, Y_{1:j-1\backslash \{i\}} = 0]\, .
\end{equation}
The above conditional expectation is related to the \textit{Gamblets} introduced in the numerical homogenization literature. Indeed, using the relation $Y_i = [\xi, \tilde{\phi}_i]$, we have 
\begin{equation}
\label{eqn: conditional expectation to gamblets}
    \bE[Y_j|Y_i = 1, Y_{1:j-1\backslash \{i\}} = 0] = [\bE[\xi|Y_i = 1, Y_{1:j-1\backslash \{i\}} = 0],\tilde{\phi}_j]  = [\psi_j^i, \tilde{\phi}_j]\, ,
\end{equation}
where $\psi_j^i(\vx) := \bE[\xi(\vx) |Y_i = 1, Y_{1:j-1\backslash \{i\}} = 0]$; it is named Gamblets in \cite{owhadi2017multigrid,owhadi2019operator}. 

Importantly, for the conditional expectation, by Proposition \ref{prop: connect cond expect to Gamblets}, we have the following variational characterization \cite{owhadi2017multigrid,owhadi2019operator}:
 \begin{equation}
    \label{eqn: optimization def basis}
    \begin{aligned}
    \psi_{j}^i = \text{argmin}_{\psi \in H_0^s(\Omega)}\quad  &[\psi, \cL \psi]\\
     \text{subject to}\quad &[\psi, \tilde{\phi}_k] = \delta_{i,k}\ \  \text{for}\ \  1 \leq k \leq j-1 \, .
 \end{aligned}
    \end{equation}

A main property of the Gamblets is that they can exhibit an exponential decay property under suitable assumptions, which can be used to prove our theorem. 
We collect the related theoretical results in Appendix \ref{sec: Regarding the exponential decay of Gamblets}; we will use them in our proof.

Our proof for the theorem consists of two steps. The first step is to bound $|{U^\star_{ij}}/{U^\star_{jj}}|$, and the second step is to bound $|U_{jj}^*|$.

For the first step, we separate the cases $j \leq M$ and $j > M$. Indeed, the case $j \leq M$ has been covered in \cite{schafer2021compression, schafer2021sparse}. Here, our proof is simplified. 

For $1\leq i \leq j \leq M$, by the discussions above, we have the relation:
\begin{equation}
    \left|\frac{U^\star_{ij}}{U^\star_{jj}}\right| = \left|[\psi_j^i, \tilde{\phi}_j]\right| = |\psi_j^i(\vx_{P(j)})|\, ,
\end{equation}
where we used the fact that $\tilde{\phi}_j = \updelta_{\vx_{P(j)}}$ because all the Diracs measurements are ordered first. To bound $|\psi_j^i(\vx_{P(j)})|$, we will use the exponential decay results for Gamblets that we prove in Proposition \ref{prop: exp decay derivative measurements}. More precisely, to apply Proposition \ref{prop: exp decay derivative measurements} to this context, we need to verify its assumptions, especially Assumption \ref{assumption for Gamblets: Domain and partition}. 

In our setting, we can construct a partition of the domain $\Omega$ by using the Voronoi diagram. Denote $X_{j-1} = \{\vx_{P(1)}, ... , \vx_{P(j-1)}\}$. We note that $X_M$ will consist of all the physical points. We define $\tau_k$, for $1\leq k \leq j-1$, to be the Voronoi cell, which contains all points in $\Omega$ that are closer to $\vx_{P(k)}$ than to any other in $X_{j-1}$. Since we assume $\Omega$ is convex, $\tau_k$ is also convex. And a bounded convex domain is itself uniformly Lipschitz. \blue{We can show that these Voronoi cells are also uniformly Lipschitz (i.e., the Lipschitz constants for each cell are bounded uniformly). To see this, note that these cells are convex polygons; the Lipschitz constant depends on the size of the angles at the corners. The angles cannot be arbitrarily small since each cell contains a ball of size $\delta h$ and is inside a ball of size $h/\delta$, as shown in the arguments below. Therefore, there is a uniform lower bound on the size of angles at each corner of these convex polygons, so they are uniformly Lipschitz.}

Furthermore, to verify the other parts in Assumption \ref{assumption for Gamblets: Domain and partition}, we analyze the homogeneity parameter of $X_{j-1}$. By definition,
\begin{equation}
    \delta(X_{j-1}; \partial \Omega) = \frac{\min_{\vx\neq\vy \in X_{j-1}} \mathrm{dist}(\vx,\{\vy\}\cup \partial\Omega)}{\max_{\vx\in\Omega} \mathrm{dist}(\vx, X_{j-1} \cup \partial \Omega)} \, .  
\end{equation}
\blue{Recall the definition of the maximin ordering: the maximin ordering conditioned on a set $\sfA = \partial \Omega$ for points $\{\vx_i, i\in I\}$ is obtained by successively selecting the point $\vx_i$ that is furthest away from $\sfA$ and the already picked points. This implies \begin{equation*}
    l_{j-1} = \mathrm{dist}(\vx_{P(j-1)}, \{\vx_{P(1)},...,\vx_{P(j-2)}\} \cup \partial\Omega)\, .
\end{equation*}} We thus have
\begin{equation}
\label{eqn: numerator larger than lj}
    \min_{\vx \neq \vy \in X_{j-1}} \mathrm{dist}(\vx,\{\vy\}\cup \partial\Omega) = l_{j-1}\, .
\end{equation}
Then, by the triangle inequality, it holds that
\begin{equation}
\label{eqn: D 7 proof}
\begin{aligned}
    \max_{\vx\in\Omega} \mathrm{dist}(\vx, X_{j-1} \cup \partial \Omega) &\leq \max_{\vx\in X_M} \mathrm{dist}(\vx, X_{j-1} \cup \partial \Omega) + \max_{\vx\in \Omega} \mathrm{dist}(\vx, X_M \cup \partial\Omega)\\
    & \leq l_{j} + l_M/ \delta(X_M; \partial \Omega) \leq (1+1/\delta(X_M; \partial \Omega))l_{j}\, ,
\end{aligned}
\end{equation}
where in the second inequality, we used the definition of the lengthscales $l_{j}$ and the homogeneity assumption of $X_M$, \blue{i.e., $\delta(X_M; \partial \Omega) > 0$, which has the following relation
\[\delta(X_{M}; \partial \Omega) = \frac{\min_{\vx\neq\vy \in X_{M}} \mathrm{dist}(\vx,\{\vy\}\cup \partial\Omega)}{\max_{\vx\in\Omega} \mathrm{dist}(\vx, X_{M} \cup \partial \Omega)}=\frac{l_M}{\max_{\vx\in \Omega} \mathrm{dist}(\vx, X_M \cup \partial\Omega)} \, .   \]
}
Combining the above two estimates, we get $\delta(X_{j-1}; \partial \Omega) \geq 1/(1+1/\delta(X_M; \partial \Omega)) > 0$ where we used the fact that $l_j \leq l_{j-1}$. So $X_{j-1}$ is also homogeneously distributed with $\delta(X_{j-1}; \partial \Omega)>0$. 



We are ready to verify Assumption \ref{assumption for Gamblets: Domain and partition}. Firstly, the balls $B(\vx, l_{j}/2), \vx \in X_{j-1}$ do not intersect, and thus inside each $\tau_k$, there is a ball of center $\vx_{P(k)}$ and radius $l_{j}/2$. Secondly, since $\max_{\vx\in\Omega} \mathrm{dist}(\vx, X_{j-1} \cup \partial \Omega) \leq (1+1/\delta(X_M; \partial \Omega))l_{j}$, we know that $\tau_i$ is contained in a ball of center $\vx_{P(k)}$ and radius $(1+1/\delta(X_M; \partial \Omega))l_{j}$. 
Therefore, Assumption \ref{assumption for Gamblets: Domain and partition} holds with $h = l_{j}/2$, $\delta = \min(1/(2+2/\delta(X_M; \partial \Omega)),1)$ and $Q=j-1$. Assumption \ref{assumption for Gamblets: Measurement functionals} readily holds by our choice of measurements.

Thus, applying Proposition \ref{prop: exp decay derivative measurements}, we then get 
\begin{equation}
    \left|\frac{U^\star_{ij}}{U^\star_{jj}}\right| = |\psi_j^i(\vx_{P(j)})| \leq C l_{j}^{-2s}\exp\left(-\frac{\mathrm{dist}(\vx_{P(i)}, \vx_{P(j)})}{Cl_{j}}\right)\, ,
\end{equation}
where $C$ is a constant depending on $\Omega, \delta, d, s, \|\cL\|, \|\cL^{-1}\|$. We obtain the exponential decay of $|{U^\star_{ij}}/{U^\star_{jj}}|$ for $j \leq M$. 

For $j > M$ and $i \leq j$, we have
\begin{equation}
\label{appendix-D-9}
    \left|\frac{U^\star_{ij}}{U^\star_{jj}}\right| = \left|[\psi_j^i, \tilde{\phi}_j]\right| \leq \max_{0\leq |\gamma| \leq J } \left|D^\gamma \psi_j^i(\vx_{P(j)})\right|\, ,
\end{equation}
where we used the fact that $\tilde{\phi}_j$ is of the form $\updelta_{\vx_{P(j)}} \circ D^{\gamma}$. Now, for $\psi_j^i$, the set of points we need to deal with is $X_M$. Similar to the case $j \leq M$, we define $\tau_k$, for $1\leq k \leq M$, to be the Voronoi cell of these points in $\Omega$. Then, using the same arguments in the previous case, we know that Assumption \ref{assumption for Gamblets: Domain and partition} holds with $h = l_M/2$, $\delta = \min(\delta(X_M; \partial \Omega)/2,1)$ and $Q=M$. Assumption \ref{assumption for Gamblets: Measurement functionals} readily holds by our choice of measurements. Therefore Proposition \ref{prop: exp decay derivative measurements} implies that
\begin{equation}
\label{appendix-eqn-D-10}
    \max_{0\leq |\gamma| \leq J } \left|D^\gamma \psi_j^i(\vx_{P(j)})\right| \leq C l_M^{-2s}\exp\left(-\frac{\mathrm{dist}(\vx_{P(i)}, \vx_{P(j)})}{Cl_M}\right)\, ,
\end{equation}
where $C$ is a constant depending on $\Omega, \delta, d, s, \|\cL\|, \|\cL^{-1}\|, J$.

Summarizing the above arguments, we have obtained that 
\begin{equation}
    \left|\frac{U^\star_{ij}}{U^\star_{jj}}\right|  \leq C l_j^{-2s}\exp\left(-\frac{\mathrm{dist}(\vx_{P(i)}, \vx_{P(j)})}{Cl_j}\right)\, ,
\end{equation}
for any $1\leq i \leq j \leq N$, noting that by our definition $l_j = l_M$ for $j > M$.

Now, we analyze $|U_{jj}^*|$. Note that $\Theta = K(\tilde{\bphi},\tilde{\bphi}) \in \bR^{N\times N} $ and $\Theta^{-1}=U^\star {U^\star}^T$. By the arguments above, we know that the assumptions in Proposition \ref{prop: eigenvalue lower bounds on K phi phi} are satisfied with $h = l_M/2$, $\delta = \min(\delta(X_M; \partial \Omega)/2,1)$ and $Q=M$. Thus, for any vector $w \in \bR^{N}$, by Proposition \ref{prop: eigenvalue lower bounds on K phi phi}, we have \[w^T\Theta^{-1} w \leq C l_M^{-2s+d} |w|^2\, ,\] for some constant $C$ depending on 
$\delta(X_M), d, s, \|\cL\|, J$. This implies that \[|{U^\star}^T w|^2 \leq Cl_M^{-2s+d} |w|^2\, .\]
Taking $w$ to be the standard basis vector $\textbf{e}_{j}$, we get 
\[\sum_{k=j}^N |{U^\star_{jk}}|^2 \leq Cl_M^{-2s+d}\, ,\]
which leads to $|U_{jj}^\star| \leq C l_M^{-s+d/2}$ for some $C$ depending on 
$\delta(X_M; \partial \Omega), d, s, \|\cL\|, J$.
\end{proof}
\blue{
\begin{newremark}
\label{appendix-laplace-measurements-proof}
    We provide some remarks regarding our assumption in Section \ref{sec: theory, set-up}, where we assume all the measurements are of the type $\updelta_{\vx_i} \circ D^{\gamma}$ with the multi-index $\gamma = (\gamma_1,...,\gamma_d) \in \bN^d$ and $|\gamma|:=\sum_{k=1}^d \gamma_k \leq J$; here $D^{\gamma}=D^{\gamma_1}_{\vx^{(1)}}\cdots D^{\gamma_d}_{\vx^{(d)}}$. This assumption is for convenience of presentation and can be generalized. For example, similar proofs can be applied if Laplacian measurements are considered. In fact, for such a case, the main change of the proof of Theorem \ref{thm: cholesky factor decay} is the bound on the right hand side of \eqref{appendix-D-9}, where one now needs to bound $|\Delta \psi_j^i|$. Here $\psi_j^i$ is defined by \eqref{eqn: optimization def basis}, with $\tilde{\phi}_j = \updelta_{\vx_{P(j)}} \circ \Delta $. Note that Theorem \ref{theorem: exp decay of Gamblets} applies to general measurements, and its conditions are satisfied when Laplace measurements are added once linear independence between the measurements is ensured (which is needed for the condition \eqref{eqn: decay cond3}). With Theorem \ref{theorem: exp decay of Gamblets}, we can prove Proposition \ref{prop: exp decay derivative measurements} when Laplace measurements are added, which then implies the bound on the right hand side of \eqref{appendix-D-9}.
\end{newremark}
}
\subsection{Proof of Theorem \ref{thm: KL minimization accuracy}}
\label{appendix: Proof of thm: KL minimization accuracy}
We need the following lemma, which is taken from Lemma B.8 in \cite{schafer2021sparse}.
\begin{lemma}
\label{lemma: comparison between KL and Fro}
Let $\lambda_{\min}, \lambda_{\max}$ be the minimal and maximal eigenvalues of $\Theta \in \bR^{N\times N}$, respectively. Then there exists a universal constant $\eta > 0$ such that for any matrix $M \in \bR^{N \times N}$, we have
\begin{itemize}
    \item If $\lambda_{\max} \|\Theta^{-1} - MM^T\|_{\mathrm{Fro}} \leq \eta$, then \[\operatorname{KL}\left(\cN(0,\Theta)\parallel\cN(0,(MM^T)^{-1})\right) \leq \lambda_{\max} \|\Theta^{-1} - MM^T\|_{\mathrm{Fro}}\, ;\]
    \item If $\operatorname{KL}\left(\cN(0,\Theta)\parallel\cN(0,(MM^T)^{-1})\right) \leq \eta$, then
    \[\|\Theta^{-1} - MM^T\|_{\mathrm{Fro}} \leq \lambda_{\min}^{-1}\operatorname{KL}\left(\cN(0,\Theta)\parallel\cN(0,(MM^T)^{-1})\right) \, .\]
\end{itemize}
\end{lemma}

With this lemma, we can prove Theorem \ref{thm: KL minimization accuracy}. The proof is similar to that of Theorem B.6 in \cite{schafer2021sparse}.
\begin{proof}[Proof of Theorem \ref{thm: KL minimization accuracy}]
    First, by a covering argument, we know $l_M = O(M^{-1/d}) = O(N^{-1/d})$. By Proposition \ref{prop: eigenvalue lower bounds on K phi phi} with $h = l_M/2$, we know that 
    \begin{equation}
    \label{eqn: lambda max min bound}
        \lambda_{\max}(\Theta)\leq C_1N \quad \mathrm{and}\quad  \lambda_{\min}(\Theta) \geq C_1N^{-2s/d+1}
    \end{equation}
    for some constant $C_1$ depending on $\delta(\{\vx_i\}_{i\in I}; \partial \Omega), d, s, \|\cL\|, \|\cL^{-1}\|, J$.

    Theorem \ref{thm: cholesky factor decay} implies that for $(i,j) \notin S_{P,l,\rho}$, it holds that
    \[|U_{ij}^\star| \leq C l_M^{-s+d/2}l_j^{-2s}\exp\left(-\frac{\mathrm{dist}(\vx_{P(i)}, \vx_{P(j)})}{Cl_j}\right) \leq CN^{\alpha} \exp\left(-\frac{\rho}{C}\right)\, , \]
    for some $\alpha$ depending on $s, d$, where $C$ is a generic constant that depends on $\Omega$, $\delta(\{\vx_i\}_{i\in I}; \partial \Omega), d, s, J, \|\cL\|, \|\cL^{-1}\|$; \blue{in fact $\alpha \leq 3s/d-1/2$, using that $l_j^{-1}, l_M^{-1} = O(N^{1/d})$}. Moreover, from the proof in the last subsection, we know that $|U_{ij}^*|\leq C N^{s/d - 1/2}$ for all $1\leq i \leq j \leq N$.

    Now, consider the upper triangular Cholesky factorization $\Theta^{-1}=U^\star {U^\star}^T$. Define $M^{\rho} \in \bR^{N \times N}$ such that
    \begin{equation}
        M^{\rho}_{ij} = \begin{cases}
    U^\star_{ij}, &\text{if } (i,j) \in S_{P,l,\rho}\\
    0, &\text{otherwise}\, .
\end{cases}
    \end{equation}

    Then, by using the above bounds on $U_{ij}^\star$, we know that there exists a constant $\beta$ depending on $s, d$, such that 
    \begin{equation}
    \label{eqn: U truncation err}
        \|\Theta^{-1} - M^{\rho}{M^{\rho}}^T\|_{\mathrm{Fro}} \leq CN^{\beta} \exp\left(-\frac{\rho}{C}\right)\, .
    \end{equation}
    \blue{Here, we can use a simple bound 
    \begin{equation*}
        \begin{aligned}
            &\|\Theta^{-1} - M^{\rho}{M^{\rho}}^T\|_{\mathrm{Fro}} = \|U^\star {U^\star}^T - M^{\rho}{M^{\rho}}^T\|_{\mathrm{Fro}} \\
            \leq &\sqrt{N^2\cdot N^2 \cdot N^{2s/d-1} \cdot C^2 N^{2\alpha}\exp(-\frac{2\rho}{C})}
        \end{aligned}
    \end{equation*}
    since there are $N^2$ entries in the matrix $U^\star {U^\star}^T - M^{\rho}{M^{\rho}}^T$, each entry is the summation of at most $N$ terms, and each term is of the form $U_{ij}^\star U_{lk}^\star$, one of which must be exponentially decaying and bounded by $CN^{\alpha} \exp\left(-\frac{\rho}{C}\right)$ and another is always bounded by $O(l_M^{-s+d/2}) = O(N^{s/d-1/2})$ due to the result at the end of the section \ref{appendix: Proof of thm: cholesky factor decay}. Using the Cauchy-Swarchz inequality, we get that each entry of the matrix is bounded by $\sqrt{N^2 \cdot N^{2s/d-1} \cdot C^2 N^{2\alpha}\exp(-\frac{2\rho}{C})}$; squaring, summing, and then taking square root leads to the final bound. Consequently, we can take $\beta =  s/d+3/2+\alpha \leq 4s/d+1$.}
    
    Since $\lambda_{\max}(\Theta) \leq C_1N$, we know that there exists a constant $C'$, such that when $\rho \geq C'\log(N)$,  
    \[\lambda_{\max}(\Theta)\|\Theta^{-1} - M^{\rho}{M^{\rho}}^T\|_{\mathrm{Fro}} \leq \eta\, ,\]
    for the $\eta$ defined in Lemma \ref{lemma: comparison between KL and Fro}. Using Lemma \ref{lemma: comparison between KL and Fro}, we get
\[\operatorname{KL}\left(\cN(0,\Theta)\parallel\cN(0,(M^{\rho}{M^{\rho}}^T)^{-1})\right) \leq \lambda_{\max}(\Theta) \|\Theta^{-1} - M^{\rho}{M^{\rho}}^T\|_{\mathrm{Fro}}\, .\]
    By the KL optimality, the optimal solution $U^{\rho}$ will satisfy
    \begin{equation}
    \label{eqn: KL upper bound}
        \operatorname{KL}\left(\cN(0,\Theta)\parallel\cN(0,(U^{\rho}{U^{\rho}}^T)^{-1})\right) \leq \lambda_{\max}(\Theta) \|\Theta^{-1} - M^{\rho}{M^{\rho}}^T\|_{\mathrm{Fro}}\, .
    \end{equation}
    \blue{This means that \[ \operatorname{KL}\left(\cN(0,\Theta)\parallel\cN(0,(U^{\rho}{U^{\rho}}^T)^{-1})\right) \leq C_1CN^{\beta+1}\exp(-\frac{\rho}{C})\leq C_1CN^{4s/d+2}\exp(-\frac{\rho}{C})\, .\]}
    Again, by Lemma \ref{lemma: comparison between KL and Fro}, we get
    \begin{equation}
    \label{eqn: theta inv Fro upper bounds}
        \|\Theta^{-1} - U^{\rho}{U^{\rho}}^T\|_{\mathrm{Fro}} \leq \frac{\lambda_{\max}(\Theta)}{\lambda_{\min}(\Theta)} \|\Theta^{-1} - M^{\rho}{M^{\rho}}^T\|_{\mathrm{Fro}} \blue{\leq CN^{6s/d+1}\exp(-\frac{\rho}{C})} \, .
    \end{equation}
Moreover, 
\begin{equation}
\label{eqn: theta fro upper bounds}
    \|\Theta - (U^{\rho}{U^{\rho}}^T)^{-1}\|_{\mathrm{Fro}} \leq \|\Theta\|_{\mathrm{Fro}} \|\Theta^{-1} - U^{\rho}{U^{\rho}}^T\|_{\mathrm{Fro}} \|(U^{\rho}{U^{\rho}}^T)^{-1}\|_{\mathrm{Fro}}\, ,
\end{equation}
Using the above estimates
, we know that there exists a constant $C''$ depending on $\Omega, \delta(\{\vx_i\}_{i\in I}; \partial \Omega), d, s, J, \|\cL\|,$ $\|\cL^{-1}\|$, such that when $\rho \geq C'' \log(N/\epsilon)$, it holds that
\begin{equation*}
    \operatorname{KL}\left(\cN(0,\Theta)\parallel\cN(0,(U^{\rho} {U^{\rho}}^T)^{-1})\right) + \|\Theta^{-1} - U^{\rho} {U^{\rho}}^T\|_{\mathrm{Fro}} + \|\Theta - (U^{\rho} {U^{\rho}}^T)^{-1}\|_{\mathrm{Fro}} \leq \epsilon\, .
\end{equation*}
\blue{Such $C''$ depends on $C$, the generic constant in Theorem 4.1, which relies on the analysis in Theorem \ref{theorem: exp decay of Gamblets}, for which we do not track a tight constant in the scope of this work. These constants are at most polynomially on $s,d$, which can be suppressed by the exponential decay term.

The above bound shows that with $\rho$ logarithmic on $N/\epsilon$, we can get an approximation error bounded by $\epsilon$. 
}
\end{proof}
\subsection{Connections between Cholesky factors, conditional covariance, conditional expectation, and Gamblets}
\label{appendix: Connections between Cholesky factors, conditional covariance, conditional expectation, and Gamblets}
We first present two lemmas. 

The first lemma is about the inverse of a block matrix \cite{lu2002inverses}.
\begin{lemma}
\label{lemma: Theta inv formula}
    For a positive definite matrix $\Theta \in \bR^{N\times N}$, if we write it in the block form
    \begin{equation}
         \Theta = \begin{pmatrix}
    \Theta_{YY} & \Theta_{YZ} \\
    \Theta_{ZY} & \Theta_{ZZ} 
    \end{pmatrix}\, ,
    \end{equation}
    then, $\Theta^{-1}$ equals
    \begin{equation}
          \begin{pmatrix}
    (\Theta_{YY}-\Theta_{YZ}\Theta_{ZZ}^{-1}\Theta_{ZY})^{-1} & -\Theta_{YY}^{-1}\Theta_{YZ}(\Theta_{ZZ} - \Theta_{ZY}\Theta_{YY}^{-1}\Theta_{YZ})^{-1} \\
    -\Theta_{ZZ}^{-1}\Theta_{ZY}(\Theta_{YY}-\Theta_{YZ}\Theta_{ZZ}^{-1}\Theta_{ZY})^{-1} & (\Theta_{ZZ} - \Theta_{ZY}\Theta_{YY}^{-1}\Theta_{YZ})^{-1}
    \end{pmatrix}\, .
    \end{equation}
\end{lemma}

The second lemma is about the relation between the conditional covariance of a Gaussian random vector and the precision matrix.
\begin{lemma}
\label{lemma: cond cov formula}
    Let $N \geq 3$ and $X = (X_1, ..., X_N)$ be a $\cN(0,\Theta)$ Gaussian vector in $\bR^N$. Let $Y=(X_1,X_2)$ and $Z=(X_3,...,X_N)$. Then,
    \begin{equation}
        \mathrm{Cov}(Y_1,Y_2|Z) = \frac{-A_{12}}{A_{11}A_{22} - A_{12}^2}\, ,
    \end{equation}
    and
    \begin{equation}
        \mathrm{Var}(Y_1|Z) = \frac{A_{22}}{A_{11}A_{22} - A_{12}^2}\, ,
    \end{equation}
    where $A = \Theta^{-1}$. Therefore
    \begin{equation}
        \frac{\mathrm{Cov}(Y_1,Y_2|Z)}{\mathrm{Var}(Y_1|Z)} = \frac{-A_{12}}{A_{22}}\, .
    \end{equation}
\end{lemma}
\begin{proof}
    First, we know that $\mathrm{Cov}(Y|Z) = \Theta_{YY}- \Theta_{YZ}\Theta_{ZZ}^{-1}\Theta_{ZY}$. By Lemma \ref{lemma: Theta inv formula}, we have 
    \begin{equation}
        \Theta_{YY}- \Theta_{YZ}\Theta_{ZZ}^{-1}\Theta_{ZY} = \begin{pmatrix}
    A_{11} & A_{12} \\
    A_{21} & A_{22}
    \end{pmatrix}^{-1}\, .
    \end{equation}
    Inverting this matrix leads to the desired result.
\end{proof}

With these lemmas, we can show the connection between the Cholesky factors and the conditional covariance as follows.
\begin{proposition}
\label{prop: connect Cholesky factor to cond var}
    Consider the positive definite matrix $\Theta \in \bR^{N\times N}$. Suppose the upper triangular Cholesky factorization of its inverse is $U$, such that $\Theta^{-1} = UU^T$. Let $Y$ be a Gaussian vector with law $\cN(0, \Theta)$. Then, we have
    \begin{equation}
    \label{eqn: factor, cond var}
        \frac{U_{ij}}{U_{jj}} = (-1)^{i\neq j} \frac{\mathrm{Cov}[Y_i,Y_j|Y_{1:j-1\backslash \{i\}}]}{\mathrm{Var}[Y_i|Y_{1:j-1\backslash \{i\}}]}, \quad i \leq j\, .
    \end{equation}
\end{proposition}
\begin{proof}
We only need to consider the case $i\neq j$.
    We consider $j = N$ first. Denote $A = \Theta^{-1}$. By the definition of the Cholesky factorization $A=UU^{T}$, we know that $U_{:,N} = A_{:,N}/\sqrt{A_{NN}}$. Thus, 
    \[\frac{U_{iN}}{U_{NN}} = \frac{A_{iN}}{A_{NN}}\, .\]
    For $i \neq N$, by Lemma \ref{lemma: cond cov formula}, we have
    \begin{equation}
        \frac{\mathrm{Cov}[Y_i,Y_N|Y_{1:N-1\backslash \{i\}}]}{\mathrm{Var}[Y_i|Y_{1:N-1\backslash \{i\}}]} = \frac{-A_{iN}}{A_{NN}}\, .
    \end{equation}
    By comparing the above two identities, we obtain the result holds for $j=N$. 
    
    For $j < N$, we can use mathematical induction. Note that $U_{1:N-1,1:N-1}$ is also the upper triangular Cholesky factor of $(\Theta_{1:N-1,1:N-1})^{-1}$, noting the fact (by applying Lemma \ref{lemma: Theta inv formula} to the matrix $A$) that 
    \[(\Theta_{1:N-1,1:N-1})^{-1} = A_{1:N-1,1:N-1} - A_{1:N-1,N}A_{NN}^{-1}A_{N,1:N-1}\]
    is the Schur complement, which is exactly the residue block in the Cholesky factorization. Therefore, applying the result for $j=N$ to the matrix $\Theta_{1:N-1,1:N-1}$, we prove \eqref{eqn: factor, cond var} holds for $j=N-1$ as well. Iterating this process to $j=1$ finishes the proof.
\end{proof}

Furthermore, the conditional covariance is related to the conditional expectation, as follows.
\begin{proposition}
\label{prop: connect cond var to cond expect}
    For a Gaussian vector $Y \sim \cN(0, \Theta)$, we have 
    \begin{equation}
    \frac{\mathrm{Cov}[Y_i,Y_j|Y_{1:j-1\backslash \{i\}}]}{\mathrm{Var}[Y_i|Y_{1:j-1\backslash \{i\}}]} = \bE[Y_j|Y_i = 1, Y_{1:j-1\backslash \{i\}} = 0]\, .
\end{equation}
\end{proposition}
\begin{proof}
    We show that for any zero-mean Gaussian vector $Z$ in dimension $d$, and any vectors $v, w \in \bR^d$ (such that $\mathrm{Var}[Z_w] \neq 0$), it holds that
    \begin{equation}
    \label{eqn: D 19 cov to E}
        \frac{\mathrm{Cov}[Z_v, Z_w]}{\mathrm{Var}[Z_w]} = \bE[Z_v|Z_w = 1]\, ,
    \end{equation}
    where $Z_v = \langle Z, v\rangle$ and $Z_w = \langle Z, w\rangle$. Indeed this can be verified by direct calculations. We have
    \[\mathrm{Cov}(Z_v - \frac{\mathrm{Cov}[Z_v, Z_w]}{\mathrm{Var}[Z_w]}Z_w, Z_w) = 0\, ,  \]
    which implies they are independent since they are joint Gaussians. This yields
    \[\bE[Z_v -\frac{\mathrm{Cov}[Z_v, Z_w]}{\mathrm{Var}[Z_w]}Z_w|Z_w = 1] = 0\, ,\]
    which then implies \eqref{eqn: D 19 cov to E} holds. 

    Now, we set $Z$ to be the Gaussian vectors $Y$ conditioned on $Y_{1:j-1\backslash \{i\}} = 0$. It is still a Gaussian vector (with a degenerate covariance matrix). Applying \eqref{eqn: D 19 cov to E} with $v = \textbf{e}_j, w = \textbf{e}_i$, we get the desired result.
\end{proof}

Finally, the conditional expectation is connected to a variational problem. The following proposition is taken from \cite{owhadi2017multigrid, owhadi2019operator}. One can understand the result as the maximum likelihood estimator for the GP conditioned on the constraint coincides with the conditional expectation since the distribution is Gaussian. Mathematically, it can be proved by writing down the explicit formula for the two problems directly.
\begin{proposition}
\label{prop: connect cond expect to Gamblets}
    For a Gaussian process $\xi \sim \cN(0,\cL^{-1})$ where $\cL: H^s_0(\Omega) \to H^{-s}(\Omega)$ satisfies Assumption \ref{assumption on operator}. Then, for some linearly independent measurements $\phi_1,...,\phi_l \in H^{-s}(\Omega)$, the conditional expectation \[\psi^\star(\vx) := \bE\left[\xi(\vx)|[\xi, \phi_i] = c_i, 1\leq i \leq l\right]\] 
    is the solution to the following variational problem:
     \begin{equation}
     \label{eqn: gamblets variational problem}
    \begin{aligned}
    \psi^\star = \mathrm{argmin}_{\psi \in H_0^s(\Omega)}\quad  &[\psi, \cL\psi] \\
     \mathrm{subject\ to}\quad &[\psi, \phi_i] = c_i\ \  \mathrm{for}\ \  1 \leq i \leq l\, .
 \end{aligned}
    \end{equation}
\end{proposition}
The solution of the above variational problem is termed Gamblets in the literature; see Definition \ref{def: Gamblets}.

\subsection{Results regarding the exponential decay of Gamblets}
\label{sec: Regarding the exponential decay of Gamblets}
We collect and organize some theoretical results of the exponential decay property of Gamblets from \cite{owhadi2019operator}. And we provide some new results concerning the derivative measurements which are not covered in \cite{owhadi2019operator}.

The first assumption is about the domain and the partition of the domain.
\begin{assumption}[Domain and partition: from Construction 4.2 in \cite{owhadi2019operator}]
\label{assumption for Gamblets: Domain and partition}
Suppose $\Omega$ is a bounded domain in $\bR^d$ with a Lipschitz boundary. 
Consider $\delta \in (0,1)$ and $h > 0$. Let $\tau_1,...,\tau_Q$ be a partition of $\Omega \subset \bR^d$ such that the closure of each $\tau_i$ is convex, is uniformly Lipschitz, contains a ball of center $\vx_i$ and radius $\delta h$, and is contained in the ball of center $\vx_i$ and radius $h/\delta$.
\end{assumption}

The second assumption is regarding the measurement functionals related to the partition of the domain.
\begin{assumption}[Measurement functionals: from Construction 4.12 in \cite{owhadi2019operator}]
\label{assumption for Gamblets: Measurement functionals}
    Let Assumption \ref{assumption for Gamblets: Domain and partition} holds. For each $1\leq i \leq Q$, let $\phi_{i,\alpha}, \alpha \in \sfT_i$ (where $\sfT_i$ is an index set) be elements of $H^{-s}(\Omega)$ that the following conditions hold:
    \begin{itemize}
        \item Linear independence: $\phi_{i,\alpha}, \alpha \in \sfT_i$ are linearly independent when acting on the subset $H_0^s(\tau_i) \subset H^s_0(\Omega)$.
        \item Locality: $[\phi_{i,\alpha}, \psi] = 0$ for every $\psi \in C^{\infty}_0(\Omega \backslash \tau_i)$ and $\alpha \in \sfT_i$.
    \end{itemize}
\end{assumption}

With the measurement functionals, we can define Gamblets as follows via a variational problem. Note that according to Proposition \ref{prop: connect cond expect to Gamblets}, Gamblets are also conditional expectations of some GP given the measurement functionals.
\begin{definition}[Gamblets: from Section 4.5.2.1 in \cite{owhadi2019operator}]
\label{def: Gamblets}
Let Assumptions  \ref{assumption for Gamblets: Domain and partition} and \ref{assumption for Gamblets: Measurement functionals} hold, and the operator $\cL: H^s_0(\Omega) \to H^{-s}(\Omega)$ satisfies Assumption \ref{assumption on operator}. The Gamblets $\psi_{i,\alpha}, 1\leq i \leq Q, \alpha \in \sfT_i$ associated with the operator $\cL$ and measurement functionals $\phi_{i,\alpha}, 1\leq i \leq Q, \alpha \in \sfT_i$ are defined as
 \begin{equation}
 \label{eqn: def Gamblets}
    \begin{aligned}
    \psi_{i, \alpha} = &\argmin_{\psi \in H_0^s(\Omega)}\quad  [\psi, \mathcal{L}\psi] \\
     &\mathrm{subject\ to}\quad  [\psi, \phi_{k,\beta}] = \delta_{ik}\delta_{\alpha\beta}\ \  \text{for}\ \  1 \leq k \leq Q, \beta \in \sfT_k \, .
 \end{aligned}
    \end{equation}
\end{definition}

A crucial property of Gamblets is that they exhibit exponential decay; see the following Theorem \ref{theorem: exp decay of Gamblets}. 

\begin{newremark}
    Exponential decay results regarding the solution to optimization problems of the type \eqref{eqn: def Gamblets} are first established in \cite{maalqvist2014localization}, where the measurement functionals are piecewise linear nodal functions in finite element methods and the operator $\cL = -\nabla \cdot (a \nabla \cdot)$. Then, the work \cite{owhadi2017multigrid} extends the result to piecewise constant measurement functionals and uses it to develop a multigrid algorithm for elliptic PDEs with rough coefficients. Later on, the work \cite{hou2017sparse} extends the analysis to a class of strongly elliptic high-order operators with piecewise polynomial-type measurement functionals, and the work \cite{chen2020function,chen2022multiscale} focuses on detailed analysis regarding the subsampled lengthscale for subsampled measurement functionals. All these results rely on similar mass-chasing arguments, which are difficult to extend to general higher-order operators.

    The paper \cite{kornhuber2018analysis} presents a simpler and more algebraic proof of the exponential decay in \cite{maalqvist2014localization} based on the exponential convergence of subspace iteration methods. Then, the work \cite{owhadi2019operator} extends this technique (by presenting necessary and sufficient conditions) to general arbitrary integer order operators and measurement functionals. Specifically, the authors in \cite{schafer2021compression} use the conditions in \cite{owhadi2019operator} to show the desired exponential decay when the operator $\cL$ satisfies Assumption \ref{assumption on operator}, and the measurement functionals are Diracs. \blue{We also note that there are works on different localization schemes that can lead to decay estimates with a better dependence on $h$ and better decay rates \cite{henning2013oversampling,hauck2023super}.}
    
    In Theorem \ref{theorem: exp decay of Gamblets}, we present the sufficient conditions in \cite{owhadi2019operator} that ensure the exponential decay and verify that the derivative-type measurements considered in this paper indeed satisfy these conditions; see Propositions \ref{prop: verify conditions of exp decay for derivative meas} and \ref{prop: exp decay derivative measurements}.
\end{newremark}

\begin{theorem}[Exponential decay of Gamblets]
\label{theorem: exp decay of Gamblets}
    Let Assumptions \ref{assumption for Gamblets: Domain and partition} and \ref{assumption for Gamblets: Measurement functionals} hold. We define the function space \[\Phi^\perp := \{f \in H_0^s(\Omega): [f, \phi_{i,\alpha}] = 0 \text{ for any } \alpha \in \sfT_i, 1\leq i \leq Q\}\, .\]
    Assume, furthermore the following conditions hold:
    \begin{align}
    &|f|_{H^t(\Omega)} \leq C_0 h^{s-t}\|f\|_{H_0^s(\Omega)} \text{ for any } 0\leq t \leq s \text{ and } f \in \Phi^\perp; \label{eqn: decay cond1} \\
    &\sum_{1\leq i \leq Q, \alpha \in \sfT_i} [ f,\phi_{i,\alpha}]^2 \leq C_0 \sum_{t=0}^s h^{2t} |f|_{H^t(\Omega)}^2 \text{ for any } f \in H^s_0(\Omega); \label{eqn: decay cond2}\\
    &|y| \leq C_0 h^{-s} \|\sum_{\alpha \in \sfT_i} y_{\alpha}\phi_{i,\alpha}\|_{H^{-s}(\tau_i)} \text{ for any } 1\leq i \leq Q \text{ and } y \in \bR^{|\sfT_i|}. \label{eqn: decay cond3}
    \end{align}
    Here, $|\cdot|_{H^t(\Omega)}$ is the Sobolev seminorm in $\Omega$ of order $t$. 
    
    Then, for the Gamblets in Definition \ref{def: Gamblets}, we have 
    \[|D^\gamma \psi_{i, \alpha}(\vx)| \leq Ch^{-s}\exp(-\frac{\mathrm{dist}(\vx, \vx_i)}{Ch}), \vx \in \Omega \, ,\]
    for any $1\leq i \leq Q, \alpha \in \sfT_i$ and multi-index $\gamma$ satisfying $|\gamma| < s - d/2$. Here $C$ is a constant depending on $C_0, \Omega, \delta, d, s, \|\cL\|, \|\cL^{-1}\|$.
\end{theorem}
\begin{proof}
    We use results in the book \cite{owhadi2019operator}. Conditions \eqref{eqn: decay cond1}, \eqref{eqn: decay cond2} and \eqref{eqn: decay cond3} are equivalently Condition 4.15 in \cite{owhadi2019operator}. 
    Let
    \begin{equation}
        \Omega_i = \mathrm{int}\left(\bigcup_{j: \mathrm{dist}(\tau_i,\tau_j)\leq \delta h} \tau_j\right) \subset B(\vx_i, 3h/\delta)\, .
    \end{equation}
    Then, by Theorem 4.16 in \cite{owhadi2019operator}, there exists a constant $C$ that depends on $C_0, \delta, d, s$, $\|\cL\|$, $\|\cL^{-1}\|$ such that
    \begin{equation}
        \|\psi_{i,\alpha} - \psi_{i,\alpha}^n\|_{H_0^s(\Omega)} \leq Ch^{-s}\exp(-n/C)\, ,
    \end{equation}
    where $\psi_{i,\alpha}^n$ is the minimizer of \eqref{eqn: def Gamblets} after replacing the condition $\psi \in H_0^s(\Omega)$ by $\psi \in H_0^s(\Omega_i^n)$. Here, $\Omega_i^n$ is the collection of $\Omega_j$ whose graph distance to $\Omega_i$ is not larger than $n$; see the definition of graph distance in Definition 4.13 of \cite{owhadi2019operator}. Intuitively, one can understand $\Omega_i^n$ as the $n$-layer neighborhood of $\Omega_i$. We have $\Omega_i^0 = \Omega_i$.
    
    By the definition of $\Omega_i^n$ and Assumption \ref{assumption for Gamblets: Domain and partition}, we know that 
    \begin{equation}
        B(\vx_i, (n+1)\delta h) \cap \Omega \subset \Omega_i^n \subset B(\vx_i, 9(n+1)h/\delta) \cap \Omega\, .
    \end{equation}
    Now, using the Sobolev embedding theorem and the fact that $\psi_{i,\alpha}^n$ is supported in $\Omega_{i}^n$, we get
    \begin{equation}
    \label{eqn: D 19}
    \begin{aligned}
        \|D^{\gamma}\psi_{i,\alpha}\|_{L^{\infty}(\Omega \backslash B(\vx_i, 9nh/\delta))} &\leq \|D^{\gamma}\psi_{i,\alpha}\|_{L^{\infty}(\Omega \backslash \Omega_i^n)}\\
        &\leq \|D^{\gamma}\psi_{i,\alpha}- D^{\gamma}\psi_{i,\alpha}^n\|_{L^{\infty}(\Omega)}\\
        &\leq C_1 \|\psi_{i,\alpha} - \psi_{i,\alpha}^n\|_{H_0^s(\Omega)} \\
        & \leq C_1 C h^{-s}\exp(-n/C)\, ,
    \end{aligned}
    \end{equation}
    where $C_1$ is a constant depending on $\Omega,d, J$ that satisfies
     \[\sup_{0\leq |\gamma| \leq J} \|D^{\gamma} u\|_{L^{\infty}(\Omega)} \leq C_1 \|u\|_{H^s_0(\Omega)}\, , \]
    for any $ u \in H^s_0(\Omega)$.

The result \eqref{eqn: D 19} implies that for any $n \in \bN$, once $\mathrm{dist}(\vx, \vx_i) \geq 9(n+1)h/\delta$, it holds that
\[|D^\gamma \psi_{i, \alpha}(\vx)| \leq Ch^{-s}\exp(-\frac{n}{C}) \, ,\]
where $C$ is some constant depending on $C_0, \Omega, \delta, d, s, \|\cL\|, \|\cL^{-1}\|$. Here we used the fact that the function $C \to Ch^{-s}\exp(-\frac{n}{C})$ is increasing.

In particular, the above inequality holds when  $9(n+1)h/\delta \leq \mathrm{dist}(\vx, \vx_i) \leq 9(n+2) h/\delta$; the relation yields $n \sim \mathrm{dist}(\vx, \vx_i)/h$. By replacing $n$ in terms of $\mathrm{dist}(\vx, \vx_i)/h$, we obtain that there exists a constant $C$ depending on the same set of variables, such that  
 \[|D^\gamma \psi_{i, \alpha}(\vx)| \leq Ch^{-s}\exp(-\frac{\mathrm{dist}(\vx, \vx_i)}{Ch})\, .\]
The proof is complete.
\end{proof}

In the following, we show that conditions \eqref{eqn: decay cond1}, \eqref{eqn: decay cond2}, and \eqref{eqn: decay cond3} are satisfied by the derivative measurements that we are focusing on in the paper. We need to scale each derivative measurement by a power of $h$; this will only change the resulting Gamblets by a corresponding scaling, which would not influence the exponential decay result; see Proposition \ref{prop: exp decay derivative measurements}.
\begin{proposition}
\label{prop: verify conditions of exp decay for derivative meas}
Let Assumptions \ref{assumption for Gamblets: Domain and partition} and \ref{assumption for Gamblets: Measurement functionals} hold. Consider $J \in \bN$ and $J < s - d/2$. For each $1\leq i \leq Q$, we choose $\{h^{d/2}\updelta_{\vx_i}\} \subset \{\phi_{i,\alpha}, \alpha \in \sfT_i\} \subset \{h^{d/2+|\gamma|}\updelta_{\vx_i} \circ D^{\gamma}, 0 \leq |\gamma|\leq J \}$. Then, Conditions \eqref{eqn: decay cond1}, \eqref{eqn: decay cond2}, \eqref{eqn: decay cond3} hold with the constant $C_0$ depending on $\delta, d, s$ and $J$ only.
\end{proposition}
\begin{proof}
    We will verify the conditions one by one. 

    Condition \eqref{eqn: decay cond1} follows from Section 15.4.5 of the book \cite{owhadi2019operator}, where the case of $\{\phi_{i,\alpha}, \alpha \in \sfT_i\} = \{h^{d/2}\updelta_{\vx_i}\}$ is covered. More precisely, in our case, we have more measurement functionals compared to these Diracs, and thus the space $\Phi^\perp$ is smaller than that considered in Section 15.4.5 of the book \cite{owhadi2019operator}. This implies that their results directly apply and \eqref{eqn: decay cond1} readily holds in our case.

    For Conditions \eqref{eqn: decay cond2} and \eqref{eqn: decay cond3}, we can assume $\{\phi_{i,\alpha}, \alpha \in \sfT_i\} = \{h^{d/2+|\gamma|}\updelta_{\vx_i} \circ D^{\gamma}, 0 \leq |\gamma|\leq J \}$ since this could only make the constant $C_0$ larger. We start with Condition \eqref{eqn: decay cond2}. We note that in our proof, $C$ represents a generic constant and can vary from place to place; we will specify the variables it can depend on. 
    
    Consider the unit ball $B(0,1) \subset \bR^d$. Since $J < s -d/2$, the Sobolev embedding theorem implies that
    \begin{equation}
        \sum_{0\leq |\gamma|\leq J} \|D^{\gamma} f\|_{L^{\infty}(B(0,1))} \leq C\|f\|_{H^s(B(0,1))}\, ,
    \end{equation}
    for any $f \in H^s(B(0,1))$, where $C$ is a constant depending on $d$ and $s$. This implies that
    \begin{equation}
        \sum_{0\leq |\gamma|\leq J} |D^{\gamma}f(0)|^2 \leq C\sum_{t=0}^s|f|_{H^t(B(0,1))}^2\, ,
    \end{equation}
    where $C$ depends on $d, s$ and $J$. Consequently, using the change of variables $\vx = \vx_i + \delta h \vx'$, we get
    \begin{equation}
        \sum_{0\leq |\gamma|\leq J}  h^{2|\gamma|}|D^{\gamma}f(\vx_i)|^2 \leq C\sum_{t=0}^s \delta^{2t-d} h^{2t-d}|f|_{H^t(B(\vx_i, \delta h))}^2\, .
    \end{equation}
    We can absorb $\delta$ into $C$ and obtain
    \begin{equation}
        \sum_{0\leq |\gamma|\leq J} h^{d+ 2|\gamma|}|D^{\gamma}f(\vx_i)|^2 \leq C\sum_{t=0}^s h^{2t}|f|_{H^t(B(\vx_i, \delta h))}^2\, ,
    \end{equation}
    where $C$ depends on $\delta, d, s$ and $J$, for any $f \in H^s(B(\vx_i, \delta h))$. Now, using the fact that $H^s(B(\vx_i, \delta h)) \subset H^s_0(\Omega)$ and  $\{\phi_{i,\alpha}, \alpha \in \sfT_i\} = \{h^{d/2+|\gamma|}\updelta_{\vx_i} \circ D^{\gamma}, 0 \leq |\gamma|\leq J \}$, we arrive at
    \begin{equation}
        \sum_{\alpha \in \sfT_i} [f, \phi_{i,\alpha}]^2 \leq C\sum_{t=0}^s h^{2t}|f|_{H^t(B(\vx_i, \delta h))}^2\, ,
    \end{equation}
    for any $f \in H^s_0(\Omega)$. Summing the above inequalities for all $\vx_i$, we get 
    \begin{equation}
        \sum_{1\leq i \leq Q, \alpha \in \sfT_i} [ f,\phi_{i,\alpha}]^2 \leq C \sum_{t=0}^s h^{2t} |f|_{H^t(\Omega)}^2 \text{ for any } f \in H^s_0(\Omega)\, ,
    \end{equation}
    where $C$ depends on $\delta, d, s$ and $J$. This verifies Condition \eqref{eqn: decay cond2}.

    For Condition \eqref{eqn: decay cond3}, we have
    \begin{equation}
    \begin{aligned}
        \|\sum_{\alpha \in \sfT_i} y_{\alpha}\phi_{i,\alpha}\|_{H^{-s}(\tau_i)} &\geq \|\sum_{\alpha \in \sfT_i} y_{\alpha}\phi_{i,\alpha}\|_{H^{-s}(B(\vx_i, \delta h))}\\
        & =  \|\sum_{0
        \leq |\gamma|\leq J} y_{\gamma}h^{d/2+|\gamma|}\updelta_{\vx_i} \circ D^{\gamma} \|_{H^{-s}(B(\vx_i, \delta h))}\, ,
    \end{aligned}
    \end{equation}
    where we abuse the notations to write $y_{\alpha}$ by $y_{\gamma}$. 

    Now, by definition, we get
    \begin{equation}
    \label{eqn: D 23}
        \begin{aligned}
            &\|\sum_{0
        \leq |\gamma|\leq J} y_{\gamma}h^{d/2+|\gamma|}\updelta_{\vx_i} \circ D^{\gamma} \|_{H^{-s}(B(\vx_i, \delta h))} \\
        = &\sup_{v \in H^s_0(B(\vx_i, \delta h))} \frac{\sum_{0\leq |\gamma| \leq J} y_{\gamma} D^{\gamma}v(\vx_i)h^{d/2+|\gamma|}}{\|v\|_{H^s_0(B(\vx_i, \delta h))}}\\
        = &\sup_{v \in H^s_0(B(0,1))} \frac{\sum_{0\leq |\gamma| \leq J} y_{\gamma} D^{\gamma}v(0) h^{d/2+|\gamma|} (\delta h)^{-|\gamma|}}{\|v\|_{H^s_0(B(0, 1))}(\delta h)^{d/2-s}}\\
        \geq & ~C h^{s}\|\sum_{0
        \leq |\gamma|\leq J} y_{\gamma} \updelta_{0} \circ D^{\gamma} \|_{H^{-s}(B(0,1))}\, ,
        \end{aligned}
    \end{equation}
    where $C$ is a constant that depends on $\delta, s$. In the second identity, we used the change of coordinates $\vx = \vx_i + \delta h \vx'$.

    Now, since $\updelta_{0} \circ D^{\gamma}, 0 \leq |\gamma| \leq J$ are linearly independent, we know that there exists $C'$ depending on $d, J$, such that
    \begin{equation}
    \label{eqn: D 24}
        \|\sum_{0
        \leq |\gamma|\leq J} y_{\gamma} \updelta_{0} \circ D^{\gamma} \|_{H^{-s}(B(0,1))} \geq C'C |y|
    \end{equation}
    holds for any $y$. Let $C_0 = C'C$, then Condition \eqref{eqn: decay cond3} is verified.
\end{proof}

By Proposition \ref{prop: verify conditions of exp decay for derivative meas} and rescaling, we can obtain the following results for the unscaled measurements.
\begin{proposition}
\label{prop: exp decay derivative measurements}
    Let Assumptions \ref{assumption for Gamblets: Domain and partition} and \ref{assumption for Gamblets: Measurement functionals} hold. Consider $J \in \bN$ and $J < s - d/2$. For each $1\leq i \leq Q$, we choose $\{\updelta_{\vx_i}\} \subset \{\phi_{i,\alpha}, \alpha \in \sfT_i\} \subset \{\updelta_{\vx_i} \circ D^{\gamma}, 0 \leq |\gamma|\leq J \}$. Then, for the Gamblets in Definition \ref{def: Gamblets}, we have 
    \[|D^\gamma \psi_{i, \alpha}(\vx)| \leq Ch^{-2s}\exp\left(-\frac{\mathrm{dist}(\vx, \vx_i)}{Ch}\right)\, ,\]
    for any $1\leq i \leq Q, \alpha \in \sfT_i$ and multi-index $\gamma$ satisfying $|\gamma| < s - d/2$. Here $C$ is a constant depending on $\Omega, \delta, d, s, \|\cL\|, \|\cL^{-1}\|, J$.
\end{proposition}
\begin{proof}
    Based on Theorem \ref{theorem: exp decay of Gamblets} and Proposition \ref{prop: verify conditions of exp decay for derivative meas}, we know that
    \[|D^\gamma \psi_{i, \alpha}(\vx)| \leq Ch^{-s}\exp\left(-\frac{\mathrm{dist}(\vx, \vx_i)}{Ch}\right)\]
    holds if $\psi_{i, \alpha}$ is the Gamblet corresponding to the measurement functionals satisfying $\{h^{d/2}\updelta_{\vx_i}\} \subset \{\phi_{i,\alpha}, \alpha \in \sfT_i\} \subset \{h^{d/2+|\gamma|}\delta_{\vx_i} \circ D^{\gamma}, 0 \leq |\gamma|\leq J \}$. Using the fact that $|\gamma|+d/2 \leq s$ and the definition of Gamblets, we know that when the measurement functionals change to $\{\updelta_{\vx_i}\} \subset \{\phi_{i,\alpha}, \alpha \in \sfT_i\} \subset \{\updelta_{\vx_i} \circ D^{\gamma}, 0 \leq |\gamma|\leq J \}$, the corresponding Gamblets will satisfy
    \[|D^\gamma \psi_{i, \alpha}(\vx)| \leq Ch^{-2s}\exp\left(-\frac{\mathrm{dist}(\vx, \vx_i)}{Ch}\right)\, .\]
    The proof is complete.
\end{proof}
\section{Eigenvalue bounds on the kernel matrices}
This section is devoted to the lower and upper bounds of the eigenvalues of the kernel matrix.
\begin{proposition}
\label{prop: eigenvalue lower bounds on K phi phi}
    Let Assumptions \ref{assumption for Gamblets: Domain and partition} and \ref{assumption for Gamblets: Measurement functionals} hold. Consider $J \in \bN$ and $J < s - d/2$. For each $1\leq i \leq Q$, we choose $\{\updelta_{\vx_i}\} \subset \{\phi_{i,\alpha}, \alpha \in \sfT_i\} \subset \{\updelta_{\vx_i} \circ D^{\gamma}, 0 \leq |\gamma|\leq J \}$. Let $\bphi$ be the collection of all $\phi_{i,\alpha}, 1\leq i \leq Q, \alpha \in \sfT_i$. Let the operator $\cL: H^s_0(\Omega) \to H^{-s}(\Omega)$ satisfies Assumption \ref{assumption on operator} and assume the kernel function to be the Green function $K(\vx,\vy) := [\updelta_{\vx}, \mathcal{L}^{-1} \updelta_{\vy}]$. Then, for the kernel matrix $K(\bphi,\bphi)$, we have
    \begin{equation}
        C_{\max} h^{-d} \mathrm{I} \succeq K(\bphi, \bphi) \succeq C_{\min} h^{2s-d} \mathrm{I}\, ,
    \end{equation}
    where $C_{\min}$ is a constant that depends on $ \delta, d, s, \|\cL\|, \|\cL^{-1}\|, J$, and $C_{\max}$ additionally depends on $\Omega$. And $\mathrm{I}$ is the identity matrix.
\end{proposition}
\begin{proof}
    It suffices to consider the case $\{\phi_{i,\alpha}, \alpha \in \sfT_i\} = \{\updelta_{\vx_i} \circ D^{\gamma}, 0 \leq |\gamma|\leq J \}$; the kernel matrix in other cases can be seen as a principal submatrix of the case considered here. The eigenvalues of the principal submatrix can be lower and upper bounded by the smallest and largest eigenvalues of the original matrix, respectively.
    
    Suppose $K(\bphi, \bphi) \in \bR^{N \times N}$. For any $y \in \bR^N$, by definition, we have
    \begin{equation}
    \label{eqn E 2}
    \begin{aligned}
        y^TK(\bphi, \bphi)y = [\sum_{1\leq i \leq Q} \sum_{0\leq |\gamma| \leq J} y_{i,\gamma}\updelta_{\vx_i} \circ D^{\gamma}, \cL^{-1}(\sum_{1\leq i \leq Q} \sum_{0\leq |\gamma| \leq J} y_{i,\gamma}\updelta_{\vx_i} \circ D^{\gamma})]\, .
    \end{aligned}
    \end{equation}
    We first deal with the largest eigenvalue. By \eqref{eqn E 2}, we get 
    \[y^TK(\bphi, \bphi)y \leq \|\cL^{-1}\| \|\sum_{1\leq i \leq Q} \sum_{0\leq |\gamma| \leq J} y_{i,\gamma}\updelta_{\vx_i} \circ D^{\gamma}\|_{H^{-s}(\Omega)}^2\, .\]
    Due to the assumption $J < s- d/2$, there exists a constant $C_1$ depending on $\Omega, d, J$ such that
    \[\sup_{0\leq |\gamma| \leq J} \|D^{\gamma} u\|_{L^{\infty}(\Omega)} \leq C_1 \|u\|_{H^s_0(\Omega)}\, , \]
    for any $ u \in H^s_0(\Omega)$. This implies that for any $0\leq |\gamma| \leq J$ and $\vx_i \in \Omega$, it holds
    \[ \|\updelta_{\vx_i} \circ D^{\gamma}\|_{H^{-s}(\Omega)} = \sup_{u \in H_0^s(\Omega)} \frac{|D^\gamma u(\vx_i)|}{\|u\|_{H^s_0(\Omega)}} \leq C_1 \, .\]
    Therefore, we get
    \begin{equation}
    \label{eqn: E3 v2}
        \begin{aligned}
            &\|\sum_{1\leq i \leq Q} \sum_{0\leq |\gamma| \leq J} y_{i,\gamma}\updelta_{\vx_i} \circ D^{\gamma}\|_{H^{-s}(\Omega)}\\
            \leq & \sum_{1\leq i \leq Q} \sum_{0\leq |\gamma| \leq J} |y_{i,\gamma}| \cdot \| \updelta_{\vx_i} \circ D^{\gamma}\|_{H^{-s}(\Omega)}\\
            \leq & C_1\sum_{1\leq i \leq Q} \sum_{0\leq |\gamma| \leq J} |y_{i,\gamma}|\\
            \leq & C_1C_J\sum_{1\leq i \leq Q} |y_{i}| \leq C_1C_J\sqrt{Q} |y|\, ,
        \end{aligned}
    \end{equation}    
    where in the inequality, we used the triangle inequality. In the third and fourth inequalities, we used the Cauchy-Schwarz inequality, and $C_J$ is a constant that depends on $J$. Here, we abuse the notation and write $y_i$ to be the vector collecting $y_{i,\gamma}, 0 \leq |\gamma| \leq J$.

    Now, by a covering argument, we know that $Q = O(h^{-d})$. Therefore, combining \eqref{eqn: E3 v2} and \eqref{eqn E 2}, we arrive at
\[y^TK(\bphi, \bphi)y \leq \|\cL^{-1}\|C_1^2C_J^2Q |y|^2 \leq C_{\max} h^{-d} |y|^2\, ,\]
where $C_{\max}$ is a constant that depends on $\Omega, \delta, d, s, \|\cL\|, \|\cL^{-1}\|, J$. We have obtained the upper bound for the largest eigenvalue of $K(\bphi, \bphi)$ as desired.

    For the smallest eigenvalue, using \eqref{eqn E 2} again, we have 
    \[y^TK(\bphi, \bphi)y \geq \|\cL\|^{-1} \|\sum_{1\leq i \leq Q} \sum_{0\leq |\gamma| \leq J} y_{i,\gamma}\updelta_{\vx_i} \circ D^{\gamma}\|_{H^{-s}(\Omega)}^2\, .\]
    \blue{Now, we will use the following fact
    \begin{equation*}
        \|w\|_{H^{-s}(\Omega)}^2 = \sup_{v \in H_0^s(\Omega)} 2\int wv - \|v\|_{H_0^s(\Omega)}^2
    \end{equation*}
    for any $w \in H^{-s}(\Omega)$; note that here $\int wv$ should be understood as the dual pairing between $H^{-s}(\Omega)$ and $H^s(\Omega)$. To see why this fact is true, we denote by $\tilde{w}$ the Riesz representer of $w$ in $H_0^s(\Omega)$, so that $\int wv = \langle \tilde{w}, v\rangle_{H_0^s(\Omega)}$ where $\langle \cdot, \cdot\rangle_{H_0^s(\Omega)}$ is the $H_0^s(\Omega)$ inner product and $\|w\|_{H_0^{-s}(\Omega)} = \|\tilde{w}\|_{H_0^s(\Omega)}$. Then, it is straightforward to derive that the supremum is achieved by $v = \tilde{w}$, and the optimum is $\|w\|_{H^{-s}(\Omega)}^2$, the left-hand side.
    }
    
    Using the above fact, we get
    \begin{equation}
    \label{eqn E 3}
        \|\sum_{1\leq i \leq Q} \sum_{0\leq |\gamma| \leq J} y_{i,\gamma}\updelta_{\vx_i} \circ D^{\gamma}\|_{H^{-s}(\Omega)}^2 = \sup_{v \in H^s_0(\Omega)} 2\sum_{1\leq i \leq Q} \sum_{0\leq |\gamma| \leq J} y_{i,\gamma} D^{\gamma}v(\vx_i) - \|v\|_{H^s_0(\Omega)}^2\, .
    \end{equation}
    By restricting $v \in H^s_0(\Omega)$ to $v = \sum_{1\leq i \leq Q} v_i$ with each $v_i \in H_0^s(B(\vx_i,\delta h))$, we obtain
    \begin{equation}
    \label{eqn E 4}
    \begin{aligned}
        &\sup_{v \in H^s_0(\Omega)} 2\sum_{1\leq i \leq Q} \sum_{0\leq |\gamma| \leq J} y_{i,\gamma} D^{\gamma}v(\vx_i) - \|v\|_{H^s_0(\Omega)}^2 \\
        \geq & \sum_{1\leq i \leq Q} \left(\sup_{v_i \in H_0^s(B(\vx_i,\delta h))} 2 \sum_{0\leq |\gamma| \leq J} y_{i,\gamma}D^\gamma v_i(\vx_i) - \|v_i\|_{H_0^s(B(\vx_i,\delta h))}\right)\\
        = & \sum_{1\leq i \leq Q} \|\sum_{0\leq |\gamma| \leq J} y_{i,\gamma}\updelta_{\vx_i} \circ D^{\gamma}\|_{H^{-s}(B(\vx_i,\delta h)}^2\, ,
    \end{aligned}
    \end{equation}
    where in the first inequality, we used the fact that the balls $B(\vx_i,\delta h), 1\leq i \leq Q$ are disjoint.

    By \eqref{eqn: D 23} and \eqref{eqn: D 24}, we know that
    \begin{equation}
        \|\sum_{0
        \leq |\gamma|\leq J} y_{i, \gamma}h^{d/2+|\gamma|}\updelta_{\vx_i} \circ D^{\gamma} \|_{H^{-s}(B(\vx_i, \delta h))} \geq Ch^s |y_i|\, ,
    \end{equation}
    for some constant $C$ depending on $\delta, d, s, J$. Here again, we write $y_i$ to be the vector collecting $y_{i,\gamma}, 0 \leq |\gamma| \leq J$. By change of variables, the above inequality implies that
    \begin{equation}
    \label{eqn E 6}
        \|\sum_{0
        \leq |\gamma|\leq J} y_{i, \gamma}\updelta_{\vx_i} \circ D^{\gamma} \|_{H^{-s}(B(\vx_i, \delta h))} \geq Ch^{s-d/2} |y_i|\, ,
    \end{equation}
    for some constant $C$ depending on $\delta, d, s, J$.
    
    Combining \eqref{eqn E 3}, \eqref{eqn E 4}, \eqref{eqn E 6}, we obtain
    \begin{equation}
        \|\sum_{1\leq i \leq Q} \sum_{0\leq |\gamma| \leq J} y_{i,\gamma}\updelta_{\vx_i} \circ D^{\gamma}\|_{H^{-s}(\Omega)}^2 \geq C h^{2s-d}|y|^2  \, .
    \end{equation}
    With \eqref{eqn E 2}, we obtain
    \begin{equation}
        y^TK(\bphi, \bphi)y \geq C\|\cL\|^{-1}h^{2s-d}|y|^2\, .
    \end{equation}
    The proof is complete.
\end{proof}




\end{document}